\documentclass[11pt,a4paper]{article}
\usepackage[a4paper,margin=2.0cm]{geometry}
\usepackage{amssymb,latexsym}
\usepackage{amsmath}
\usepackage{amsthm}
\usepackage{authblk}
\usepackage{graphicx}
\usepackage{xcolor}
\makeatletter

\newcommand{\Rmnum}[1]
{\expandafter@slowromancap\romannumeral#1@}
\makeatother

\numberwithin{equation}{section}
\newtheorem{theorem}{Theorem}[section]
\newtheorem{proposition}[theorem]{Proposition}
\newtheorem{lemma}[theorem]{Lemma}

\usepackage{multirow}

\newcommand{\R}{{\mathbb{R}}}
\newcommand{\cF}{\mathcal F}
\newcommand{\cR}{\mathcal R}
\newcommand{\cM}{\mathcal M}
\newcommand{\inte}{\int_{\mathbb{R}}}

\theoremstyle{definition}

\allowdisplaybreaks[4]

\newtheorem{remark}{Remark}[section]

\title{}
\author{}

\begin{document}
\title{Cubic Nonlinear Schr\"{o}dinger Systems and Orthonormal Fermionic Ground States in $\R$}
\author{
Bin Chen\thanks{Academy of Mathematics and Systems Science, Chinese Academy of Sciences, Beijing 100190, P. R. China. B. Chen  is partially supported by NSF of China (Grant 12501151). Email:  binchen@amss.ac.cn.},\ \
Yujin Guo\thanks{School of Mathematics and Statistics,  Key Laboratory of Nonlinear Analysis $\&$ Applications (Ministry of Education), Central China Normal University, Wuhan 430079, P. R. China. Y. J.  Guo is partially supported by  NSF of China
	(Grants 12225106 and 12371113) and National Key R $\&$ D Program of China (Grant 2023YFA1010001). Email: yguo@ccnu.edu.cn. }, \
Yong Luo\thanks{School of Mathematics and Statistics, and Hubei Key Laboratory of Mathematical Sciences,  Central China Normal University, Wuhan 430079, P. R. China. Y. Luo is partially supported by NSF of China (Grant  12571119). Email: yluo@ccnu.edu.cn.},\
and Juncheng Wei\thanks{Department of Mathematics, Chinese University of Hong Kong, Shatin, NT, Hong Kong. J. C. Wei is partially supported  by   Hong Kong General Research Funds ``New frontiers in singular limits of nonlinear partial differential equations'' and ``On critical and supercritical Fujita equations''. Email: wei@math.cuhk.edu.hk.}  }

\date{}

\smallbreak \maketitle

\begin{abstract}
We study the one-dimensional  nonlinear Schr\"{o}dinger system
\[
-u_i''-2\Big(\sum_{k=1}^N u_k^2\Big)^{p-1}u_i=\mu_i u_i
\quad \mbox{in } \ \mathbb R,\ \ i=1,2,\cdots,N, \eqno{(S)}
\]
where  $u=(u_1,\cdots,u_N)\in H^1(\mathbb{R})^N$, $\mu_1\leq\mu_2\leq\cdots\leq\mu_N<0$, $1<p<3$ and $N\geq 2$ are arbitrary. For the cubic case where $p=2$ and  $N\ge 2$, we establish the following  structural  results of the system $(S)$: (i) all nontrivial solutions can be classified completely;  (ii) the linearized operator at any nontrivial solution is non-degenerate; (iii) exact $L^2$-mass formulas of all nontrivial solutions  are derived in terms of distinct Lagrange multipliers $\mu_i$.
As the applications of the above structural results, we first show that the system \((S)\) with \(p=2\) does not have any orthonormal solution for any \(N\ge2\). We then investigate
a fermionic variational problem \(J_N(p)\) with \(N\ge2\), whose minimizers  correspond to   ground states of the system \((S)\) under orthonormal constraints.
More precisely, we prove that there exists a sufficiently small \(\delta_N>0\) such that \(J_N(p)\) is attained for any \(p\in(2-\delta_N,2)\), whereas \(J_N(p)\) is not attained for any \(p\in[2,2+\delta_N)\).  These settle some conjectures of  [Gontier--Lewin--Nazar, ARMA, 2021] and [Frank--Gontier--Lewin, CMP, 2021].
\end{abstract}

\tableofcontents

\section{Introduction and main results}
We consider the following $N$-component nonlinear Schr\"{o}dinger system
\begin{equation}\label{Guo:FM}
	-u''_i-2\Big(\sum_{j=1}^N u^2_j\Big)^{p-1}u_i=\mu_iu_i\ \ \ \hbox{in}\ \ \R,\ \ i=1,2,\cdots,N,
\end{equation}
where \(N\ge2\), $1<p<3$, \(\mu_1\le \mu_2\le\cdots\le \mu_N<0\), and
$u=(u_1,\cdots,u_N)\in H^1(\mathbb{R})^N$. The coupled nonlinear Schr\"odinger system (\ref{Guo:FM}) arises in several physical contexts, including multi-component Bose--Einstein condensates, nonlinear optics, and effective models of many-particle quantum  fermionic systems, see e.g. \cite{LS,M,RL,ZS,GLN,Lewin}. In such models, different components interact with each other through a common nonlinear density, which often creates the complicated structures absent in a single scalar equation. In elliptic PDE theory, much attention has been devoted to the existence, uniqueness, and other qualitative properties  of standing waves and normalized solutions for these systems, see \cite{A,BS,BZZ,CS,LW,WZZ,WW} and the references therein.

A more specific motivation of studying \eqref{Guo:FM} arises from nonlinear Schr\"odinger equations for orthonormal functions, which are naturally related   (cf. \cite{ii,Lewin3,GLN,HLT,LT}) to  the effective descriptions of the Lieb--Thirring (L-T) inequalities and fermionic systems. Specially, the finite-rank version of the L-T inequality is concerned with the following estimate
\begin{equation}\label{1:FM}
	\sum_{j=1}^N |\lambda_j(-\Delta-V)|^\gamma
	\leq L_{\gamma,d}^{(N)}
	\int_{\mathbb R^d} V_+^{\gamma+d/2}\,dx ,
\end{equation}
where \(\lambda_j(-\Delta-V)\) denotes the $j$-th  min-max level  of the
Schr\"odinger operator \(-\Delta-V\) in $\R^d$.  The optimizer  of this
finite-rank inequality was studied in detail by Frank--Gontier--Lewin
\cite{Lewin3}. In particular, it was proved in \cite{Lewin3} that for the particular case where \(d=1\) and  the
exponent \(\gamma=\frac{1}{p-1}+\frac{1}{2}>\frac{1}{2}\), the optimizer of (\ref{1:FM}) satisfies  the  nonlinear system \eqref{Guo:FM} after  an appropriate scaling transformation of the eigenfunctions.

A classification on the solutions of \eqref{Guo:FM} gives an
eigenfunction-level description of bound-state profiles associated with
the soliton potential. This solitonic point of view is already visible in the system \eqref{Guo:FM} with $p=2$ for the cases \(N=2,3\). More precisely, when   \(N=2, 3\), all  solutions of \eqref{Guo:FM} with $p=2$   can be
classified completely (cf. \cite{ii,GLW2026}) by applying  Hirota-type bilinear
methods, together with algebraic  constants of motion.  These arguments also lead to some conjectures in \cite{GLW2026} on the structure of
solutions for \eqref{Guo:FM} with $p=2$ in the general case \(N\ge 2\).
These  conjectures are however nontrivial, since the proofs of \cite{ii,GLW2026} for the cases \(N=2,3\) heavily rely on the identities whose explicit
forms become  increasingly complicated as the number $N$ increases.
Thus, even though the formulas suggested by the cases \(N=2,3\) exhibit a clear  pattern, a uniform verification and a complete classification for the general case \(N\geq 2\) require a more systematic approach. {\em The first main purpose of the present paper} is to address rigorously the complete classification, non-degeneracy and \(L^2\)-mass identities of all finite-energy solutions for the cubic nonlinear system \eqref{Guo:FM} with $p=2$.

%

The above structural theory of \eqref{Guo:FM} is also naturally connected with fermionic variational problems under orthonormal constraints. We write $U=(u_1,\cdots,u_N)^\top$ and set
\[
\mathcal M_N:=\bigl\{U\in H^1(\R;\R^N):\,
\langle u_i,u_j\rangle_{L^2}=\delta_{ij},\ 1\le i,j\le N\bigr\}.
\]
As studied in \cite{ii,GLN,Lewin} and the references therein, the system \eqref{Guo:FM} with $N\ge2$ arises from the variational problem
\begin{equation}\label{wei1.1}
J_N(p):=\inf_{U\in\mathcal M_N}E_p(U),\ \  1<p<3,
\end{equation}
where the energy functional is defined by
\begin{equation}\label{wei1.2}
E_p(U):=\inte \sum_{i=1}^N\Big(|\nabla u_{i}|^2\Big)dx -\frac{1}{p}\int_{\R}\Big(\sum_{i=1}^N u_i^2\Big)^pdx.
\end{equation}
See \cite{G} for related physical experiments.
If $U$ is a minimizer of $J_N(p)$, then the variational theory and \cite{GLN} imply that, up to an orthogonal transformation, $U$ is a ground state of the following fermionic NLS system
\begin{equation}\label{wei1:x4}
-u_i''-\Big(\sum_{j=1}^N u_j^2\Big)^{p-1}u_i
=\mu_i u_i\quad\text{in }\ \mathbb R, \ \ i=1,\cdots,N,
\end{equation}
where $\langle u_i,u_j\rangle_{L^2}=\delta_{ij}$, and \(\mu_1\le\cdots\le\mu_N<0\) are the first $N$ eigenvalues (counted with multiplicity) of the operator $-D_{xx}-\Big(\sum_{j=1}^{N}|u_j|^2\Big)^{p-1}$ in $\R$. Throughout the whole paper, the minimizers of $J_N(p)$ are thus referred to as ground states of the  system \eqref{wei1:x4} under orthonormal constraints $\langle u_i,u_j\rangle_{L^2}=\delta_{ij}$.  After a simple scaling, one can note that the system \eqref{wei1:x4} is equivalent to
\eqref{Guo:FM}.
Therefore, the classification and \(L^2\)-mass identities of solutions for the  system \eqref{Guo:FM} with $N\geq2$ can provide the structural information on the profiles  of the minimizers for \(J_N(p)\).


The existing  results of \(J_N(p)\) exhibit a sharp contrast across
the critical exponent \(p=2\). Especially, it was proved in \cite[Theorems~3 and~4]{GLN} that \(J_2(p)\), $i.e.,$ the case $N=2$ of \(J_N(p)\), admits minimizers for all \(1<p<2\). However,  \(J_N(2)\) admits no minimizers for any \(N\ge2\), see \cite[Theorem~7]{ii}. It was further conjectured in \cite[Conjectures 2 and 5]{GLN} that for every \(N\ge2\),  the problem \(J_N(p)\) admits a minimizer for \(1<p<2\), but no minimizer for \(2\le p<3\).
Thus, \(p=2\) is the conjectured transition point  of \(J_N(2)\)  between the $L^2-$subcritical existence regime and the $L^2-$subcritical nonexistence regime. Stimulated by above conjectures, as an application of the above structural theory for \eqref{Guo:FM},
{\em the second main purpose of the present paper is to investigate the existence and nonexistence   of minimizers for \(J_N(p)\) in the regime where \(1<p<3\)} and $N\ge 2$. We prove the existence and nonexistence of minimizers for \(J_N(p)\) in the first nontrivial case where
\(N\geq2\) and \(p\) is sufficiently close to \(2\).

\subsection{Main results}
In this subsection, we introduce the main results of the present paper. For convenience, we rewrite the system  \eqref{Guo:FM} with $p=2$ as the following cubic form
\begin{equation}\label{FM}
	-u''_i-2\Big(\sum_{j=1}^N u^2_j\Big) u_i=\mu_iu_i\ \ \ \hbox{in}\,\ \R,\ \ i=1,2,\cdots,N,
\end{equation}
where \(N\ge2\),   \(\mu_1\le \mu_2\le\cdots\le \mu_N<0\), and
$u=(u_1,\cdots,u_N)\in H^1(\mathbb{R})^N$.
We first study the cubic nonlinear system \eqref{FM} for the case of pairwise distinct spectral parameters, whose full determinant structure turns out to be  visible.  Since the identically zero
components of solutions for \eqref{FM} can  be essentially removed, it is enough to state the following theorem for the
solutions of \eqref{FM} whose components are all nontrivial.

\begin{theorem}\label{thm1.1}
	For $\mu_1<\mu_2<\cdots<\mu_N<0$, suppose $u=(u_1,u_2,\cdots,u_N)$ is a  solution of \eqref{FM} in $(H^1(\R)\setminus\{0\})^N$. Then there exists a vector $a=(a_1,a_2,\cdots,a_N)\in(\R\setminus\{0\})^N$ such that $u$ can be expressed exactly as
	\begin{equation}\label{exp1}
		u_i(x)=\frac{a_ie^{\eta_ix}\det(I+B^{(i)}M(x))}{\det(I+M(x))},\ \ i=1,2,\cdots,N,
	\end{equation}
	where the matrix $M(x)$ satisfies
	\begin{equation}\label{MatrixM}
		M(x):=\bigl(m_{jk}(x)\bigr)_{1\le j,k\le N},
		\ \
		m_{jk}(x):=\frac{a_j a_k}{2\sqrt{\eta_j\eta_k}(\eta_j+\eta_k)}e^{(\eta_j+\eta_k)x},
	\end{equation}
	and for any $i=1,2,\cdots,N$,  $\eta_i=\sqrt{-\mu_i}>0$, and
	\begin{equation}\label{Bi}
		B^{(i)}:=\operatorname{diag}\Big(
		\frac{\eta_i-\eta_1}{\eta_i+\eta_1},
		\frac{\eta_i-\eta_2}{\eta_i+\eta_2},
		\dots,
		\frac{\eta_i-\eta_N}{\eta_i+\eta_N}
		\Big).
	\end{equation}
\end{theorem}

We remark that the parameters \(a_1,\cdots,a_N\) of Theorem~\ref{thm1.1} are not
auxiliary constants introduced by the determinant ansatz. They are precisely
the leading coefficients of the solution $u$ at \(-\infty\). Indeed, since
\(M(x)\to0\) as \(x\to-\infty\), the formula \eqref{exp1} gives that
\[
u_i(x)=a_i e^{\eta_i x}+o(e^{\eta_i x})
\ \ \mbox{as}\ \  x\to-\infty,\ \ i=1,\cdots,N .
\]
Thus, Theorem~\ref{thm1.1} gives a one-to-one parametrization of all
solutions with prescribed nonzero left asymptotic data. Moreover, translations and independent sign changes of Theorem~\ref{thm1.1} are encoded directly in
the parameters. More precisely, if \(U^a=(U_1^a, \cdots, U_N^a)\) denotes a solution of \eqref{FM} corresponding to some vector \(a\), then
\[
U^a(x+t)=U^{a(t)}(x),
\ \
a_i(t)=a_i e^{\eta_i t},\ \ a(t)=\big(a_1(t), \cdots, a_N(t)\big),
\]
and $U^{(\sigma_1a_1,\cdots,\sigma_Na_N)}_i
=
\sigma_i U^a_i$ for \(\sigma_j\in\{\pm1\}\).
Consequently, the determinant formula of Theorem~\ref{thm1.1} describes completely the full solution manifold of \eqref{FM},
including its natural translation and sign symmetries.

In the following we sketch  the proof of Theorem~\ref{thm1.1}.
For any \(N\geq2\), the main point of proving  Theorem~\ref{thm1.1} is to verify the determinant ansatz without expanding its components one by one.  By a determinant-ratio identity, the formula of Theorem \ref{thm1.1} can be rewritten as
\[
U_i^a=\sqrt{2\eta_i}\,q_i,
\ \  q=(q_1,\cdots,q_N)=(I+M)^{-1}v,\ \ i=1,\cdots,N,
\]
where $U^a=(U_1^a,\cdots,U_N^a)$ satisfies \eqref{exp1} and \(v\) is a suitable vector given by \eqref{vvect} below.  The verification is then reduced to the claim
that \(q\) satisfies a closed second-order differential system.  This can
be checked uniformly in \(N\geq2\) by exploiting the special Cauchy-type
structure \eqref{cauchystru} of \(M(x)\) in (\ref{MatrixM}).

To finish the proof of Theorem~\ref{thm1.1}, it then remains to prove that except $U^a$, no other \(H^1\)-solutions of \eqref{FM} exist. For an arbitrary
solution \(u\) of \eqref{FM}, each component's equation is an \(L^1\)-perturbation of
\(y''-\eta_i^2y=0\) near \(-\infty\). The \(H^1\)-condition however excludes the
growing mode $e^{-\eta_ix}$ as $x\to-\infty$, and hence it essentially determines a unique vector
\(a=(a_1,\cdots,a_N)\) from the asymptotic behavior
\[
u_i(x)=a_i e^{\eta_i x}+o(e^{\eta_i x})
\ \ \mbox{as}\ \  x\to-\infty,\ \ i=1,\cdots,N.
\]
Compare any solution \(u\) of \eqref{FM} with the determinant solution \(U^a\) under the
same asymptotic data. An asymptotic uniqueness principle then gives
\(u\equiv U^a\) on a left half-line, and the uniqueness theorem of
ODE systems further yields \(u\equiv U^a\) on \(\mathbb R\). We refer to Sections 2 and 3 for the detailed proof of   Theorem~\ref{thm1.1}.

Based on the classification of Theorem~\ref{thm1.1}, we next describe the linearized
structure of solutions for \eqref{FM}.
For any fixed vector
$a=(a_1,\cdots,a_N)\in(\mathbb R\setminus\{0\})^N$, let
$
U^a:=(U_1^a,\cdots,U_N^a)
$
be a solution given by Theorem~\ref{thm1.1}. The linearized operator at $U^a$  is then defined as
\begin{equation}\label{LUa}
	(\mathcal L_{U^a}\phi)_i
	:=
	-\phi_i''
	-\mu_i\phi_i
	-2\left(\sum_{j=1}^N (U_j^a)^2\right)\phi_i
	-4U_i^a\sum_{j=1}^N U_j^a\phi_j\quad\hbox{in}\ \ \R,
	\ \  i=1,\cdots,N.
\end{equation}
The natural kernel directions of (\ref{LUa}) are obtained by differentiating
the family with respect to the parameters \(a_1,\cdots,a_N\), and the
following nondegeneracy theorem shows that they are the only \(H^1\)-kernel directions.

\begin{theorem}
	\label{prop:nondegeneracy}
	For \(\mu_1<\mu_2<\cdots<\mu_N<0\),  let \(U^a\) be the
	solution given by Theorem~\ref{thm1.1} for some vector
	\(a=(a_1,\cdots,a_N)\in(\mathbb R\setminus\{0\})^N\).
	Then we have $\dim\ker_{H^1(\mathbb R)^N}\mathcal L_{U^a}=N$, and
	\[
	\ker_{H^1(\mathbb R)^N}\mathcal L_{U^a}
	=
	\operatorname{span}
	\left\{
	\partial_{a_1}U^a,\cdots,\partial_{a_N}U^a
	\right\},
	\]
	where \(\mathcal L_{U^a}\) is defined by \eqref{LUa}.
\end{theorem}

We now introduce the main idea of proving Theorem~\ref{prop:nondegeneracy}. Since \(U^a\) is an
\(N\)-parameter family of solutions given by Theorem~\ref{thm1.1},
differentiating the system of \(U^a\)  with respect to \(a_k\) gives
\(\mathcal L_{U^a}(\partial_{a_k}U^a)=0\) in $\R$.
Moreover, the exponential decay of \(U^a\) obtained from the determinant
formula implies that \(\partial_{a_k}U^a\in H^1(\mathbb R)^N\). By the expression of \(U^a\), direct
computations    also
give that
\[
\partial_{a_k}U_i^a(x)
=
\delta_{ik}e^{\eta_i x}+o(e^{\eta_i x})
\quad\text{as}\ \ x\to-\infty.
\]
Therefore, these \(N\) kernel elements are linearly independent.

Conversely, let
\(\phi\in\ker_{H^1(\mathbb R)^N}\mathcal L_{U^a}\). Applying
Lemma~\ref{lem:L1-asymptotic} to the linearized system \eqref{LUa}, we obtain that there exist
constants \(b_i\in\mathbb R\) such that
\[
\phi_i(x)=b_i e^{\eta_i x}+o(e^{\eta_i x})
\quad\text{as}\ \ x\to-\infty,\quad i=1, 2, \cdots, N.
\]
Set
\(
z:=\phi-\sum_{k=1}^N b_k\partial_{a_k}U^a .
\)
Then \(z\in\ker_{H^1(\mathbb R)^N}\mathcal L_{U^a}\), i.e., \(\mathcal L_{U^a}z=0\) in $\R$, and
\[
z_i(x)=o(e^{\eta_i x})
\quad\text{as}\quad x\to-\infty,
\quad i=1,\cdots,N.
\]
We prove in Lemma~\ref{lem:zero-left-principle} that \(z\equiv0\) holds on a sufficiently
far left half-line. The uniqueness theorem of ODE systems then yields
\(z\equiv0\) in \(\mathbb R\). Hence, every kernel element of $\mathcal L_{U^a}$ in $\R$ belongs to the set
\[
\operatorname{span}\{\partial_{a_1}U^a,\cdots,\partial_{a_N}U^a\},
\]
which proves the desired characterization of the kernel. The proof of Theorem~\ref{prop:nondegeneracy} is thus complete.

We next focus on the general case of \eqref{FM}, which admits repeated spectral parameters.
Roughly speaking,  these repeated parameters of \eqref{FM} do not however create new scalar profiles of solutions for \eqref{FM}. More precisely, we shall prove that after grouping
equal spectral values, all components  in the same spectral block are the
multiples of one common profile, which has only the natural rotational freedom
inside that block.
Let
\[
\mu_1<\mu_2<\cdots<\mu_K<0
\]
be the distinct spectral values, which have their multiplicities
\(k_1,\cdots,k_K\), so that $\sum_{\alpha=1}^Kk_\alpha=N$.
Set
\begin{equation}\label{eq:multi-index}
	N_\alpha:=k_1+\cdots+k_\alpha,\ \  N_0:=0,
	\ \
	I_\alpha:=\{N_{\alpha-1}+1,\cdots,N_\alpha\}.
\end{equation}
Denote $\nu_1,\cdots,\nu_N$ the full list of the spectral parameters counted with their
multiplicities, namely
\begin{equation}\label{eq:multi-nu}
	\nu_i=\mu_\alpha,
	\ \
	i\in I_\alpha,\quad \alpha=1,\cdots,K .
\end{equation}
Let
\begin{equation}\label{1.10A}
	V^a=(V_1^a,\cdots,V_K^a),
	\ \
	a=(a_1,\cdots,a_K)\in(\mathbb R\setminus\{0\})^K,
\end{equation}
be the determinant family given by Theorem~\ref{thm1.1}, so that $V_\alpha^a$ satisfies the reduced distinct spectral system
\begin{equation}\label{eq:reduced-system}
	(V^a_\alpha)''
	+2\left(\sum_{\beta=1}^K (V^a_\beta)^2\right)V^a_\alpha
	+\mu_\alpha V^a_\alpha=0\quad\hbox{in}\ \ \R,
	\ \  \alpha=1,\cdots,K,
\end{equation}
For $n^{(\alpha)}=\bigl(
n_{N_{\alpha-1}+1}^{(\alpha)},\cdots,
n_{N_\alpha}^{(\alpha)}
\bigr)
\in \mathbb S^{k_\alpha-1},
$
namely,
\(
\sum_{i=N_{\alpha-1}+1}^{N_\alpha} (n_i^{(\alpha)})^2=1,
\)
define
\begin{equation}\label{eq:multi-lift}
	U^{a,n}=(U_1^{a,n},\cdots,U_N^{a,n}),\,\  \,\  U_i^{a,n}(x)=n_i^{(\alpha)}V_\alpha^a(x),
	\ \
	i\in I_\alpha,\ \ \alpha=1,\cdots,K,
\end{equation}
where $V_\alpha^a(x)$ is defined by \eqref{1.10A}. We denote the extended family by
\begin{equation}\label{eq:multi-family}
	\mathcal M_{\mathrm{multi}}
	:=
	\left\{
	U^{a,n}\,\ \text{is defined by } \eqref{eq:multi-lift}:\,
	a\in(\mathbb R\setminus\{0\})^K,\
	n^{(\alpha)}\in\mathbb S^{k_\alpha-1}
	\right\}.
\end{equation}
We shall derive the following property.

\begin{theorem}
	\label{thm:completeness-multiple-root}
	For $\mu_1<\mu_2<\cdots<\mu_K<0$, let $I_\alpha$ and $\nu_1,\cdots,\nu_N$ be defined by \eqref{eq:multi-index} and \eqref{eq:multi-nu},  respectively, and assume that $
	u=(u_1,\cdots,u_N)\in H^1(\mathbb R)^N$
	is a solution of
	\begin{equation}\label{eq:multi-system}
		u_i''
		+2\left(\sum_{j=1}^N u_j^2\right)u_i
		+\nu_i u_i=0\ \ \hbox{in} \,\ \R,
		\ \  i=1,\cdots,N.
	\end{equation}
	If \(\sum_{i\in I_\alpha}u_i^2\not\equiv0\) holds for every
	\(\alpha=1,\cdots,K\), then \(u\in\mathcal M_{\mathrm{multi}}\), where
	\(\mathcal M_{\mathrm{multi}}\) is given by \eqref{eq:multi-family}.
\end{theorem}

The key point of proving Theorem~\ref{thm:completeness-multiple-root} is to use
Wronskian identities to show that all components of the solution $u$ in each spectral block are
parallel. After this reduction, the $N$-component system of
Theorem~\ref{thm:completeness-multiple-root} becomes essentially a $K$-component system with pairwise distinct spectral
parameters. Theorem~\ref{thm1.1} can be then applied, which finally  yields
\(u\in\mathcal M_{\mathrm{multi}}\).
\begin{remark}\label{rem:conj51}
	One can note that Theorems~\ref{thm1.1} and
	\ref{thm:completeness-multiple-root} give a complete classification
	on \(H^1(\mathbb R)^N\) solutions of \eqref{FM} for any
	\(N\ge2\). In particular, in the case of pairwise distinct spectral
	parameters, Theorem~\ref{thm1.1} proves the explicit formula conjectured in
	\cite[Conjecture~5.1]{GLW2026}.
\end{remark}

We now describe the nondegeneracy of the extended family
$\mathcal M_{\mathrm{multi}}$ for the general case  $\nu_1\le\nu_2\le\cdots \le\nu_N<0$. Let
$U^{a,n}\in\mathcal M_{\mathrm{multi}}$
be given by \eqref{eq:multi-lift}, namely
\begin{equation}\label{eq:multi-U}
	U^{a,n}_i(x)=n_i^{(\alpha)}V_\alpha^a(x),
	\ \,
	i\in I_\alpha,\,\ \alpha=1,\cdots,K.
\end{equation}
Since $n^{(\alpha)}\in \mathbb S^{k_\alpha-1}$, we have
\begin{equation}\label{eq:mul-S}
	S^{a}=\sum_{\ell=1}^N (U^{a,n}_\ell)^2
	=
	\sum_{\beta=1}^K (V_\beta^a)^2 .
\end{equation}
The linearized operator at $U^{a,n}$ is then defined by
\begin{equation}\label{eq:multi-LU}
	(\mathcal L_{U^{a,n}}\phi)_i
	:=
	-\phi_i''
	-\nu_i\phi_i
	-2S^{a}\phi_i
	-4U^{a,n}_i\sum_{\ell=1}^N U^{a,n}_\ell\phi_\ell\ \ \,\mbox{in} \,\ \R,
	\ \  i=1,\cdots,N.
\end{equation}
In view of \eqref{eq:multi-nu}, this can equivalently be written as
\begin{equation}\label{eq:multi-LU-block}
	(\mathcal L_{U^{a,n}}\phi)_i
	=
	-\phi_i''
	-\mu_\alpha\phi_i
	-2S^{a}\phi_i
	-4U^{a,n}_i\sum_{\ell=1}^N U^{a,n}_\ell\phi_\ell\ \ \,\mbox{in} \,\ \R,
	\ \  i\in I_\alpha .
\end{equation}

The linearized operator at the reduced solution $V^a$ in \eqref{1.10A} is  defined by
\begin{equation}\label{eq:multi-LV}
	(\mathcal L_{V^a}p)_\alpha
	:=
	-p_\alpha''
	-\mu_\alpha p_\alpha
	-2S^{a}p_\alpha
	-4V_\alpha^a\sum_{\beta=1}^K V_\beta^a p_\beta\ \ \,\mbox{in} \,\ \R,
	\ \  \alpha=1,\cdots,K.
\end{equation}
Applying Theorem~\ref{prop:nondegeneracy}
to the reduced distinct spectral system \eqref{eq:reduced-system}, we have
\begin{equation}\label{eq:reduced-nondeg}
	\ker_{H^1(\mathbb R)^K}\mathcal L_{V^a}
	=
	\operatorname{span}
	\left\{
	\partial_{a_1}V^a,\cdots,\partial_{a_K}V^a
	\right\}.
\end{equation}
The nondegeneracy  in the repeated spectral  case has two types
of kernel directions.  The first type comes from changing the reduced
parameters \(a_\gamma\).  The second type comes from rotating the unit
vectors \(n^{(\alpha)}\) inside the spectral blocks.
The tangent space $T_{U^{a,n}}\mathcal M_{\mathrm{multi}}$ is defined as the span for the
following directions of two different types. Firstly, set for $\gamma=1,\cdots,K$,
\begin{equation}\label{eq:tangent-a-direction}
	\Phi_i^{(\gamma)}
	=
	n_i^{(\alpha)}\partial_{a_\gamma}V_\alpha^a,
	\ \
	i\in I_\alpha,\,\ \alpha=1,\cdots,K.
\end{equation}
Secondly, for each $\alpha=1,\cdots,K$ and each tangent vector
$\xi$ to the unit sphere $\mathbb S^{k_\alpha-1}$ at $n^{(\alpha)}$, namely
\[
\xi=(\xi_i)_{i\in I_\alpha}\in T_{n^{(\alpha)}}\mathbb S^{k_\alpha-1}
=
\left\{\zeta\in\mathbb R^{k_\alpha}:\ \, \zeta\cdot n^{(\alpha)}=0\right\},
\]
set
\begin{equation}\label{eq:tangent-rotation-direction}
	\Psi_i^{(\alpha,\xi)}
	=
	\begin{cases}
		\xi_i V_\alpha^a, & i\in I_\alpha;\\
		0, & i\notin I_\alpha.
	\end{cases}
\end{equation}
Thus
\begin{equation}\label{eq:tangent-space}
	T_{U^{a,n}}\mathcal M_{\mathrm{multi}}
	=
	\operatorname{span}
	\left\{
	\Phi^{(\gamma)},\ \Psi^{(\alpha,\xi)}:\,\
	\gamma=1,\cdots,K,\
	\alpha=1,\cdots,K,\
	\xi\cdot n^{(\alpha)}=0
	\right\}.
\end{equation}
Together with Theorem~\ref{prop:nondegeneracy}, the following theorem addresses the  nondegeneracy  for the general case  $\nu_1\le\cdots\le\nu_N<0$, where $\nu_i$ is as in (\ref{eq:multi-nu}).

\begin{theorem}
	\label{prop:nondegeneracy-multiple-roots}
	Let $U^{a,n}\in\mathcal M_{\mathrm{multi}}$ and $\mathcal L_{U^{a,n}}$ be given by
	\eqref{eq:multi-lift} and
	\eqref{eq:multi-LU}, respectively.
	Then we have
	$\dim\ker_{H^1(\mathbb R)^N}\mathcal L_{U^{a,n}}=N$ and
	$
	\ker_{H^1(\mathbb R)^N}\mathcal L_{U^{a,n}}
	=
	T_{U^{a,n}}\mathcal M_{\mathrm{multi}},
	$
	where $T_{U^{a,n}}\mathcal M_{\mathrm{multi}}$ is given by \eqref{eq:tangent-space}.
\end{theorem}

The proof of Theorem~\ref{prop:nondegeneracy-multiple-roots} follows mainly from the nondegeneracy of Theorem \ref{prop:nondegeneracy} for the distinct spectral case with the block
reduction.  The kernel of $\mathcal L_{U^{a,n}}$ in $\R$ is generated by the reduced parameter directions
\eqref{eq:tangent-a-direction} and the internal block-rotation directions \eqref{eq:tangent-rotation-direction},
which then gives that \(\ker_{H^1(\mathbb R)^N}\mathcal L_{U^{a,n}}
=T_{U^{a,n}}\mathcal M_{\mathrm{multi}}\).

We next derive the following exact $L^2$-mass identity of   solutions for \eqref{FM}, whose proof is based on an algebraic identity generated by the constants of motion, rather than on the explicit determinant formula of solutions for \eqref{FM}.

\begin{theorem}\label{thm:L2iden}
	For \(\mu_1<\mu_2<\cdots<\mu_K<0\),
	let \(I_\alpha\) and \(\nu_1,\cdots,\nu_N\) be defined by
	\eqref{eq:multi-index} and \eqref{eq:multi-nu}, respectively.  Assume that
	\(u=(u_1,\cdots,u_N)\in H^1(\mathbb R)^N\) is a solution of
	\eqref{eq:multi-system} and satisfies
	\(
	\sum_{i\in I_\alpha}u_i^2\not\equiv 0\)
	for $\alpha=1,\cdots,K$.
	Then
	\begin{equation}\label{L2ident}
		\int_{\mathbb R}\sum_{i\in I_\alpha}u_i^2(x)\,dx
		=
		2\sqrt{|\mu_\alpha|}>0\quad\hbox{for}\,\  \alpha=1,\cdots,K .
	\end{equation}
\end{theorem}

The key point of  proving Theorem~\ref{thm:L2iden} is the pointwise algebraic identity obtained in
Proposition~\ref{thm:algiden}, which is derived from the spectral
polynomial identity of Lemma~\ref{lem:polynomial-identity-function} and
reflects the integrable structure of the system. More precisely, under the
same assumptions of Theorem~\ref{thm:L2iden},
Proposition~\ref{thm:algiden} gives that for every
\(\alpha=1,\cdots,K\),
\[
\sum_{j\in I_\alpha}
\left[
(u_j')^2+
\left(\sum_{\ell=1}^N u_\ell^2+\mu_\alpha\right)u_j^2
\right]
+
\sum_{\beta\neq\alpha}
\sum_{\substack{j\in I_\alpha\\ \ell\in I_\beta}}
\frac{(u_j'u_\ell-u_ju_\ell')^2}{\mu_\beta-\mu_\alpha}
=0\quad\hbox{in}\,\ \R.
\]
This identity leads to the first-order differential relation
\eqref{eq:Q-prime} below, and the \(L^2\)-mass formula \eqref{L2ident} then follows by
integrating \eqref{eq:Q-prime} and using the sharp asymptotics of $u_i$ as
\(x\to\pm\infty\).
%


\begin{remark}\label{rem:normalized-solutions}
Theorems~\ref{thm:completeness-multiple-root} and
\ref{thm:L2iden} yield a simple classification of the
componentwise normalized solutions of \eqref{FM}, which generalizes
the classification result in \cite[Theorem~8]{ii}. More precisely, a solution
\(u=(u_1, \cdots, u_N)\in\mathcal M_{\mathrm{multi}}\) of \eqref{FM} is normalized, in the sense that
\(\int_{\mathbb R}u_i^2(x)\,dx=1\) holds
for \(i=1,\cdots,N,\) if and only if
\begin{equation}\label{eq:normalized-multipliers}
	\mu_\alpha=-\frac{k_\alpha^2}{4}<0,
	\ \ \alpha=1,\cdots,K,
\end{equation}
and there exist
\(a\in(\mathbb R\setminus\{0\})^K\) and
\(\varepsilon_i\in\{-1,1\}\), \(i=1,\cdots,N\), such that
\begin{equation}\label{eq:normalized-classification}
	u_i(x)
	=
	\frac{\varepsilon_i}{\sqrt{k_\alpha}}\,
	V_\alpha^a(x),
	\ \
	i\in I_\alpha,\ \ \alpha=1,\cdots,K,
\end{equation}
where \(k_\alpha=|I_\alpha|\ge 1\) is the multiplicity of
\(\mu_\alpha\) satisfying
\(\sum_{\alpha=1}^K k_\alpha=N\), and
\(V^a=(V_1^a,\cdots,V_K^a)\) is the determinant solution of the
reduced system \eqref{eq:reduced-system}.
In particular, Theorems~\ref{thm:completeness-multiple-root} and
\ref{thm:L2iden} prove that for any $N\ge2$, the system \eqref{FM} does not admit any orthonormal solution, which solves the conjectures of Frank--Gontier--Lewin~\cite[page 1793]{ii} and
Guo--Luo--Wei~\cite[Conjecture~5.2]{GLW2026}.
\end{remark}

We finally focus on an application of  the preceding analytical
results. For each fixed $N\ge2$, the classification
and mass identities at $p=2$ describe the possible limiting profiles of
minimizers for $J_N(p)$ as $p\to2$. Applying preceding theorems and binding inequalities, the following theorem shows that \(p=2\) is the  transition of $J_N(p)$ from the existence of minimizers to the nonexistence of minimizers.

\begin{theorem}\label{thm:main}
For any fixed $N\ge2$, suppose $J_N(p)$ is defined by \eqref{wei1.1}.
Then there exists $\delta_N\in(0,1)$ such that 
\begin{enumerate}
\item  [(i).]  $J_N(p)$ admits at least one minimizer
for $p\in(2-\delta_N,2)$;
\item  [(ii).] $J_N(p)$ admits no minimizers for
	$p\in[2,2+\delta_N)$.
\end{enumerate}
\end{theorem}

\begin{remark}\label{rem:gln-conjectures}
Gontier--Lewin--Nazar \cite[Conjectures 2 and 5]{GLN} conjectured that  for every \(N\ge2\), \(J_N(p)\) admits a minimizer for \(1<p<2\), whereas it admits no minimizer for \(2\le p<3\). Theorem~\ref{thm:main} confirms this conjectured transition locally near \(p=2\).
\end{remark}

We next introduce the main idea of proving Theorem~\ref{thm:main}.
Since the case \(p=2\) is already known from \cite[Theorem~7]{ii},  the main issue of proving Theorem~\ref{thm:main} is to understand the minimizers of \(J_N(p)\) as $p$ approaches $2$ from both sides. For \(p\in I_{1/8}:=[15/8,17/8]\), let \(Q_p\) be the unique normalized positive minimizer of \(J_1(p)\) satisfying
\begin{equation}\label{eq:scalar-profile-definition}
	Q_p(x):=
	\left(
	\frac{2\kappa_p\sqrt p}
	{e^{(p-1)\kappa_px}+e^{-(p-1)\kappa_px}}
	\right)^{1/(p-1)}>0,
\end{equation}
where \(\kappa_p>0\) is determined by \(\|Q_p\|_2=1\), and the corresponding Lagrange multiplier is \(-\kappa_p^2\).
Applying the  classification and exact mass identities of previous theorems at \(p=2\), we can prove in Lemma~\ref{prop:profile}  a precise asymptotic description of minimizers \(U=(u_1,\cdots,u_N)\)  of \(J_N(p)\) as \(p\to2\). After an orthogonal transformation, $u_1,\cdots,u_N$ converge in \(H^1\) to \(N\) different translations of \(Q_p\) as \(p\to2\), whose mutual distances tend to infinity as \(p\to2\).

The problem of proving Theorem~\ref{thm:main} is therefore reduced to resolve the exponentially small interaction among these separated profiles as \(p\to2\). For this purpose, we use the Gram matrix to orthonormalize the separated profiles and obtain the approximate state \(\Psi_{p,z}\) as \(p\to2\).
Applying the coercivity and energy reduction of Lemmas~\ref{prop:slice-coercivity} and~\ref{lem:energy-taylor}, we shall prove in Proposition~\ref{prop:two-sided-energy-lower}  the following lower energy estimate of every nearby minimizer \(U\):
\begin{equation}\label{1:9.30}
\begin{aligned}
E_p(U)-NJ_1(p)
\ge &
\frac14\sum_{i<j}e^{-2\kappa_p d_{ij}}
\left[
1+d_{ij}\left(
1-\frac{1-e^{-(p-2)\kappa_p d_{ij}}}
{(p-2)\kappa_p d_{ij}}
\right)
\right]+\frac{c_*}{2}\|K-H_{\mathrm F}\|_{H^1}^2\\
&
-C_N|2-p|D(z)e^{-2\kappa_pD(z)}
e^{\max\{2-p,0\}\kappa_pD(z)} \ \,\ \mbox{as}\ \  p\to 2,
\end{aligned}
\end{equation}
where \(K\) is the small correction for \(\Psi_{p,z}\)
in the representation of \(U\), and \(H_{\mathrm F}\) is the leading term of \(K\),    \(d_{ij}=|z_i-z_j|\) denotes the distance between any two bubbles \(Q_p(\cdot-z_i)\) and \(Q_p(\cdot-z_j)\), and \(D(z)=\min_{i<j}d_{ij}\).
The first term on the right hand side of (\ref{1:9.30}) sums the pair  interacting  contributions between the translations \(Q_p(\cdot-z_i)\) and \(Q_p(\cdot-z_j)\).
The distinction between both sides of \(p=2\) is entirely contained in this first term as follows: the factor
\begin{equation}\label{1:9.31}
e_t:=1-\frac{1-e^{-t}}{t},
\ \ t=(p-2)\kappa_p d_{ij},
\end{equation}
is positive for \(2<p<3\), and negative for \(1<p<2\). The second term  on the right hand side of (\ref{1:9.30}) remains nonnegative for \(p\in I_\delta:=[2-\delta , 2+\delta ]\), while the last term  on the right hand side of (\ref{1:9.30}) is of lower order as \(p\to2\).

For the case \(p>2\), the term $e_t$ of (\ref{1:9.31}) satisfies
\(e_t:=1-\frac{1-e^{-t}}{t}>0\).
If \(J_N(p)\) admits minimizers  for a sequence \(p\searrow2\), then it follows from
Proposition~\ref{prop:two-sided-energy-lower} and Lemmas~\ref{prop:profile} $\&$~\ref{lem:center}  that \(J_N(p)> NJ_1(p)\). This however contradicts with \eqref{eq:s2-8}, which gives that
\(J_N(p)\le NJ_1(p)\). This thus proves Theorem~\ref{thm:main} (ii) on the nonexistence of minimizers for \(J_N(p)\) as \(p>2\) is sufficiently close to \(2\).

As for the case \(p<2\), the same lower estimate of (\ref{1:9.30}) is used in the opposite direction. Indeed, the proof of Lemma~\ref{prop:left-multiplier-clustering} first shows that any existing minimizer at a fixed rank \(m\) must have the following minimum separation
\begin{equation}
D(z)=\min_{1\leq i<j\leq m}d_{ij}\gtrsim (2-p)^{-1/2}.
\end{equation}
Applying Lemma~\ref{lem:energy-taylor}, the above estimate then
 implies that the small correction
term \(K\)  of (\ref{1:9.30}) satisfies
\[
\|K\|_{H^1}\le Ce^{-\kappa_pD(z)}=o(2-p)\ \,\ \mbox{as}\ \ p\nearrow 2.
\]
Using this estimate in the Euler--Lagrange equation of \(U\), the proof of
Lemma~\ref{prop:left-multiplier-clustering} further shows that
the associated Lagrange multiplier matrix $\Lambda$ satisfies
\begin{equation}
	\|\Lambda+\kappa_p^2I_m\|_{\mathrm F}
	=o(2-p)\ \,\ \mbox{as}\ \ p\nearrow 2,
\end{equation}
which implies that all eigenvalues of the minimizers for $J_N(p)$ are close to the scalar multiplier \(-\kappa_p^2I_m\) in the order \(o(2-p)\) as \(p\nearrow2\). This verifies the spectral condition of Lemma~\ref{prop:gln-input}, which then yields the required strict binding inequalities. A finite induction on the particle number, starting from the scalar case, finally proves  Theorem~\ref{thm:main} (i) on the existence of minimizers for \(J_N(p)\) as \(1<p<2\) is sufficiently close to \(2\).

This paper is organized as follows. In Section~2, we construct a determinant family of  \eqref{FM} and verify the equations in the distinct spectral case. Section~3 is then devoted to the proofs of Theorems~\ref{thm1.1} and \ref{prop:nondegeneracy}, which give the classification and nondegeneracy of solutions for \eqref{FM} in the distinct spectral case.  In Section \ref{sec:multiple-spectral-parameters}, we first prove Theorems~\ref{thm:completeness-multiple-root} and \ref{prop:nondegeneracy-multiple-roots} on the classification and nondegeneracy result for the repeated spectral case.  Theorem \ref{thm:L2iden} is then proved in Subsection \ref{sub4.3} by deriving an algebraic identity.
In Section \ref{sec:profile-reduction}, we first establish an exact representation of existing minimizers for \(J_N(p)\) as \(p \to 2\), and then derive the coercivity estimates needed for the energy reduction.
In Section \ref{sec:energy-conclusion}, we establish an explicit energy expansion and a uniform lower bound  of \(J_N(p)\) as \(p \to 2\). These estimates are then used in Section \ref{sec:left-existence} to prove the existence and nonexistence results of Theorem~\ref{thm:main}. Appendix~\ref{app:technical-estimates} finally contains auxiliary estimates used in Sections~\ref{sec:profile-reduction} and~\ref{sec:energy-conclusion}.

\section{Construction  of determinant solutions}
In this section, we mainly verify that the explicit determinant formula \eqref{exp1} of
Theorem~\ref{thm1.1} satisfies \eqref{FM}. Our argument is divided into two
parts. In Subsection 2.1,  we  rewrite the explicit formula \eqref{exp1} for
\(u=(u_1,u_2,\cdots,u_N)\) in terms of an auxiliary vector \(q\). In
Subsection 2.2, we then use this vector formulation to verify the differential
equation. The key point of our argument is that the auxiliary vector \(q\) satisfies a closed second-order
system, which is equivalent to \eqref{FM} after a simple rescaling.

We now introduce the notation used in the determinant formula. Throughout this section, we assume that
\[
\mu_1<\mu_2<\cdots<\mu_N<0
\quad\hbox{and}\quad
\eta_i=\sqrt{|\mu_i|}>0,\quad i=1,2,\cdots,N.
\]
Fixing
$
a=(a_1,\cdots,a_N)\in(\mathbb R\setminus\{0\})^N,
$
we recall that
\[
m_{jk}(x)
=
\frac{a_j a_k}{2\sqrt{\eta_j\eta_k}(\eta_j+\eta_k)}
e^{(\eta_j+\eta_k)x},
\ \
M(x)=(m_{jk}(x))_{1\le j,k\le N},
\]
and
\[
B^{(i)}
=
\operatorname{diag}\left(
\frac{\eta_i-\eta_1}{\eta_i+\eta_1},
\cdots,
\frac{\eta_i-\eta_N}{\eta_i+\eta_N}
\right).
\]
Define
\begin{equation}\label{f-formula}
f(x):=\det\bigl(I+M(x)\bigr),
\end{equation}
and
\begin{equation}\label{g-formula}
g_i(x):=a_i e^{\eta_i x}\det\bigl(I+B^{(i)}M(x)\bigr),
\ \  i=1,\cdots,N,
\end{equation}
such that
\[
u_i(x)=\frac{g_i}{f}\quad\hbox{for}\,\ i=1,2,\cdots,N.
\]
The main purpose of this section is to derive the following precise verification.

\begin{proposition}\label{prop:determinant-verification}
For
$\mu_1<\mu_2<\cdots<\mu_N<0$ and $\eta_i=\sqrt{-\mu_i}>0$, set \(a=(a_1,\cdots,a_N)\in(\R\setminus\{0\})^N\), and define \(f\) and \(g_i\)
by \eqref{f-formula} and \eqref{g-formula}, respectively. Then
\[
u_i(x):=\frac{g_i(x)}{f(x)},\ \  i=1,\cdots,N,
\]
defines a classical solution $u=(u_1,\cdots,u_N)\in \big(H^1(\R)\setminus\{0\}\big)^N$ of \eqref{FM}.
\end{proposition}

The proof of Proposition~\ref{prop:determinant-verification} is divided into two steps.  We first establish in Subsection 2.1 a determinant
ratio identity which rewrites \(g_i/f\) in terms of \(q\).  We then complete in Subsection 2.2
the proof of Proposition~\ref{prop:determinant-verification} by analyzing the vector \(q\).

\subsection{A determinant ratio formula}
We begin with the introduction of an auxiliary vector.
Set
\begin{equation}\label{vvect}
v(x):=
\left(
v_1(x),\cdots,v_N(x)
\right)^T,
\ \
v_j(x):=\frac{a_j}{\sqrt{2\eta_j}}e^{\eta_jx}.
\end{equation}
It then follows from \eqref{MatrixM}  that
\begin{equation}\label{cauchystru}
M(x)=V(x)C V(x),	
\end{equation}
where
\begin{equation}\label{Vm}
V(x):=\operatorname{diag}(v_1(x),\cdots,v_N(x)),
\ \
C:=\left(\frac{1}{\eta_j+\eta_k}\right)_{1\le j,k\le N}.
\end{equation}
We first observe that \(I+M(x)\) is invertible for every \(x\in\mathbb R\).
Indeed,
note that
\[
\frac1{\eta_j+\eta_k}
=
\int_0^\infty e^{-\eta_jt}e^{-\eta_kt}\,dt,
\]
and thus \(C\) is the Gram matrix of the functions
\(e^{-\eta_jt}\) in \(L^2(0,\infty)\). Since the numbers
\(\eta_j\) are distinct, these functions are linearly independent, and
therefore \(C\) is positive definite. Since \(V(x)\) is invertible, \(M(x)=V(x)CV(x)\) is positive definite by
congruence. Hence, \(I+M(x)\) is positive definite and
invertible for every \(x\in\mathbb R\).

Since $I+M(x)$ is invertible in $\R$,  define
\begin{equation}\label{eq:q-definition}
q(x):=(I+M(x))^{-1}v(x)\quad\hbox{in}\,\ \R.
\end{equation}
We now prove the following lemma.

\begin{lemma}\label{lem:ratio-formula}
Let \(f(x)\), \(g_i(x)\)  and \(q(x)\) be defined by
\eqref{f-formula}, \eqref{g-formula}, and \eqref{eq:q-definition}, respectively.  Set $u_i(x):=\frac{g_i(x)}{f(x)}$, then for every
\(i=1,\cdots,N\),
\[
u_i(x)=
\sqrt{2\eta_i}\,q_i(x),
\]
where \(q_i(x)\) denotes the \(i\)-th component of \(q(x)\).
\end{lemma}

\begin{proof}
For any \(i\in\{1,\cdots,N\}\), we prove the identity by Cramer's rule.
For simplicity, we suppress the variable \(x\) in the following computation.
Set
\[
A:=I+M.
\]
Let \(A_i\) be the matrix obtained from \(A\) by replacing its \(i\)-th
column by \(v\). Since \(q=A^{-1}v\), Cramer's rule gives that
$
q_i=\frac{\det A_i}{\det A}.
$
Thus, it remains to identify the numerator \(\det A_i\). We claim that
\begin{equation}\label{eq:Ai-main-identity}
\det A_i
=
v_i\det\bigl(I+B^{(i)}M\bigr).
\end{equation}
Once this is proved, Lemma~\ref{lem:ratio-formula} then follows immediately from the facts that
\(f=\det A\) and
\[
a_i e^{\eta_i x}=\sqrt{2\eta_i}\,v_i(x).
\]

The rest is to prove the claim \eqref{eq:Ai-main-identity}. Write
\[
\beta_j^{(i)}
:=
\frac{\eta_i-\eta_j}{\eta_i+\eta_j},
\ \
B^{(i)}=\operatorname{diag}\bigl(\beta_1^{(i)},\cdots,\beta_N^{(i)}\bigr).
\]
For any \(N\times N\) matrix \(K\) and any subset
\(I\subset\{1,\cdots,N\}\), let \(K[I]\) denote the principal submatrix of
\(K\) indexed by \(I\), where the elements of \(I\) are always taken in their
natural increasing order. We also use the convention
\(\det K[\emptyset]=1\). We then have the following standard principal minor expansion (see, e.g., \cite[(1.2.13)]{HornJohnson2013}):
\[
\det(I+K)=\sum_{I\subset\{1,\cdots,N\}}\det K[I].
\]
Applying this identity to \(K=B^{(i)}M\), and using the fact that
\(B^{(i)}\) is diagonal, we obtain that
\begin{equation}\label{eq:BiM-principal-expansion}
\det\bigl(I+B^{(i)}M\bigr)
=
\sum_{I\subset\{1,\cdots,N\}\setminus\{i\}}
\left(\prod_{j\in I}\beta_j^{(i)}\right)\det M[I],
\end{equation}
since all terms with \(i\in I\)   vanish by the fact that \(\beta_i^{(i)}=0\).
Because $M_{kj}=\frac{v_kv_j}{\eta_k+\eta_j},$
we have
\[
M[I]
=
\operatorname{diag}(v_j)_{j\in I}
\left(\frac{1}{\eta_k+\eta_j}\right)_{k,j\in I}
\operatorname{diag}(v_j)_{j\in I}.
\]
Hence
\begin{equation}\label{eq:MI-cauchy-form}
\det M[I]
=
\left(\prod_{j\in I}v_j^2\right)
\det\left(\frac{1}{\eta_k+\eta_j}\right)_{k,j\in I}.
\end{equation}
Combining \eqref{eq:BiM-principal-expansion} and
\eqref{eq:MI-cauchy-form}, we obtain that
\begin{equation}\label{eq:BiM-expansion}
\det\bigl(I+B^{(i)}M\bigr)
=
\sum_{I\subset\{1,\cdots,N\}\setminus\{i\}}
\left(\prod_{j\in I}\beta_j^{(i)}\right)
\left(\prod_{j\in I}v_j^2\right)
\det\left(\frac{1}{\eta_k+\eta_j}\right)_{k,j\in I}.
\end{equation}

We next expand \(\det A_i\) by the multilinearity of the determinant with
respect to the columns. For each \(j\neq i\), the \(j\)-th column of \(A_i\)
is
$
e_j+M_{\cdot j},
$
where \(e_j\) is the \(j\)-th column of the identity matrix and
\(M_{\cdot j}\) is the \(j\)-th column of \(M\). The \(i\)-th column of
\(A_i\) is \(v\). Thus, in the multilinear expansion,
we choose either \(e_j\) or \(M_{\cdot j}\) for each \(j\neq i\). Denote by
\(I\subset\{1,\cdots,N\}\setminus\{i\}\) the set of indices for which
\(M_{\cdot j}\) is chosen. Then
\begin{equation}\label{eq:Ai-expansion}
\det A_i
=
\sum_{I\subset\{1,\cdots,N\}\setminus\{i\}} \Delta_I.
\end{equation}
Here \(\Delta_I\) is the determinant of the matrix whose \(j\)-th column is
\(M_{\cdot j}\) for \(j\in I\), is \(e_j\) for
\(j\notin I\cup\{i\}\), and is \(v\) for \(j=i\).

To compute \(\Delta_I\), set
$S:=I\cup\{i\}.$
For \(j\notin S\), the \(j\)-th column is the unit vector \(e_j\). Hence this column has only one nonzero entry, namely the entry \(1\) in the \(j\)-th row. Expanding the determinant successively along these unit columns, we may delete all rows and columns with labels \(j\notin S\). Therefore, \(\Delta_I\) is equal to the determinant of the matrix obtained by keeping only the rows and columns whose labels belong to \(S\). In this reduced matrix, the column with label \(i\) is \((v_k)_{k\in S}\), while the column with label \(j\in I\) is \((M_{kj})_{k\in S}\).

We then move simultaneously  the row and the column with label \(i\)  to the
first position. Since the row permutation and the column permutation have the same
sign,  their signs are cancelled. Thus, the determinant is unchanged. We then obtain that
\begin{equation}\label{eq:DeltaI-block}
\Delta_I
=
\det
\begin{pmatrix}
	v_i & (M_{ij})_{j\in I}\\
	(v_k)_{k\in I} & (M_{kj})_{k,j\in I}
\end{pmatrix}.
\end{equation}
Using \(M_{kj}=v_kv_j/(\eta_k+\eta_j)\), we factor \(v_k\) from each row
with label \(k\in S\), and factor \(v_j\) from each column with label \(j\in I\).
This gives that
\begin{equation}\label{eq:DeltaI-factor}
\Delta_I
=
v_i\left(\prod_{j\in I}v_j^2\right)
\det
\begin{pmatrix}
	1 & \left(\frac{1}{\eta_i+\eta_j}\right)_{j\in I}\\
	1 & \left(\frac{1}{\eta_k+\eta_j}\right)_{k,j\in I}
\end{pmatrix}.
\end{equation}
In the last determinant of \eqref{eq:DeltaI-factor}, subtract the first row from each of the remaining
rows and then expand along the first column. It yields that
\[
\det
\begin{pmatrix}
1 & \left(\frac{1}{\eta_i+\eta_j}\right)_{j\in I}\\
1 & \left(\frac{1}{\eta_k+\eta_j}\right)_{k,j\in I}
\end{pmatrix}
=
\det
\left(
\frac{1}{\eta_k+\eta_j}
-
\frac{1}{\eta_i+\eta_j}
\right)_{k,j\in I}.
\]
Moreover, the matrix of the above right-hand side admits the factorization
\[
\left(
\frac{1}{\eta_k+\eta_j}
-
\frac{1}{\eta_i+\eta_j}
\right)_{k,j\in I}
=
\operatorname{diag}(\eta_i-\eta_k)_{k\in I}
\left(\frac{1}{\eta_k+\eta_j}\right)_{k,j\in I}
\operatorname{diag}\left(\frac{1}{\eta_i+\eta_j}\right)_{j\in I}.
\]
We thus conclude from above that
\begin{equation}\label{eq:DeltaI-final}
\Delta_I
=
v_i
\left(\prod_{j\in I}
\frac{\eta_i-\eta_j}{\eta_i+\eta_j}
\right)
\left(\prod_{j\in I}v_j^2\right)
\det\left(\frac{1}{\eta_k+\eta_j}\right)_{k,j\in I}.
\end{equation}
Substituting \eqref{eq:DeltaI-final} into \eqref{eq:Ai-expansion}, we get that
\begin{equation*}
\det A_i
=
\sum_{I\subset\{1,\cdots,N\}\setminus\{i\}}
v_i\left(\prod_{j\in I}\beta_j^{(i)}\right)
\left(\prod_{j\in I}v_j^2\right)
\det\left(\frac{1}{\eta_k+\eta_j}\right)_{k,j\in I}.
\end{equation*}
Combining this with \eqref{eq:BiM-expansion} gives
the claim \eqref{eq:Ai-main-identity}.
This completes the proof of Lemma~\ref{lem:ratio-formula}.
\end{proof}

\subsection{Proof of Proposition~\ref{prop:determinant-verification} }
In this subsection we shall complete the proof of Proposition~\ref{prop:determinant-verification}.
By Lemma~\ref{lem:ratio-formula}, the determinant ratio  $u_i$ is reduced to
the vector
\[
q=(I+M)^{-1}v,
\ \
u_i=\sqrt{2\eta_i}\,q_i .
\]
It remains to show that \(q\) satisfies a closed second-order system, which is
equivalent to \eqref{FM}.
\begin{proof}[\normalfont\bfseries Proof of Proposition~\ref{prop:determinant-verification}]
By Lemma~\ref{lem:ratio-formula}, the determinant solution $u_i$ satisfies
\[
u_i=\sqrt{2\eta_i}\,q_i,\quad i=1,2,\cdots,N.
\]
Set
\begin{equation}\label{2.15M}
\Lambda=\operatorname{diag}(\eta_1,\cdots,\eta_N).
\end{equation}
By \eqref{vvect} and \eqref{MatrixM}, we have
\[
m_{jk}(x)=\frac{v_j(x)v_k(x)}{\eta_j+\eta_k},
\]
and hence
\[
M'(x)=v(x)v(x)^T,
\ \
v'(x)=\Lambda v(x).
\]
Recall that
\[
A(x)=I+M(x)\quad\hbox{and}\quad q(x)=A(x)^{-1}v(x).
\]
Set
\begin{equation}\label{2.16A}
s:=v^\top q,\ \  r:=q^\top\Lambda q.
\end{equation}

To prove  Proposition~\ref{prop:determinant-verification}, it suffices to show that $q$ solves
\begin{equation}\label{key2.1}
Aq''=A(\Lambda^2q-4rq)\quad\hbox{in}\,\ \R.
\end{equation}
Indeed, since $A=I+M$ is invertible, if \eqref{key2.1} holds true, then we have
\begin{equation}\label{qsecond-final-thm}
q''=\Lambda^2q-4rq\quad\hbox{in}\,\ \R.
\end{equation}
Because \(\mu_i=-\eta_i^2\) and
\[
\sum_{j=1}^N u_j^2
=
2\sum_{j=1}^N\eta_jq_j^2=2r,
\]
the equation \eqref{qsecond-final-thm} can be rewritten as
\[
q_i''-\eta_i^2q_i+4rq_i=0\quad\hbox{in}\,\ \R,
\]
which is exactly \eqref{FM}.
Thus, $u=(u_1,u_2,\cdots,u_N)=(\sqrt{2\eta_1}q_1,\sqrt{2\eta_2}q_2,\cdots,\sqrt{2\eta_N}q_N)$ satisfies \eqref{FM}.
The rest is to prove \eqref{key2.1} as follows.

\vspace {.15cm}
\noindent
\textbf{Step 1. Computation of $Aq''$.}
We claim that
\begin{equation}\label{qsecond2-thm}
Aq''=\Lambda^2v-s\Lambda v-3rv+\frac{s^2}{2}v,
\end{equation}
where $\Lambda$ is as in \eqref{2.15M}, $s$ and $r$ are defined by \eqref{2.16A}.

Indeed, since $Aq=v,$
we differentiate both sides of the above identity   to obtain that
$
Aq'+A'q=v'.
$
Since $A'=vv^\top$ and $v'=\Lambda v$, while
$
A'q=vv^\top q=v(v^\top q)=sv,
$
we obtain that
\begin{equation}\label{qprime-eq-thm}
Aq'=(\Lambda-sI)v.
\end{equation}
Differentiating \eqref{qprime-eq-thm} yields that
\[
Aq''+A'q'
=
(\Lambda-sI)'v+(\Lambda-sI)v'
=
-s'v+(\Lambda-sI)\Lambda v.
\]
Since $A'=vv^\top$, this becomes
\begin{equation}\label{qsecond1-thm}
Aq''+v(v^\top q')
=
\Lambda^2v-s\Lambda v-s'v.
\end{equation}

We next compute $v^\top q'$ and $s'$. Since
$v=Aq=(I+M)q=q+Mq,$
we have
\[
q^\top\Lambda v
=
q^\top\Lambda q+q^\top\Lambda Mq
=
r+q^\top\Lambda Mq.
\]
Using \(m_{jk}=v_jv_k/(\eta_j+\eta_k)\) and the fact that
the scalar
\(q^\top\Lambda Mq\) satisfies
\[
q^\top\Lambda Mq=(q^\top\Lambda Mq)^\top=q^\top M\Lambda q,
\]
we get that
\[
2q^\top\Lambda Mq
=
q^\top(\Lambda M+M\Lambda)q
=
\sum_{j,k=1}^N v_jq_jv_kq_k
=
s^2.
\]
It then follows from above that
\begin{equation}\label{def:qLq}
q^\top\Lambda v
=
r+\frac{s^2}{2}.
\end{equation}
Multiplying \eqref{qprime-eq-thm} by \(q^\top\) from the left and
using \(q^\top A=v^\top\), we obtain that
\[
v^\top q'
=
q^\top(\Lambda-sI)v
=
q^\top\Lambda v-s^2.
\]
By \eqref{def:qLq}, this gives that
\begin{equation}\label{vtqprime-thm}
v^\top q'
=
r-\frac{s^2}{2}.
\end{equation}

Finally, differentiating \(s=v^\top q\), we have
\[
s'
=
(v^\top q)'
=
(v')^\top q+v^\top q'
=v^\top\Lambda q +v^\top q'
=q^\top\Lambda v+v^\top q'.
\]
We then get from \eqref{def:qLq} and \eqref{vtqprime-thm} that
\begin{equation}\label{sprime2-thm}
s'=2r.
\end{equation}
Applying \eqref{sprime2-thm} and \eqref{vtqprime-thm},
it yields from \eqref{qsecond1-thm}  that
\[
Aq''
=
\Lambda^2v-s\Lambda v-s'v-v(v^\top q')
=
\Lambda^2v-s\Lambda v-3rv+\frac{s^2}{2}v,
\]
which thus proves the claim \eqref{qsecond2-thm}.

\vspace {.15cm}
\noindent
\textbf{Step 2: Computation of $A(\Lambda^2q-4rq)$.}
We claim that
\begin{equation}\label{def:cla1}
A(\Lambda^2q-4rq)
=
\Lambda^2v-s\Lambda v-3rv+\frac{s^2}{2}v,
\end{equation}
where $s$ and $r$ are again as in \eqref{2.16A}.
Indeed, since $Aq=v$, we have
\[
A(\Lambda^2q-4rq)=A\Lambda^2q-4rv.
\]
On the other hand, it follows from $A=I+M$ that
\begin{equation}\label{2.24A}
\begin{aligned}
A\Lambda^2q=\Lambda^2q+M\Lambda^2q=\Lambda^2(v-Mq)+M\Lambda^2q=\Lambda^2v-(\Lambda^2M-M\Lambda^2)q.
\end{aligned}
\end{equation}
For the commutator term of the above right hand side, its $i$-th component is
\[
\bigl((\Lambda^2M-M\Lambda^2)q\bigr)_i
=
\sum_{k=1}^N
(\eta_i^2-\eta_k^2)\frac{v_iv_k}{\eta_i+\eta_k}q_k
=
\sum_{k=1}^N
(\eta_i-\eta_k)v_iv_kq_k,
\]
which gives from \eqref{2.16A} that
\[
\bigl((\Lambda^2M-M\Lambda^2)q\bigr)_i
=
\eta_i s\,v_i-v_i\sum_{k=1}^N \eta_k v_k q_k.
\]
Since it follows from \eqref{def:qLq} that
\[
\sum_{k=1}^N \eta_k v_k q_k=q^\top \Lambda v=r+\frac{s^2}{2},
\]
we conclude from above that
\[
(\Lambda^2M-M\Lambda^2)q
=
s\Lambda v-\left(r+\frac{s^2}{2}\right)v.
\]
Substituting this into \eqref{2.24A}, we obtain that
\[
A\Lambda^2q
=
\Lambda^2v-s\Lambda v+\left(r+\frac{s^2}{2}\right)v,
\]
and hence the claim \eqref{def:cla1} holds true.

We finally conclude from \eqref{qsecond2-thm} and \eqref{def:cla1} that \eqref{key2.1} holds true, which therefore completes the proof of Proposition \ref{prop:determinant-verification}.
\end{proof}

\section{Solutions for the distinct spectral case}
In this section, we prove   Theorems~\ref{thm1.1} and \ref{prop:nondegeneracy} on  the classification and nondegeneracy results of solutions for \eqref{FM}
in the case where the spectral parameters $\mu_1<\mu_2<\cdots<\mu_N<0$ are distinct.  Our proofs are given  by  comparing the solutions through their leading asymptotic
data at \( -\infty \).  We first show that every
\(H^1\) solution has a well-defined leading coefficient at \(-\infty\).
These coefficients are then   used as the parameters of the determinant
solutions constructed in Section~2.  A left-uniqueness principle finally
shows that any two solutions with the same leading left asymptotics must
coincide with each other.

\subsection{Left  asymptotic uniqueness}
In this subsection, we mainly prove the following scalar ODE lemma to identify the leading asymptotic
behavior of each component \(u_i\) as \(x\to-\infty\), where
\(u=(u_1,\cdots,u_N)\) is a solution of \eqref{FM}.

\begin{lemma}
\label{lem:L1-asymptotic}
Let \(\eta>0\),
\(q\in C((-\infty,R))\cap L^1((-\infty,R))\), and
\(h\in C((-\infty,R))\), where $R\in\R$ is arbitrary. Suppose
\(
y\in H^1((-\infty,R))\cap C^2((-\infty,R))
\) solves
\begin{equation}\label{eq:L1-linear-ode}
y''-\eta^2y=q(x)y+h(x)
\quad\text{on}\,\ (-\infty,R),
\end{equation}
where  there exists \(\sigma>0\) such that
\begin{equation}\label{3.1A}
h(x)=O(e^{(\eta+\sigma)x})\quad \text{as}\,\ x\to-\infty.
\end{equation}
Then there exists a constant \(a\in\mathbb R\) such that
\begin{equation}\label{3.1B}
y(x)=a e^{\eta x}+o(e^{\eta x}),
\ \
y'(x)=\eta a e^{\eta x}+o(e^{\eta x})\quad \text{as}\,\ x\to-\infty.
\end{equation}
Specially, if \(h(x)\equiv0\) and \(a=0\), then
$y\equiv0\ \text{on}\ (-\infty,\, R).
$
\end{lemma}

\noindent\mbox{\bf Proof.}
The proof is divided into three steps. Roughly speaking, for $\eta>0$, we choose \(R_0<R\) sufficiently negative so that
\begin{equation}\label{eq:q-small-tail}
\frac{1}{2\eta}\|q\|_{L^1((-\infty,R_0))}<1 .
\end{equation}
We then
construct a fundamental set \(\{\varphi,\psi\}\) of solutions for the
homogeneous equation
\begin{equation}\label{eq:homoequ}
w''-\eta^2w=q(x)w\quad\hbox{in}\,\ (-\infty,R_0),
\end{equation}
such that \(\varphi\) has the decaying asymptotic behavior \(e^{\eta x}\), while
\(\psi\) has the growing asymptotic behavior \(e^{-\eta x}\) as
\(x\to-\infty\). Also, we use the variation of constants to write the solution $y$ of \eqref{eq:L1-linear-ode}  as
\[
y=\alpha\varphi+\beta\psi+p,
\]
where \(p\) is a particular solution produced by  \(h(x)\). Finally, the decaying
assumption \eqref{3.1A} on \(h(x)\) implies that \(p\) and \(p'\) are both of lower order than
\(e^{\eta x}\) as $x\to-\infty$. Therefore, the condition \(y\in H^1((-\infty,R))\) excludes the
growing mode \(\psi\) and the desired asymptotic estimates  are obtained.

\vspace {.15cm}
\noindent\textbf{Step 1.} In this step, we construct  two fundamental solutions $\varphi(x)$ and $\psi(x)$ of \eqref{eq:homoequ}.

We first construct a decaying solution of the homogeneous equation
\eqref{eq:homoequ} in $(-\infty,R_0)$. To obtain a solution $\varphi(x)$ with the leading behavior
\(e^{\eta x}\) as \(x\to-\infty\), we write
\(
\varphi(x)=e^{\eta x}u(x).
\)
Then $\varphi$ solves \eqref{eq:homoequ} with the  leading behavior
\(e^{\eta x}\) as \(x\to-\infty\), if
\begin{equation}\label{diffequ}
u''+2\eta u'=q(x)u\ \,\ \hbox{in}\ \ (-\infty,R_0),
\ \
u(x)\to1\,\ \hbox{and}\,\ u'(x)\to0
\quad\text{as}\ \ x\to-\infty .	
\end{equation}
For a bounded continuous solution \(u\), $u$ solves \eqref{diffequ} if
$u$ satisfies the following integral equation
\begin{equation}\label{eq:volterra-u}
u(x) = 1+ \frac{1}{2\eta} \int_{-\infty}^{x} \bigl(1-e^{-2\eta(x-s)}\bigr)q(s)u(s)\,ds, \ \  x<R_0 . \end{equation}
Direct differentiation of \eqref{eq:volterra-u} shows that \(u\) satisfies \eqref{diffequ},
$i.e.,$ \(u(x)\to1\) and \(u'(x)\to0\) as \(x\to-\infty\).
Further, differentiating \eqref{eq:volterra-u} once more yields that
\[ u''=q(x)u-2\eta u'\ \,\ \hbox{in}\ \ (-\infty,R_0), \]
and hence \(u''+2\eta u'=q(x)u\) in $(-\infty,R_0)$.

We next prove the existence of a solution to
\eqref{eq:volterra-u}. Set \(X=C_b((-\infty,R_0])\), equipped with the supremum norm. The
right-hand side of \eqref{eq:volterra-u} defines a map \(T:X\to X\).
Moreover, for any \(u_1,u_2\in X\),
\[
\|Tu_1-Tu_2\|_{L^\infty((-\infty,R_0))}
\le
\frac{1}{2\eta}\|q\|_{L^1((-\infty,R_0))}
\|u_1-u_2\|_{L^\infty((-\infty,R_0))} .
\]
By \eqref{eq:q-small-tail}, \(T\) is a contraction, which implies that
\eqref{eq:volterra-u} has a unique bounded continuous solution
\(u\in C_b((-\infty,R_0])\).
By the argument around \eqref{eq:volterra-u}, this solution \(u\) satisfies \eqref{diffequ}. Applying the standard one-dimensional regularity of ODE, we have
\(u\in C^2((-\infty,R_0))\).
Thus,
\(\varphi=e^{\eta x}u\) is a \(C^2\) solution of the homogeneous equation
\eqref{eq:homoequ} and satisfies
\begin{equation}\label{eq:phi-asymptotic}
\varphi(x)=e^{\eta x}[1+o(1)],
\ \
\varphi'(x)=\eta e^{\eta x}[1+o(1)]\quad \text{as } \ x\to-\infty.
\end{equation}
In view of \eqref{eq:phi-asymptotic}, by further decreasing \(R_0\)  if necessary,
we may assume that
$
\varphi(x)\neq0
\ \text{on}\ (-\infty,R_0).
$

We next construct the growing homogeneous solution of \eqref{eq:homoequ} by the reduction of order:
\begin{equation}\label{eq:psi-definition}
\psi(x)
:=
2\eta\,\varphi(x)\int_x^{R_0}\frac{dt}{\varphi(t)^2},
\ \  x<R_0 .
\end{equation}
Then \(\psi\) also solves the homogeneous equation \eqref{eq:homoequ}. Moreover, it yields that the Wronskian of $\varphi$ and $\psi$ satisfies
\[
W(\varphi,\psi)
:=
\varphi\psi'-\varphi'\psi
=
-2\eta.
\]
Thus, \(\varphi\) and \(\psi\) form a fundamental set of solutions for \eqref{eq:homoequ} on
\((-\infty,R_0)\).
It then follows from \eqref{eq:psi-definition} that
\[
\psi(x)=e^{-\eta x}\big[1+o(1)\big]\quad \text{as } \ x\to-\infty.
\]
Differentiating \eqref{eq:psi-definition}, we get that
\[
\psi'
=
2\eta\varphi'\int_x^{R_0}\frac{dt}{\varphi(t)^2}
-
\frac{2\eta}{\varphi},\,\ x<R_0.
\]
Using \eqref{eq:phi-asymptotic} again, we therefore obtain from above that
\begin{equation}\label{eq:psi-asymptotic}
\psi(x)=e^{-\eta x}\big[1+o(1)\big],
\ \
\psi'(x)=-\eta e^{-\eta x}\big[1+o(1)\big]\quad \text{as } \ x\to-\infty.
\end{equation}

\vspace {.05cm}
\noindent\textbf{Step 2.} In this step, we derive a variation-of-constants representation of \eqref{eq:L1-linear-ode}.

We now treat the inhomogeneous problem \eqref{eq:L1-linear-ode}. Define
\[
I_1(x):=\int_{-\infty}^{x}\psi(s)h(s)\,ds,
\ \
I_2(x):=\int_{-\infty}^{x}\varphi(s)h(s)\,ds,
\]
where $\varphi(x)$ and $\psi(x)$ are as in \eqref{eq:phi-asymptotic} and
\eqref{eq:psi-asymptotic}, respectively.
These integrals are well defined. Indeed, by \eqref{eq:phi-asymptotic},
\eqref{eq:psi-asymptotic}, and the assumption on \(h(x)\), we have
\[
\psi(s)h(s)=O(e^{\sigma s}),
\ \
\varphi(s)h(s)=O(e^{(2\eta+\sigma)s})\quad \text{as } \   s\to-\infty,
\]
where $\sigma>0$ is as in \eqref{3.1A}.
Set
\begin{equation}\label{eq:p-definition}
p(x)
:=
\frac{1}{2\eta}\varphi(x)I_1(x)
-
\frac{1}{2\eta}\psi(x)I_2(x).
\end{equation}
Since
\[
I_1'=\psi h,
\ \
I_2'=\varphi h,
\ \
W(\varphi,\psi)=-2\eta,
\]
direct computations give that
\[
p''-(\eta^2+q(x))p=h(x).
\]
Thus, \(p\) is a particular solution of \eqref{eq:L1-linear-ode}. Hence,
\(y-p\) solves the homogeneous equation \eqref{eq:homoequ}, and there exist constants
\(\alpha,\beta\in\mathbb R\) such that the general solution of \eqref{eq:L1-linear-ode}   satisfies
\begin{equation}\label{eq:y-decomposition}
y(x)=\alpha\varphi(x)+\beta\psi(x)+p(x),
\ \  x<R_0,
\end{equation}
where $p(x)$ is given by \eqref{eq:p-definition}.

\vspace {.15cm}
\noindent\textbf{Step 3.} This step is to analyze the asymptotic estimates of \(H^1((-\infty,R))\) solutions for \eqref{eq:L1-linear-ode}.

The above estimates of Step 2 yield that for $\sigma>0$ given by \eqref{3.1A},
\[
I_1(x)=O(e^{\sigma x}),
\ \
I_2(x)=O(e^{(2\eta+\sigma)x})\quad \text{as } \ x\to-\infty,
\]
which imply from Step 1 that
\[
\varphi(x)I_1(x)=O(e^{(\eta+\sigma)x}),
\ \
\psi(x)I_2(x)=O(e^{(\eta+\sigma)x})\quad \text{as } \ x\to-\infty.
\]
Consequently, the particular solution $p(x)$ of \eqref{eq:L1-linear-ode} satisfies
\begin{equation}\label{eq:p-estimate}
p(x)=O(e^{(\eta+\sigma)x})=o(e^{\eta x})\quad \text{as }\ x\to-\infty.
\end{equation}
Differentiating \eqref{eq:p-definition}, whose two terms involving \(h(x)\)
are cancelled,  we have
\[
p'(x)
=
\frac{1}{2\eta}\varphi'(x)I_1(x)
-
\frac{1}{2\eta}\psi'(x)I_2(x).
\]
By the estimates
of \(I_1\) and \(I_2\), we then get from   \eqref{eq:phi-asymptotic} and \eqref{eq:psi-asymptotic} that
\begin{equation}\label{eq:p-prime-estimate}
p'(x)=O(e^{(\eta+\sigma)x})=o(e^{\eta x})\quad \text{as } \ x\to-\infty.
\end{equation}

Applying \eqref{eq:phi-asymptotic}, \eqref{eq:psi-asymptotic}, and
\eqref{eq:p-estimate}, we get from \eqref{eq:y-decomposition} that the general solution $y$ of \eqref{eq:L1-linear-ode} satisfies
\[
y(x)
=
\alpha e^{\eta x}\big[1+o(1)\big]
+
\beta e^{-\eta x}\big[1+o(1)\big]
+
o(e^{\eta x})\quad \text{as }\ x\to-\infty.
\]
Since \(y\in H^1((-\infty,R))\), we have \(y\in L^2((-\infty,R))\).
The function \(e^{-\eta x}\) is not in \(L^2((-\infty,R_0))\), whereas
\[
\varphi\in L^2((-\infty,R_0)),
\ \
p=o(e^{\eta x})\quad \text{as }\ x\to-\infty.
\]
Thus \(\beta=0\) and
\[
y(x)=\alpha e^{\eta x}+o(e^{\eta x})\quad \text{as }\ x\to-\infty.
\]
Applying \eqref{eq:y-decomposition} with \(\beta=0\), we obtain from
\eqref{eq:phi-asymptotic} and \eqref{eq:p-prime-estimate} that
\[
y'(x)
=
\alpha\varphi'(x)+p'(x)
=
\eta\alpha e^{\eta x}+o(e^{\eta x})\quad \text{as }\ x\to-\infty.
\]
Taking \(a:=\alpha\), we thus obtain the desired asymptotics \eqref{3.1B}.

It remains  to prove the last assertion in the homogeneous case
\(h\equiv0\). In this case, the solution $y$ of \eqref{eq:L1-linear-ode} satisfies
\begin{equation}\label{3.14K}
y(x)=\alpha\varphi(x)+\beta\psi(x),
\ \  x<R_0,
\end{equation}
where $\varphi(x)$ and $\psi(x)$ are as in Step 1.
As before, it follows from \(y\in L^2((-\infty,R_0))\)   that \(\beta=0\). Additionally, if
\(a=\alpha=0\), then it yields from \eqref{3.14K} that
\[
y\equiv0
\quad\text{on } \ (-\infty,R_0).
\]
Since \(y\) solves the second-order linear ordinary differential equation \eqref{eq:L1-linear-ode} with
continuous coefficients, the uniqueness of Cauchy problem implies that
$
y\equiv0
\ \text{on}\  (-\infty,R).
$
This proves Lemma~\ref{lem:L1-asymptotic}.
\qed

Applying Lemma~\ref{lem:L1-asymptotic} to each component $u_i$ of $u$, we obtain the following left asymptotic behavior of \(H^1(\mathbb R)^N\) solutions for \eqref{FM}.

\begin{lemma}
\label{lem:left-asymptotic}
For \(\mu_i=-\eta_i^2<0\), \(i=1,\cdots,N\), suppose
$u=(u_1,\cdots,u_N)\in H^1(\mathbb R)^N$
solves
\[
u_i''+2\left(\sum_{j=1}^N u_j^2\right)u_i
=
\eta_i^2u_i\ \ \hbox{in } \ \R,
\ \  i=1,\cdots,N.
\]
Then for each \(i=1,\cdots,N\), there exists a constant
\(a_i^-\in\mathbb R\) such that
\[
u_i(x)=a_i^- e^{\eta_i x}+o(e^{\eta_i x}), \ \
u_i'(x)=\eta_i a_i^- e^{\eta_i x}+o(e^{\eta_i x})\ \ \hbox{as}\ \ x\to-\infty.
\]
Moreover, if \(u_i\not\equiv0\) in $\R$, then
$a_i^-\neq0$.

\end{lemma}

\begin{proof}
Since \(u\in H^1(\mathbb R)^N\), we have
$2\sum_{j=1}^N u_j^2(x)\in L^1(\mathbb R)$.
For each fixed \(i\), where $i=1,2,\cdots,N$, the equation of \(u_i\) can be rewritten as
\[
u_i''-\eta_i^2u_i=q_i(x)u_i\ \ \hbox{in}\,\ \R,
\ \
q_i(x):=-2\sum_{j=1}^N u_j^2(x)\in C(\mathbb R)\cap L^1(\mathbb R).
\]
By the standard one-dimensional regularity of ODE systems, we have
\(u\in \big(C^2(\mathbb R)\big)^N\). Applying Lemma~\ref{lem:L1-asymptotic} with \(y=u_i\),
\(\eta=\eta_i\) and \(q=q_i\), we then obtain that there exists  $a_i^-\in\R$ such that
\[
u_i(x)=a_i^- e^{\eta_i x}+o(e^{\eta_i x}),\,\
u_i'(x)=\eta_i a_i^- e^{\eta_i x}+o(e^{\eta_i x})\,\  \hbox{as}\,\ x\to-\infty.
\]
Specially, if \(a_i^-=0\), then the last assertion of
Lemma~\ref{lem:L1-asymptotic} gives that
$
u_i\equiv0\ \text{on}\ (-\infty,R)
$
for any \(R\in\mathbb R\), which further implies that \(u_i\equiv0\) in
\(\mathbb R\). Therefore, if \(u_i\not\equiv0\), it must have
$a_i^-\neq0$, and we are done.
\end{proof}

If a
solution of a diagonal second-order system is smaller than its natural
decaying rates at \(-\infty\), and if the source term is quadratically
small on the left, then the following left-uniqueness principle further implies that the solution must vanish on a sufficiently far
left half-line.

\begin{lemma}
\label{lem:zero-left-principle}
Let \(\eta_i>0\), and
set $\eta_*:=\min_{1\le i\le N}\eta_i$,  \(i=1,\cdots,N\).
Assume that
\[
z=(z_1,\cdots,z_N)\in H^1((-\infty,R_0))^N
\cap C^2((-\infty,R_0))^N
\]
satisfies  for \(i=1,\cdots,N\),
\[
z_i''-\eta_i^2z_i=F_i(x)
\ \ \text{in } \ (-\infty,R_0),
\ \
z_i(x)=o(e^{\eta_i x})\ \ \text{as }\ x\to-\infty,
\]
where \(F_i(x)\in C((-\infty,R_0))\), and there exists a constant \(C>0\) such that for  any \(x<R_0\),
\begin{equation}\label{Fibound}
|F_i(x)|
\le
C e^{(\eta_i+2\eta_*)x}
\max_{1\le j\le N} e^{-\eta_jx}|z_j(x)|,
\ \  i=1,\cdots,N.
\end{equation}
Then there exists $R<R_0$ such that \(z\equiv0\) on  $(-\infty,R)$.
\end{lemma}

\begin{proof}
Define for any \(R<R_0\),
\[
W_R:=
\max_{1\le j\le N}\sup_{x\le R}
e^{-\eta_jx}|z_j(x)|\geq 0.
\]
By the assumption \(z_j(x)=o(e^{\eta_jx})\) as \(x\to-\infty\),
we obtain that \(W_R<\infty\).
Fix \(R<R_0\) and define
\begin{equation}\label{3.15Q}
Y_i(x)
:=
\frac1{\eta_i}
\int_{-\infty}^{x}
\sinh\bigl(\eta_i(x-t)\bigr)F_i(t)\,dt ,
\ \  x\le R,\quad i=1,\cdots,N.
\end{equation}
By the definition of \(W_R\), we obtain from \eqref{Fibound} that
\begin{equation}\label{Fibound2}
|F_i(x)|
\le
CW_R e^{(\eta_i+2\eta_*)x},\,\ x\le R,\quad i=1,\cdots,N.
\end{equation}
Hence the above integral
is well-defined for \(x\le R\). Moreover, differentiating \eqref{3.15Q}
gives that
\[
Y_i''-\eta_i^2Y_i=F_i(x)\quad\text{in } \ (-\infty,R].
\]
Thus,
\begin{equation}\label{3.16Q}
h_i:=z_i-Y_i
\end{equation}
satisfies the homogeneous equation
\[
h_i''-\eta_i^2h_i=0
\quad\text{in }\ (-\infty,R].
\]
Hence,
\[
h_i(x)=\alpha_i e^{\eta_i x}+\beta_i e^{-\eta_i x}.
\]
Since \(z_i,Y_i\in H^1((-\infty,R))\), we have
\(h_i\in H^1((-\infty,R))\). This excludes the mode \(e^{-\eta_i x}\), and hence \(\beta_i=0\).
On the other hand, we obtain from  \eqref{Fibound2} that
\[
\begin{aligned}
|Y_i(x)|
&\le
\frac{CW_R}{2\eta_i}
\int_{-\infty}^{x}
e^{\eta_i(x-t)}e^{(\eta_i+2\eta_*)t}\,dt
\le
C W_R e^{\eta_i x}e^{2\eta_*x}
=o(e^{\eta_i x})\quad \text{as}\,\ x\to-\infty.
\end{aligned}
\]
By the assumption
\(
z_i(x)=o(e^{\eta_i x})\)
as \(x\to-\infty\),
we then have
\(
h_i(x)=o(e^{\eta_i x})
\) as $x\to-\infty$.
Since \(h_i(x)=\alpha_i e^{\eta_i x}\), we get that \(\alpha_i=0\), and it hence follows from \eqref{3.16Q} that $z_i=Y_i$, $i=1,2,\cdots,N$.

We obtain from above that for $i=1,\cdots,N$,
\begin{equation}\label{3.14A}
z_i(x)=Y_i(x)
=
\frac1{\eta_i}
\int_{-\infty}^{x}
\sinh\bigl(\eta_i(x-t)\bigr)F_i(t)\,dt ,
\ \  x\le R.
\end{equation}
Since
\[
\sinh\bigl(\eta_i(x-t)\bigr)
\le
\frac12 e^{\eta_i(x-t)},
\ \  t\le x,
\]
we obtain from \eqref{Fibound2} and \eqref{3.14A} that
\[
\begin{aligned}
|z_i(x)|
&\le
\frac{C}{2\eta_i}W_R
\int_{-\infty}^{x}
e^{\eta_i(x-t)}e^{(\eta_i+2\eta_*)t}\,dt
=
\frac{C}{2\eta_i}W_R e^{\eta_i x}
\int_{-\infty}^{x}e^{2\eta_*t}\,dt  \\
&=
\frac{C}{4\eta_i\eta_*}
W_R e^{\eta_i x}e^{2\eta_*x},\quad i=1,\cdots,N.
\end{aligned}
\]
Since \(\eta_i\ge\eta_*\), there exists a constant \(\Theta:=\frac{C}{4\eta_*^2}>0\),
independent of \(R\), such that
\[
e^{-\eta_i x}|z_i(x)|
\le
\Theta W_R e^{2\eta_*x}
\le
\Theta W_R e^{2\eta_*R}\quad\hbox{in}
\,\ (-\infty,R],\,\ i=1,\cdots,N.
\]
Taking the supremum over \(x\le R\) and then the maximum over \(i\),
we get that
\[
0\leq W_R\le \Theta e^{2\eta_*R}W_R.
\]
Choosing \(R\ll-1\) such that
\(
\Theta e^{2\eta_*R}<1,
\)
we thus conclude from above that \(W_R=0\). Hence, it yields from \eqref{3.14A} that
\[
z_i(x)\equiv0,
\ \  x\le R,\,\ i=1, \cdots, N,
\]
which completes the proof of Lemma~\ref{lem:zero-left-principle}.
\end{proof}

\subsection{Proofs of Theorems~\ref{thm1.1} and \ref{prop:nondegeneracy}}
In this subsection, we mainly prove Theorems~\ref{thm1.1} and
\ref{prop:nondegeneracy}, which are concerned with the classification and nondegeneracy of solutions for \eqref{FM} in the distinct-spectral case.

For any \(H^1\) solution
\[
u=(u_1,\cdots,u_N)\in (H^1(\mathbb R)\setminus\{0\})^N,
\]
roughly speaking, the previous lemmas give the leading coefficient of
each component \(u_i\) at \(-\infty\).  In order to prove Theorem~\ref{thm1.1},  we shall use these coefficients as the
parameters in the determinant solution \(U^a\) constructed in
Section~2.  Since \(u\) and \(U^a\) have the same leading coefficients
componentwise at \(-\infty\), we shall employ the left-uniqueness principle to conclude that
\(u\equiv U^a\) in $\R$.

\begin{proof}[\normalfont\bfseries Proof of  Theorem~\ref{thm1.1}]
Suppose $u=(u_1,\cdots,u_N)$ is a solution of \eqref{FM}, and set
$
\eta_i:=\sqrt{-\mu_i}>0\ \text{for}\  i=1,\cdots,N,
$
so that the system \eqref{FM} can be rewritten as
\[
u_i''
+
2\left(\sum_{j=1}^N u_j^2\right)u_i
=
\eta_i^2u_i\quad\hbox{in}\,\ \R,
\ \  i=1,\cdots,N.
\]
By Lemma~\ref{lem:left-asymptotic}, there exist constants
\(a_i\in\mathbb R\) such that for each \(i=1,\cdots,N\),
\begin{equation}\label{3.17A}
u_i(x)=a_i e^{\eta_i x}+o(e^{\eta_i x})\quad \text{as } \   x\to-\infty.
\end{equation}
Since \(u_i\not\equiv0\), the last assertion of
Lemma~\ref{lem:left-asymptotic} gives that
\[
a_i\neq0,\ \  i=1,\cdots ,N.
\]
Let \(a=(a_1,\cdots,a_N)\) be the vector of leading coefficients obtained
above. We  denote by \(U^a=(U_1^a,\cdots,U_N^a)\)   the determinant solution of Theorem \ref{thm1.1}
with the parameter vector \(a\), namely
\begin{equation}\label{3.17B}
U_i^a(x)
=\frac{a_i e^{\eta_i x}\det(I+B^{(i)}M(x))
}{\det(I+M(x))},
\ \  i=1,\cdots,N.
\end{equation}
By Proposition~\ref{prop:determinant-verification}, \(U^a\) is an
\((H^1(\mathbb R)\setminus\{0\})^N\) solution of the  system \eqref{FM}. Moreover, since
\(M(x)\to0\) as \(x\to-\infty\), we have
\[
U_i^a(x)=a_i e^{\eta_i x}+o(e^{\eta_i x})\quad \text{as }\ x\to-\infty.
\]

It remains to prove that \(u\equiv U^a\) in $\R$. Set
\[
w_i:=u_i-U_i^a,
\ \
S_u:=\sum_{j=1}^N u_j^2,
\ \
S_a:=\sum_{j=1}^N (U_j^a)^2.
\]
Then
\[
w_i''-\eta_i^2w_i
=
-2\left(S_u u_i-S_a U_i^a\right)
=:F_i(x).
\]
Since \(u\) and \(U^a\) have the same left asymptotic data, we have
\[
w_i(x)=o(e^{\eta_i x})\quad \text{as } \ x\to-\infty.
\]
Set
$\eta_*:=\min_{1\le j\le N}\eta_j .
$
By the  asymptotics \eqref{3.17A} and \eqref{3.17B}, there exist \(R_0<0\) and \(C_0>0\)
such that
\[
|u_i(x)|+|U_i^a(x)|
\le
C_0 e^{\eta_i x}\quad\hbox{in}\,\ (-\infty,R_0],
\ \  i=1,\cdots,N.
\]
Define
\[
Z(x):=
\max_{1\le j\le N}e^{-\eta_jx}|w_j(x)|, \ \  x\le R_0,
\]
so that
\[
|w_j(x)|\le Z(x)e^{\eta_jx},
\ \  x\le R_0.
\]
Since
\[
S_u u_i-S_a U_i^a
=
S_u w_i+(S_u-S_a)U_i^a,
\ \
S_u-S_a
=
\sum_{j=1}^N (u_j+U_j^a)w_j,
\]
we obtain that
\[
\begin{aligned}
|F_i(x)|
&\le
2S_u(x)|w_i(x)|
+
2|U_i^a(x)|
\sum_{j=1}^N |u_j(x)+U_j^a(x)|\,|w_j(x)|  \\
&\le
C e^{(\eta_i+2\eta_*)x}Z(x), \ \  x\le R_0,
\end{aligned}
\]
where \(C>0\) is independent of \(x\). Thus, \(w=(w_1,\cdots,w_N)\) satisfies the
assumptions of Lemma~\ref{lem:zero-left-principle}, which then implies that
\[
w\equiv0
\ \ \text{on a sufficiently far left half-line}.
\]
In particular, we have for some \(R<R_0\),
\[
u_i(R)=U_i^a(R),
\ \
u_i'(R)=(U_i^a)'(R),
\ \  i=1,\cdots,N.
\]
By the uniqueness theorem of   ODE systems, we then obtain that
$
u\equiv U^a
\ \text{on}\ \mathbb R.
$
This proves Theorem~\ref{thm1.1}.
\end{proof}

We next prove Theorem~\ref{prop:nondegeneracy} on the
nondegeneracy of the determinant family in the distinct-spectral case.
It is clear that the tangent directions obtained by differentiating the
determinant family belong to the kernel \eqref{LUa}. To prove that these are the only \(H^1\)-kernel
directions, we again use the left asymptotics and the left-uniqueness
principle established above.

\begin{proof}[\normalfont\bfseries Proof of Theorem~\ref{prop:nondegeneracy}]
Suppose $U^a=(U_1^a,\cdots,U_N^a)$ is a solution of \eqref{FM}. For each \(k=1,\cdots,N\), differentiating the identity
\[
(U_i^a)''
+
2\left(\sum_{j=1}^N (U_j^a)^2\right)U_i^a
=
\eta_i^2U_i^a \quad\hbox{in}\,\ \R,
\ \  i=1,\cdots,N,
\]
with respect to \(a_k\), we obtain that
\(
\mathcal L_{U^a}(\partial_{a_k}U^a)=0\,\ \hbox{in}\,\ \R,
\)
where the operator $\mathcal L_{U^a}$  is given by
\[
(\mathcal L_{U^a}\phi)_i
:=
-\phi_i''
+\eta_i^2\phi_i
-2S^a\phi_i
-4U_i^a\sum_{j=1}^N U_j^a\phi_j\quad\hbox{in}\,\ \R,
\ \  i=1,\cdots,N,
\]
together with $S^a:=\sum_{j=1}^N (U_j^a)^2$.
Furthermore, direct computations yield that $\partial_{a_k}U^a\in H^1(\R)^N$ holds for $a_k\neq 0$. Thus, we have for every \(a\in(\mathbb R\setminus\{0\})^N\),
\[
\operatorname{span}
\left\{
\partial_{a_1}U^a,\cdots,\partial_{a_N}U^a
\right\}
\subset
\ker_{H^1(\mathbb R)^N}\mathcal L_{U^a} .
\]
Conversely, we have for
$\phi=(\phi_1,\cdots,\phi_N)
\in
\ker_{H^1(\mathbb R)^N}\mathcal L_{U^a}$,
\begin{equation}\label{3:15M}
\phi_i''-\eta_i^2\phi_i
=
-2S^a\phi_i
-4U_i^a\sum_{j=1}^N U_j^a\phi_j\ \,\ \hbox{in } \ \R,
\ \  i=1,\cdots,N.
\end{equation}
Since \(U_i^a(x)=O(e^{\eta_i x})\) as \(x\to-\infty\), we obtain that
\(S^a\in L^1((-\infty,R_0))\). Moreover, the one-dimensional Sobolev
embedding theorem gives that
$\phi\in L^\infty(\mathbb R)^N$.

For each \(i\), rewrite the equation (\ref{3:15M}) as
\[
\phi_i''-\eta_i^2\phi_i
=
-2S^a\phi_i+h_i(x) \ \,\ \hbox{in } \ \R,
\]
where
\(
h_i(x):=-4U_i^a(x)\sum_{j=1}^N U_j^a(x)\phi_j(x).
\)
Set
$
\eta_*:=\min_{1\le j\le N}\eta_j,
$
so that
\[
|h_i(x)|
\le
C |U_i^a(x)|\sum_{j=1}^N |U_j^a(x)|
\le
C e^{(\eta_i+\eta_*)x}\quad \text{as }\ x\to-\infty.
\]
In particular,
$
h_i(x)=o(e^{\eta_i x})
$ as $x\to-\infty$.
Applying Lemma~\ref{lem:L1-asymptotic}, we then obtain that there exists a constant
\(b_i^-\in\mathbb R\) such that
\[
\phi_i(x)=b_i^-e^{\eta_i x}+o(e^{\eta_i x})\quad \text{as } \ x\to-\infty.
\]
On the other hand, the determinant formula of \(U^a\) gives that
\[
\partial_{a_k}U_i^a(x)
=
\delta_{ik}e^{\eta_i x}+o(e^{\eta_i x})\quad \text{as}\,\  x\to-\infty.
\]
Define
\[
\psi
:=
\sum_{k=1}^N b_k^-\partial_{a_k}U^a,
\]
so that $\mathcal L_{U^a}\psi=0$ in $\R$
and
\[
\psi_i(x)=b_i^-e^{\eta_i x}+o(e^{\eta_i x})\quad \text{as }\ x\to-\infty.
\]
Hence
\(
z:=\phi-\psi
\)
satisfies
$\mathcal L_{U^a}z=0$ in $\R$,
and
\[
z_i(x)=o(e^{\eta_i x})
\ \ \text{as }\ x\to-\infty, \ \ i=1,\cdots,N.
\]
Moreover, $z(x)=(z_1(x),\cdots,z_N(x))$ satisfies
\begin{equation}\label{3.16}
z_i''-\eta_i^2 z_i
=
-2S^a z_i
-4U_i^a\sum_{j=1}^N U_j^a z_j
=: G_i(x)\quad\hbox{in}\,\ \R.
\end{equation}

Set
\[
Z(x):=\max_{1\le j\le N}e^{-\eta_jx}|z_j(x)|, \ \  x\le R_0.
\]
Since \(U_i^a(x)=O(e^{\eta_i x})\) as \(x\to-\infty\), after further decreasing
\(R_0\) if necessary, we have
\[
|G_i(x)|
\le
C e^{(\eta_i+2\eta_*)x}Z(x),
\ \  x\le R_0,\quad\hbox{where}\,\ \eta_*=\min_{1\leq i\leq N}\eta_i>0.
\]
Thus, \(z\) satisfies the assumptions of
Lemma~\ref{lem:zero-left-principle}, which then  gives that there exists \(R<R_0\) such that
$
z\equiv0 \ \text{on}\ (-\infty,R).
$
In particular,
\[
z_i(R)=z_i'(R)=0,\ \  i=1,\cdots,N.
\]
Since \(z\) satisfies the linearized ODE system \eqref{3.16} with continuous
coefficients, the uniqueness theorem of ODE systems
implies that
$z\equiv0$ on $\R$.
Thus,
\(
\phi
=
\sum_{k=1}^N b_k^-\partial_{a_k}U^a,
\)
which then proves that
\[
\ker_{H^1(\mathbb R)^N}\mathcal L_{U^a}
\subset
\operatorname{span}
\left\{
\partial_{a_1}U^a,\cdots,\partial_{a_N}U^a
\right\}.
\]
Combining the above two inclusions, we thus obtain that
\[
\ker_{H^1(\mathbb R)^N}\mathcal L_{U^a}
=
\operatorname{span}
\left\{
\partial_{a_1}U^a,\cdots,\partial_{a_N}U^a
\right\}.
\]
Finally, since
\[
\partial_{a_k}U_i^a(x)
=
\delta_{ik}e^{\eta_i x}+o(e^{\eta_i x})\quad \text{as } \   x\to-\infty,
\]
the vector-valued functions
$
\partial_{a_1}U^a,\cdots,\partial_{a_N}U^a
$
are linearly independent. Hence,
$
\dim\ker_{H^1(\mathbb R)^N}\mathcal L_{U^a}=N,
$
which completes the proof of Theorem~\ref{prop:nondegeneracy}.
\end{proof}

\section{Solutions for the repeated spectral case} 

\label{sec:multiple-spectral-parameters}

The main purpose of this section is to prove Theorems \ref{thm:completeness-multiple-root}--\ref{thm:L2iden} for   the repeated spectral case.
We first prove  Theorem~\ref{thm:completeness-multiple-root} on the classification of solutions for \eqref{eq:multi-system}. The  nondegeneracy of Theorem~\ref{prop:nondegeneracy-multiple-roots} is then proved in Subsection~\ref{sub4.2} for the linearized operator \(\mathcal L_{U^{a,n}}\) defined in \eqref{eq:multi-LU}, while the proof of Theorem~\ref{thm:L2iden} is  left to  Subsection  \ref{sub4.3}.

Let
\[
\mu_1<\cdots<\mu_K<0
\]
be the distinct spectral parameters, and denote \(k_\alpha\) to be the
multiplicity of \(\mu_\alpha\), so that
\[
k_1+\cdots+k_K=N.
\]
Set
\[
N_0:=0,\ \
N_\alpha:=k_1+\cdots+k_\alpha,
\ \
I_\alpha:=\{N_{\alpha-1}+1,\cdots,N_\alpha\}.
\]
The full list of spectral parameters counted with the multiplicity is denoted
by \(\nu_1,\cdots,\nu_N\) and
\[
\nu_i=\mu_\alpha,\ \  i\in I_\alpha.
\]
For $\alpha=1,2,\cdots,K$, we also write
$
\mu_\alpha=-\eta_\alpha^2<0$.

The key observation of   this section lies in the fact that repeated spectral parameters do not produce new
scalar profiles for the solution. Instead, the components of the solution in the same spectral block are
locked together and differ only by constant coefficients. Hence, each
nontrivial block \(I_\alpha\) reduces to a single scalar profile
\(V_\alpha\), and the collection \(V=(V_1,\cdots,V_K)\) solves the reduced
distinct-spectral system \eqref{eq:reduced-system}. Therefore, the
classification in the repeated spectral  case follows from the
distinct-spectral classification of Theorem~\ref{thm1.1}, together with the
internal rotational freedom inside each block.

We first prove
Theorem~\ref{thm:completeness-multiple-root}, namely the complete
classification of \(H^1(\mathbb R)^N\) solutions for
\eqref{eq:multi-system} with repeated spectral parameters.

\begin{proof}[\normalfont\bfseries Proof of Theorem~\ref{thm:completeness-multiple-root}]
Let
\(
S=\sum_{\ell=1}^N u_\ell^2 ,
\)
and fix \(\alpha\in\{1,\cdots,K\}\). It then yields from \eqref{eq:multi-nu} that for any \(i,j\in I_\alpha\),
\(
\nu_i=\nu_j=\mu_\alpha .
\)
Hence, both \(u_i\) and \(u_j\) solve the same linear equation
\begin{equation}\label{4.1N}
y''+2S y+\mu_\alpha y=0\quad\hbox{in}\,\ \R.
\end{equation}
Define
$
W_{ij}:=u_i'u_j-u_i u_j',
$
so that $W_{ij}'
=
u_i''u_j-u_i u_j''
=0$.
Thus, \(W_{ij}\) is a constant on \(\mathbb R\). Moreover, since
\(u_i,u_j\in H^1(\mathbb R)\), we have
$
u_i'u_j,\ u_i u_j'\in L^1(\mathbb R),
$
which implies that \(W_{ij}\in L^1(\mathbb R)\). We thus have
$
W_{ij}\equiv0\ \hbox{for any }i,j\in I_\alpha.
$

Since $\sum_{i\in I_\alpha} u_i^2(x)\not\equiv 0$, there exists
\(i_\alpha\in I_\alpha\) such that
$u_{i_\alpha}\not\equiv0$.
For any \(i\in I_\alpha\), choose \(x_\alpha\in\mathbb R\) such that
\(u_{i_\alpha}(x_\alpha)\neq0\), and set
\[
c_i:=\frac{u_i(x_\alpha)}{u_{i_\alpha}(x_\alpha)},\ \ i\in I_\alpha .
\]
Define
\begin{equation}\label{4.1M}
w_i:=u_i-c_i u_{i_\alpha},
\end{equation}
so that \(w_i\) satisfies the same linear equation
\[
w_i''+2S w_i+\mu_\alpha w_i=0\ \,\ \hbox{in}\,\ \R.
\]
Moreover, we have $w_i(x_\alpha)=0$.
Since \(W_{i i_\alpha}\equiv0\), we also have
\[
u_i'(x_\alpha)u_{i_\alpha}(x_\alpha)
-
u_i(x_\alpha)u_{i_\alpha}'(x_\alpha)=0,\ \ i\in I_\alpha,
\]
which gives that $w_i'(x_\alpha)=0$, $i\in I_\alpha$.
By the uniqueness of   second-order linear ODE, we then obtain that
$
w_i\equiv0
$ in $\R$,  $i\in I_\alpha$.
We thus derive from \eqref{4.1M} that
\[
u_i=c_i u_{i_\alpha}\ \ \hbox{in}\,\ \R,
\, \  i\in I_\alpha .
\]

Set
\[
d_\alpha:=\left(\sum_{i\in I_\alpha}c_i^2\right)^{1/2}>0,
\ \
V_\alpha:=d_\alpha u_{i_\alpha},
\ \
n_i^{(\alpha)}:=\frac{c_i}{d_\alpha},
\quad i\in I_\alpha .
\]
Then
\[
n^{(\alpha)}=(
n_{N_{\alpha-1}+1}^{(\alpha)},\cdots,
n_{N_\alpha}^{(\alpha)}
)
\in \mathbb S^{k_\alpha-1},
\]
and
\begin{equation}\label{4.1}
u_i=n_i^{(\alpha)}V_\alpha,
\ \  i\in I_\alpha .
\end{equation}
Consequently,
$
\sum_{i\in I_\alpha}u_i^2=V_\alpha^2,
$
and hence
$
S=\sum_{\beta=1}^K V_\beta^2.
$
Substitute $u_i=n_i^{(\alpha)}V_\alpha$
into the equation \eqref{4.1N}, and choose \(i=i_\alpha\) so that
\(n_{i_\alpha}^{(\alpha)}\neq0\). We then obtain that
\begin{equation}\label{4.1MM}
V_\alpha''
+
2\left(\sum_{\beta=1}^K V_\beta^2\right)V_\alpha
+
\mu_\alpha V_\alpha
=0\quad\hbox{in}\,\ \R,
\ \  \alpha=1,\cdots,K.
\end{equation}
Thus,
$
V=(V_1,\cdots,V_K)
$
is an \(H^1(\mathbb R)^K\) solution of the reduced distinct spectral system
\eqref{eq:reduced-system}. Since $\sum_{i\in I_\alpha} u_i^2(x)\not\equiv 0$, we have
\(V_\alpha\not\equiv0\) in $\R$.

Applying Theorem~\ref{thm1.1} to the reduced system
\eqref{4.1MM}, it yields that there exists
$
a=(a_1,\cdots,a_K)\in(\mathbb R\setminus\{0\})^K
$
such that
$
V=V^a=(V^a_1,\cdots,V^a_K).
$
Therefore, we get from \eqref{4.1} that
\[
u_i=n_i^{(\alpha)}V_\alpha^a,
\ \  i\in I_\alpha,\ \ \alpha=1,\cdots,K.
\]
The definition of \(\mathcal M_{\mathrm{multi}}\) in
\eqref{eq:multi-family} then gives that
$
u\in \mathcal M_{\mathrm{multi}},
$
and we are done.
\end{proof}

\subsection{Proof of Theorem~\ref{prop:nondegeneracy-multiple-roots}}\label{sub4.2}
In this subsection, we prove the nondegeneracy of Theorem~\ref{prop:nondegeneracy-multiple-roots} with repeated spectral parameters. The argument of our proof mainly follows the same reduction principle as Theorem \ref{thm1.1}, except that it is now focused on the linearized operator. The component of a kernel element parallel to \(n^{(\alpha)}\) is governed by the reduced \(K\)-component linearized operator, while the orthogonal component inside each block is generated exactly by the  rotations of \(n^{(\alpha)}\).

\begin{proof}[\normalfont\bfseries Proof of Theorem~\ref{prop:nondegeneracy-multiple-roots}]
Let
\(U^{a,n}\in\mathcal M_{\mathrm{multi}}\) be given by
\[
U_i^{a,n}(x):=n_i^{(\alpha)}V_{\alpha}^a(x),
\ \
i\in I_\alpha,\quad \alpha=1,\cdots,K,
\]
where \(V^a=(V_1^a,\cdots,V_K^a)\) is the reduced solution of \eqref{eq:reduced-system} with distinct
spectral parameters, and
\(n^{(\alpha)}=(n_{N_{\alpha-1}+1}^{(\alpha)},\cdots,
n_{N_\alpha}^{(\alpha)}
)\in
\mathbb S^{k_\alpha-1}\). Since \(|n^{(\alpha)}|=1\), we have
\[
S^{a}(x):=\sum_{\ell=1}^N (U_\ell^{a,n})^2
=
\sum_{\beta=1}^K (V_\beta^a)^2 .
\]
The linearized operator at \(U^{a,n}\) is then defined by
\begin{equation}\label{4.2M}
(\mathcal L_{U^{a,n}}\phi)_i
=
-\phi_i''
-\nu_i\phi_i
-2S^{a}\phi_i
-4U_i^{a,n}\sum_{\ell=1}^N U_\ell^{a,n}\phi_\ell\ \ \,\hbox{in}\,\ \R,
\ \  i=1,\cdots,N.
\end{equation}
We also recall the tangent directions of the manifold
\(\mathcal M_{\mathrm{multi}}\) at \(U^{a,n}\). The directions generated by
the reduced parameters are
\[
\Phi^{(\gamma)}=(\Phi_1^{(\gamma)},\cdots,\Phi_N^{(\gamma)}),\ \ \hbox{where}\,\ \Phi_i^{(\gamma)}
=
n_i^{(\alpha)}\partial_{a_\gamma}V_\alpha^a,
\ \
i\in I_\alpha,\quad \gamma=1,\cdots,K.
\]
The rotational directions inside the block \(I_\alpha\) are
\[
\Psi^{(\alpha,\xi)}=(\Psi_1^{(\alpha,\xi)},\cdots,\Psi_N^{(\alpha,\xi)}),\ \ \hbox{where}\,\ \Psi_i^{(\alpha,\xi)}
=
\begin{cases}
\xi_i V_\alpha^a, & i\in I_\alpha,\\
0, & i\notin I_\alpha,
\end{cases}
\ \
\xi\in T_{n^{(\alpha)}}\mathbb S^{k_\alpha-1},
\]
where
\[
T_{n^{(\alpha)}}\mathbb S^{k_\alpha-1}
=
\left\{\xi\in\mathbb R^{k_\alpha}:\ \xi\cdot n^{(\alpha)}=0\right\}.
\]
Thus, \(T_{U^{a,n}}\mathcal M_{\mathrm{multi}}\) is the span of all
\(\Phi^{(\gamma)}\) and \(\Psi^{(\alpha,\xi)}\).

We now prove that
\begin{equation}\label{4.2N}
T_{U^{a,n}}\mathcal M_{\mathrm{multi}}
\subset
\ker_{H^1(\mathbb R)^N}\mathcal L_{U^{a,n}} .
\end{equation}
Indeed, the directions \(\Phi^{(\gamma)}\) are obtained by differentiating
the family \(U^{a,n}\) with respect to the reduced parameters \(a_\gamma\).
Since \(U^{a,n}\) solves the full \(N\)-component system for all admissible
\(a\), we have
\[
\mathcal L_{U^{a,n}}\Phi^{(\gamma)}=0\quad\hbox{in}\,\ \R,
\ \  \gamma=1,\cdots,K.
\]
Similarly, fix $\alpha$ and let
$\xi\in\mathbb R^{k_\alpha}$ satisfy
$\xi\cdot n^{(\alpha)}=0$.
Then there exists a smooth curve
$n_\varepsilon^{(\alpha)}\in\mathbb S^{k_\alpha-1}$
such that
\[
n_0^{(\alpha)}=n^{(\alpha)},
\ \
\frac{d}{d\varepsilon}n_\varepsilon^{(\alpha)}
\Big|_{\varepsilon=0}
=\xi.
\]
Define the perturbed solution $U^\varepsilon=(U_1^\varepsilon,\cdots,U_N^\varepsilon)$ by
\[
U_i^\varepsilon
=
\begin{cases}
n_{\varepsilon,i}^{(\alpha)} V_\alpha^a,
& i\in I_\alpha,\\
U_i^{a,n},
& i\notin I_\alpha.
\end{cases}
\]
Since the system is invariant under orthogonal
rotations inside each block $I_\alpha$,
$U^\varepsilon=(U_1^\varepsilon,\cdots,U_N^\varepsilon)$ is still a solution of \eqref{eq:multi-system} for all
$\varepsilon$.
Differentiating \eqref{eq:multi-system} with respect to
\(\varepsilon\) and taking $\varepsilon=0$, it then
yields that
$
\mathcal L_{U^{a,n}} \Psi^{(\alpha,\xi)}=0
$  in $\R$.
This proves \eqref{4.2N}.

It remains to prove the reverse inclusion of \eqref{4.2N}. Set
\[
\phi=(\phi_1,\cdots,\phi_N)
\in
\ker_{H^1(\mathbb R)^N}\mathcal L_{U^{a,n}} .
\]
For each block \(I_\alpha\), decompose \(\phi\) into its component parallel
to \(n^{(\alpha)}\) and its orthogonal component. Define
\[
p_\alpha
:=
\sum_{i\in I_\alpha}n_i^{(\alpha)}\phi_i,\quad r_i
:=
\phi_i-n_i^{(\alpha)}p_\alpha,
\ \  i\in I_\alpha,\]
so that
$
\phi_i=n_i^{(\alpha)}p_\alpha+r_i
$ for $i\in I_\alpha$,
and
$
\sum_{i\in I_\alpha}n_i^{(\alpha)}r_i=0.
$
Multiplying the \(i\)-th equation of ${\mathcal L}_{U^{a,n}}\phi=0$ in $\R$ by \(n_i^{(\alpha)}\) and summing over
\(i\in I_\alpha\), we obtain that
\begin{equation}\label{4.2}
\begin{split}
&-\Big(\sum_{i\in I_\alpha}n_i^{(\alpha)}\phi_i\Big)''
-\mu_\alpha \sum_{i\in I_\alpha}n_i^{(\alpha)}\phi_i
-2S^a\sum_{i\in I_\alpha}n_i^{(\alpha)}\phi_i
\\
&\quad
-4\Big(\sum_{\ell=1}^N U_\ell^{a,n}\phi_\ell\Big)
\sum_{i\in I_\alpha}n_i^{(\alpha)}U_i^{a,n}
=0\ \ \hbox{in}\,\ \R.
\end{split}
\end{equation}
Since
\(
\sum_{i\in I_\alpha}n_i^{(\alpha)}\phi_i=p_\alpha
\),
\[
\sum_{i\in I_\alpha}n_i^{(\alpha)}U_i^{a,n}
=
V_\alpha^a\sum_{i\in I_\alpha}(n_i^{(\alpha)})^2
=
V_\alpha^a\ \ \ \text{and}\ \ \ \sum_{\ell=1}^N U_\ell^{a,n}\phi_\ell
= \sum_{\beta=1}^K V_\beta^a p_\beta,
\]
we get from \eqref{4.2} that $p_\alpha$ solves
\begin{equation}\label{4.2NN}
-p_\alpha''
-\mu_\alpha p_\alpha
-2S^{a}p_\alpha
-4V_\alpha^a\sum_{\beta=1}^K V_\beta^a p_\beta
=0\quad\hbox{in}\,\ \R.
\end{equation}
Hence, \(p=(p_1,\cdots,p_K)\in
\ker_{H^1(\mathbb R)^K}\mathcal L_{V^a}\).
By the nondegeneracy of the reduced solution in Theorem~\ref{prop:nondegeneracy}, we then have
\[
\ker_{H^1(\mathbb R)^K}\mathcal L_{V^a}
=
\operatorname{span}
\left\{
\partial_{a_1}V^a,\cdots,\partial_{a_K}V^a
\right\},
\]
and hence there exist constants \(c_1,\cdots,c_K\) such that
\(
p=\sum_{\gamma=1}^K c_\gamma \partial_{a_\gamma}V^a.
\)

We now identify the orthogonal part. Recall that \(r_i=\phi_i-n_i^{(\alpha)}p_\alpha\). For \(i\in I_\alpha\), subtract
\(n_i^{(\alpha)}\) times the equation \eqref{4.2NN} from the \(i\)-th
equation  of \(\mathcal L_{U^{a,n}}\phi=0\), and use
\[
U_i^{a,n}=n_i^{(\alpha)}V_\alpha^a,
\ \
\sum_{\ell=1}^N U_\ell^{a,n}\phi_\ell
=
\sum_{\beta=1}^K V_\beta^a p_\beta.
\]
It then yields that
\[
-r_i''
-\mu_\alpha r_i
-2S^{a} r_i
=0\ \ \hbox{in}\,\ \R,
\ \  i\in I_\alpha .
\]
Therefore, both \(r_i\) and \(V_\alpha^a\) solve the same linear equation
\begin{equation}\label{4.3}
-y''
-\mu_\alpha y
-2S^a y=0\ \ \hbox{in}\,\ \R,
\end{equation}
and their Wronskian
$
W_i^{(\alpha)}
:=
r_i(V_\alpha^a)'-r_i'V_\alpha^a
$
is a constant, due to the fact that $(W_i^{(\alpha)})'=0$. Since \(r_i,V_\alpha^a\in H^1(\mathbb R)\), we have
$
r_i(V^a_\alpha)',\ r_i'V_\alpha^a\in L^1(\mathbb R).
$
Hence, \(W_i^{(\alpha)}\in L^1(\mathbb R)\), which further gives that
\(
W_i^{(\alpha)}\equiv0\) in \(\R\).
Since \(V^a_\alpha\not\equiv0\), choose \(x_0\in\mathbb R\) such that
$
V^a_\alpha(x_0)\neq0.
$
Set
$
d_i^{(\alpha)}:=\frac{r_i(x_0)}{V^a_\alpha(x_0)},
$
so that
$
r_i(x_0)=d_i^{(\alpha)}V^a_\alpha(x_0).
$
Moreover, since $W_i^{(\alpha)}\equiv0$, we have
$
r_i'(x_0)=d_i^{(\alpha)}(V^a_\alpha)'(x_0).
$
Since \(r_i\) and \(d_i^{(\alpha)}V^a_\alpha\) solve the same
linear equation \eqref{4.3},  the uniqueness of solutions for ODE gives that
\[
r_i\equiv d_i^{(\alpha)}V^a_\alpha,
\ \  i\in I_\alpha.
\]
The orthogonality condition thus gives that
$
\sum_{i\in I_\alpha}n_i^{(\alpha)}r_i=0.
$
This implies that
$
\sum_{i\in I_\alpha}n_i^{(\alpha)}d_i^{(\alpha)}=0,
$
i.e.,
$
d^{(\alpha)}\cdot n^{(\alpha)}=0.
$

By the parallel and orthogonal parts, we now obtain that for \(i\in I_\alpha\),
\[
\phi_i
=
n_i^{(\alpha)}
\sum_{\gamma=1}^K c_\gamma \partial_{a_\gamma}V^a_\alpha
+
d_i^{(\alpha)}V^a_\alpha,
\]
where
$
d^{(\alpha)}\cdot n^{(\alpha)}=0.
$
Thus, \(\phi=(\phi_1,\cdots,\phi_N)\) is a linear combination of the parameter directions
\(\Phi^{(\gamma)}\) and the rotation directions
\(\Psi^{(\alpha,\xi)}\). Hence
\[
\ker_{H^1(\mathbb R)^N}\mathcal L_{U^{a,n}}
\subset
T_{U^{a,n}}\mathcal M_{\mathrm{multi}}.
\]
Together with \eqref{4.2N}, this proves that
\[
\ker_{H^1(\mathbb R)^N}\mathcal L_{U^{a,n}}
=
T_{U^{a,n}}\mathcal M_{\mathrm{multi}}.
\]
Finally, the parameter directions are independent, due to the nondegeneracy
of the reduced solution and the local parametrization \(a\mapsto V^a\).
The rotation directions in the block \(I_\alpha\) form the tangent space
\(T_{n^{(\alpha)}}\mathbb S^{k_\alpha-1}\), and hence contribute to
\(k_\alpha-1\) independent directions. These two types of directions are
independent, because they correspond respectively to the parallel and
orthogonal components in each block. We thus have
\[
\dim T_{U^{a,n}}\mathcal M_{\mathrm{multi}}
=
K+\sum_{\alpha=1}^K(k_\alpha-1)
=
N,
\]
and the proof of Theorem \ref{prop:nondegeneracy-multiple-roots} is therefore complete.
\end{proof}


\subsection{The $L^2$-mass formula}\label{sub4.3}
In this subsection, we shall address the proof of Theorem~\ref{thm:L2iden}, which is concerned with the $L^2$-mass formula of solutions for \eqref{FM}.

For the repeated spectral  case, it is shown in Theorem~\ref{thm:completeness-multiple-root} that the total \(L^2\)-mass of each spectral block is reduced to the \(L^2\)-mass of a single effective component in an associated distinct-spectral system.  Thus, in order to prove Theorem~\ref{thm:L2iden}, it suffices to derive the corresponding
mass identity of the distinct-spectral system by
two steps.  In the first step, we shall establish the pointwise algebraic identity
\eqref{eq:algebraic-identity} below for the solutions of the
distinct-spectral system.  This identity is obtained from the polynomial
identity in the spectral parameter given by
Lemma~\ref{lem:polynomial-identity-function}.  In the second step, we shall apply this
algebraic identity to the reduced block system, which then yields the
first-order differential identity \eqref{eq:Q-prime} below.  Integrating the identity
\eqref{eq:Q-prime} over \(\mathbb R\) finally finishes the proof of Theorem~\ref{thm:L2iden}.
We start with the following algebraic identity.

\begin{proposition}\label{thm:algiden}
Under the same assumptions of
Theorem~\ref{thm:L2iden}, we have for any \(\alpha=1,\cdots,K\),
\begin{equation}\label{eq:algebraic-identity}
\sum_{j\in I_\alpha}\Big\{
\Big[
(u_j')^2+\big(\sum_{l=1}^N u_l^2+\mu_\alpha\big)u_j^2
\Big]
+
\sum_{\substack{\beta\neq\alpha\\ l\in I_\beta}}
\frac{(u_j'u_l-u_ju_l')^2}{\mu_\beta-\mu_\alpha}\Big\}
=0\,\ \ \hbox{in}\ \ \R.
\end{equation}
\end{proposition}

We remark that the constants of motion for \eqref{FM} were established in
\cite{GLW2026}. For convenience, we record
them in their generating polynomial form. More precisely, when the left-hand
side of \eqref{eq:polynomial-identity-function} below is expanded as a
polynomial in \(\lambda\), it has degree at most \(N-1\), and the coefficient of \(\lambda^{N-k}\) is exactly the
\(k\)-th constant of motion in \cite[Lemma~5.1]{GLW2026}  for each
\(k=1,\cdots,N\). In order to prove Proposition~\ref{thm:algiden},
we first prove the following constants of motion in a simplified approach.

\begin{lemma}\label{lem:polynomial-identity-function}
For $\mu_1\leq \mu_2\leq\cdots\leq \mu_N<0$, let \(u=(u_1,\cdots,u_N)\in H^1(\mathbb R)^N\) be a solution of
\eqref{FM}, namely
\[
-u_j''-2Su_j=\mu_j u_j\ \ \hbox{in}\ \ \R,\ \
S=\sum_{i=1}^N u_i^2 ,
\]
and set
$
W_{jl}:=u_j'u_l-u_j u_l'$ for $1\leq j,l\leq N$. Then the following polynomial identity holds
\begin{equation}\label{eq:polynomial-identity-function}
	\begin{split}
		\sum_{j=1}^N
		\left[
		(u_j')^2+(S+\mu_j)u_j^2
		\right]
		\prod_{i\neq j}(\lambda+\mu_i)
		+
		\sum_{1\leq j<l\leq N}
		W_{jl}^2
		\prod_{i\neq j,l}(\lambda+\mu_i)
		\equiv 0 \ \ \text{in} \ \ \R
	\end{split}
\end{equation}
as a polynomial identity in \(\lambda\).
\end{lemma}

\begin{proof}
Set for \(j=1,\cdots,N\),
\[
P_j:=(u_j')^2+(S+\mu_j)u_j^2 ,
\]
and define for
\(\lambda\notin\{-\mu_1,\cdots,-\mu_N\}\),
\begin{equation}\label{5.3M}
\mathcal C(\lambda,x)
:=
\sum_{j=1}^N
\frac{P_j}{\lambda+\mu_j}
+
\sum_{1\leq j<l\leq N}
\frac{W_{jl}^2}{(\lambda+\mu_j)(\lambda+\mu_l)} .
\end{equation}
We claim that \(\mathcal C(\lambda,x)\) is independent of \(x\).

In order to prove the above claim, we note that
\begin{equation}\label{5.2A}
\begin{split}
\partial_x \mathcal C(\lambda,x)
&=
\sum_{j=1}^N
\frac{P_j'}{\lambda+\mu_j}
+
\sum_{1\leq j<l\leq N}
\frac{2W_{jl}W_{jl}'}{(\lambda+\mu_j)(\lambda+\mu_l)} .
\end{split}
\end{equation}
Since
\[
u_j''=-(2S+\mu_j)u_j,
\]
we obtain that
\[
\begin{split}
P_j'=2u_j'u_j''+S'u_j^2+2(S+\mu_j)u_ju_j'
=S'u_j^2-2Su_ju_j' .
\end{split}
\]
Using \(S'=2\sum_{l=1}^N u_lu_l'\), we then have
\[
P_j'
=
-2\sum_{l=1}^N u_j u_l W_{jl}.
\]
This yields from (\ref{5.2A}) that
\begin{equation}\label{5.2M}
\begin{split}
\partial_x\mathcal C(\lambda,x)
&=
\sum_{j=1}^N \frac{P_j'}{\lambda+\mu_j}
+
\sum_{1\leq j<l\leq N}
\frac{2W_{jl}W_{jl}'}{(\lambda+\mu_j)(\lambda+\mu_l)}
\\
&=
-2\sum_{j=1}^N\sum_{l=1}^N
\frac{u_j u_l W_{jl}}{\lambda+\mu_j}
+
\sum_{1\leq j<l\leq N}
\frac{2W_{jl}W_{jl}'}{(\lambda+\mu_j)(\lambda+\mu_l)} .
\end{split}
\end{equation}
Since
\[
W_{jl}'
=
u_j''u_l-u_ju_l''
=
(\mu_l-\mu_j)u_ju_l ,
\]
it follows from (\ref{5.2M}) that \(\partial_x\mathcal C(\lambda,x)\) is a sum of those terms containing \(u_ju_lW_{jl}\). Note that $W_{jl}=0$ holds for each   pair \(j=l\).  As for each   pair \(j<l\), the contribution of the first sum on the right hand side of (\ref{5.2M}) satisfies
\begin{equation}\label{5.2MM}
-\frac{2}{\lambda+\mu_j}
+
\frac{2}{\lambda+\mu_l}
=
\frac{2(\mu_j-\mu_l)}
{(\lambda+\mu_j)(\lambda+\mu_l)} .
\end{equation}
The contribution of the second sum on the right hand side of (\ref{5.2M}) satisfies
\(
\frac{2(\mu_l-\mu_j)}
{(\lambda+\mu_j)(\lambda+\mu_l)} ,
\)
which is then cancelled with (\ref{5.2M})  for every pair
\(j<l\).
We hence conclude from (\ref{5.2M}) that
\[
\partial_x\mathcal C(\lambda,x)=0,
\]
which proves the claim that \(\mathcal C(\lambda,x)\) is independent of \(x\).

On the other hand, the assumption $u\in H^1(\R)^N$ gives that $\mathcal C(\lambda,x)\in L^1(\R)$. Since $\mathcal C(\lambda,x)$ is a constant in \(x\), we now have
\[
\mathcal C(\lambda,x)\equiv0,\ \ x\in\R.
\]
Multiplying \eqref{5.3M} by
\(
\prod_{i=1}^N(\lambda+\mu_i)
\), it
gives  that \eqref{eq:polynomial-identity-function} holds for every
\(\lambda\notin\{-\mu_1,\cdots,-\mu_N\}\). Since the left-hand side of \eqref{eq:polynomial-identity-function} is a
polynomial in \(\lambda\), the identity \eqref{eq:polynomial-identity-function} holds for all \(\lambda\).
\end{proof}

\begin{lemma}\label{lem:algebraic-identity-function}
Under the assumptions of Lemma~\ref{lem:polynomial-identity-function},  assume
$
\mu_1<\mu_2<\cdots<\mu_N<0.
$
Then we have for each \(j=1,2,\cdots,N\),
\begin{equation}\label{eq:algebraic-identity-function}
(u_j')^2+(S+\mu_j)u_j^2
+
\sum_{l\neq j}
\frac{(u_j'u_l-u_ju_l')^2}{\mu_l-\mu_j}
=0
\ \ \text{in }\ \mathbb R, \ \ S=\sum_{i=1}^Nu_i^2.
\end{equation}
\end{lemma}

\begin{proof}
By Lemma~\ref{lem:polynomial-identity-function}, we have for $
W_{rs}:=u'_ru_s-u_r u'_s$,
\begin{equation}\label{5.3A}
\begin{split}
\sum_{r=1}^N
\left[
(u_r')^2+(S+\mu_r)u_r^2
\right]
\prod_{i\neq r}(\lambda+\mu_i)
+
\sum_{1\leq r<s\leq N}
W_{rs}^2
\prod_{i\neq r,s}(\lambda+\mu_i)
\equiv0\quad\hbox{in}\,\ \R.
\end{split}
\end{equation}
Fix \(j\in\{1,2,\cdots,N\}\). Since the numbers
\(\mu_1,\cdots,\mu_N\) are pairwise distinct, we  evaluate this
polynomial identity at \(\lambda=-\mu_j\). Then all terms in the
first sum of \eqref{5.3A} vanish except the term with \(r=j\), and all terms in the
second sum of \eqref{5.3A} vanish except those involving the index \(j\). This gives from \eqref{5.3A} that for $j=1,\cdots,N,$
\[
\left[
(u_j')^2+(S+\mu_j)u_j^2
\right]
\prod_{i\neq j}(\mu_i-\mu_j)
+
\sum_{l\neq j}
(u_j'u_l-u_ju_l')^2
\prod_{i\neq j,l}(\mu_i-\mu_j)
=0\ \ \hbox{in}\,\ \R,
\]
from which it follows that
\[
(u_j')^2+(S+\mu_j)u_j^2
+
\sum_{l\neq j}
\frac{(u_j'u_l-u_ju_l')^2}{\mu_l-\mu_j}
=0\ \ \hbox{in}\,\ \R.
\]
This proves \eqref{eq:algebraic-identity-function}, and we are done.
\end{proof}

Applying the above two lemmas, we are now ready to prove Proposition~\ref{thm:algiden}.

\begin{proof}[\normalfont \bfseries Proof of Proposition~\ref{thm:algiden}]
By Theorem~\ref{thm:completeness-multiple-root}, we have
\(u\in\mathcal M_{\mathrm{multi}}\). Hence, there exist
\[
a=(a_1,\cdots,a_K)\in(\mathbb R\setminus\{0\})^K\,\ \hbox{and}\,\ n^{(\alpha)}
=
(n^{(\alpha)}_{N_{\alpha-1}+1},\cdots,n^{(\alpha)}_{N_\alpha})
\in\mathbb S^{k_\alpha-1},
\ \  \alpha=1,\cdots,K,
\]
such that
$
u_j(x)=n_j^{(\alpha)}V_\alpha^a(x)
$ holds for all $j\in I_\alpha$.
For simplicity, we write \(V_\alpha=V_\alpha^a\). Since
$
\sum_{j\in I_\alpha}\bigl(n_j^{(\alpha)}\bigr)^2=1,
$
we have
\[
\sum_{j\in I_\alpha}u_j^2=V_\alpha^2,
\ \
\sum_{j\in I_\alpha}(u_j')^2=(V_\alpha')^2.
\]
Consequently,
\[
S=\sum_{j=1}^N u_j^2
=
\sum_{\gamma=1}^K V_\gamma^2.
\]
Moreover, we have for \(\alpha\neq\beta\), \(j\in I_\alpha\), and
\(l\in I_\beta\),
\[
u_j'u_l-u_ju_l'
=
n_j^{(\alpha)}n_l^{(\beta)}
\bigl(V_\alpha'V_\beta-V_\alpha V_\beta'\bigr).
\]
We then have
\[
\begin{aligned}
\sum_{\substack{j\in I_\alpha\\ l\in I_\beta}}
(u_j'u_l-u_ju_l')^2
&  =
\Big(
\sum_{j\in I_\alpha}
\bigl(n_j^{(\alpha)}\bigr)^2
\Big)
\Big(
\sum_{l\in I_\beta}
\bigl(n_l^{(\beta)}\bigr)^2
\Big)
\bigl(V_\alpha'V_\beta-V_\alpha V_\beta'\bigr)^2
=
\bigl(V_\alpha'V_\beta-V_\alpha V_\beta'\bigr)^2 .
\end{aligned}
\]

Substituting above identities into the left-hand side of
\eqref{eq:algebraic-identity}, we obtain that
\begin{equation}\label{5.8K}
\mbox{LHS  of \eqref{eq:algebraic-identity}}= (V_\alpha')^2
+
\Big(
\sum_{\gamma=1}^K V_\gamma^2+\mu_\alpha
\Big)V_\alpha^2
+
\sum_{\beta\neq\alpha}
\frac{
\bigl(V_\alpha'V_\beta-V_\alpha V_\beta'\bigr)^2
}{\mu_\beta-\mu_\alpha}\quad\hbox{in}\,\ \R.
\end{equation}
Note that the reduced profile
$
V=(V_1,\cdots,V_K)
$
solves the distinct spectral system
\begin{equation}\label{5.3B}
V_\alpha''
+
2\Big(\sum_{\gamma=1}^K V_\gamma^2\Big)V_\alpha
+
\mu_\alpha V_\alpha
=0\quad\hbox{in}\,\ \R,
\ \  \alpha=1,\cdots,K.
\end{equation}
Applying Lemma~\ref{lem:algebraic-identity-function} to the
\(K\)-component distinct spectral system \eqref{5.3B}, we obtain from \eqref{5.8K} that
\[
\mbox{LHS  of \eqref{eq:algebraic-identity}}= (V_\alpha')^2
+
\Big(
\sum_{\gamma=1}^K V_\gamma^2+\mu_\alpha
\Big)V_\alpha^2
+
\sum_{\beta\neq\alpha}
\frac{
\bigl(V_\alpha'V_\beta-V_\alpha V_\beta'\bigr)^2
}{\mu_\beta-\mu_\alpha}
=0\quad\hbox{in}\,\ \R,
\]
where $\alpha=1,\cdots,K$. Proposition~\ref{thm:algiden} is therefore proved.
\end{proof}

We are now ready to prove the exact \(L^2\)-mass formula of Theorem~\ref{thm:L2iden}. We remark that the pointwise algebraic identity of Proposition \ref{thm:algiden} allows us to express the mass of each spectral block in terms of its corresponding spectral parameter.


\begin{proof}[\normalfont \bfseries Proof of Theorem~\ref{thm:L2iden}]
Recall that
$S=\sum_{i=1}^N u_i^2$.
For each
\(\alpha=1,\cdots,K\),   Theorem~\ref{thm:completeness-multiple-root} shows that there exist
\(n^{(\alpha)}
=
\bigl(
n_{N_{\alpha-1}+1}^{(\alpha)},\cdots,
n_{N_\alpha}^{(\alpha)}
\bigr)
\in \mathbb S^{k_\alpha-1}\) and \(0\neq v_\alpha\in
H^1(\mathbb R)\cap C^2(\mathbb R)\) such that
\[
u_i=n_i^{(\alpha)}v_\alpha\quad\hbox{in}\ \ \R,\ \  i\in I_\alpha .
\]
Hence, $S=\sum_{\gamma=1}^K v_\gamma^2$
and
\[
\int_{\mathbb R}\sum_{i\in I_\alpha}u_i^2\,dx
=
\int_{\mathbb R}v_\alpha^2\,dx\neq 0.
\]
It next suffices to prove that
$
\int_{\mathbb R}v_\alpha^2\,dx=2\eta_\alpha>0.
$

By Theorem \ref{thm:completeness-multiple-root}, the reduced function \(v_\alpha\) satisfies
\begin{equation}\label{eq:reduced-v-equation}
v_\alpha''+(2S+\mu_\alpha)v_\alpha=0\quad\hbox{in}\,\ \R,
\ \  \alpha=1,\cdots,K .
\end{equation}
Fix \(\alpha\)  and  define   for \(\beta\neq\alpha\),
\[
\delta_{\alpha\beta}:=\mu_\beta-\mu_\alpha,
\ \
W_{\alpha\beta}:=v_\alpha'v_\beta-v_\alpha v_\beta' .
\]
We then derive from \eqref{eq:reduced-v-equation} that
\begin{equation}\label{eq:Wab-prime}
W_{\alpha\beta}'
=
\delta_{\alpha\beta}v_\alpha v_\beta .
\end{equation}
Moreover, Proposition~\ref{thm:algiden} gives the following identity
\begin{equation}\label{eq:reduced-alg-identity}
(v_\alpha')^2+(S+\mu_\alpha)v_\alpha^2
+
\sum_{\beta\neq\alpha}
\frac{W_{\alpha\beta}^2}{\delta_{\alpha\beta}}
=0\ \ \hbox{in}\,\ \R.
\end{equation}
Define
\begin{equation}\label{5.12M}
Q_\alpha
:=
\frac{v_\alpha'}{v_\alpha}
+
\sum_{\beta\neq\alpha}
\frac{W_{\alpha\beta}v_\beta}
{\delta_{\alpha\beta}v_\alpha}\quad\hbox{in}\,\ \{v_\alpha\neq0\}.
\end{equation}
We first compute \(Q_\alpha'\) on the open set \(\{v_\alpha\neq0\}\): it follows from
\eqref{eq:reduced-v-equation} and \eqref{eq:reduced-alg-identity} that
\begin{equation}\label{5.6A}
\begin{aligned}
\left(\frac{v_\alpha'}{v_\alpha}\right)'
=
\frac{v_\alpha''}{v_\alpha}
-
\frac{(v_\alpha')^2}{v_\alpha^2}=
-(2S+\mu_\alpha)
-
\frac{(v_\alpha')^2}{v_\alpha^2}   =
-S+
\sum_{\beta\neq\alpha}
\frac{W_{\alpha\beta}^2}
{\delta_{\alpha\beta}v_\alpha^2}.
\end{aligned}
\end{equation}
We also derive from \eqref{eq:Wab-prime} that
\begin{equation}\label{5.6B}
\begin{aligned}
\left(
\frac{W_{\alpha\beta}v_\beta}
{\delta_{\alpha\beta}v_\alpha}
\right)'
&=
\frac{W_{\alpha\beta}'v_\beta+W_{\alpha\beta}v_\beta'}
{\delta_{\alpha\beta}v_\alpha}
-
\frac{W_{\alpha\beta}v_\beta v_\alpha'}
{\delta_{\alpha\beta}v_\alpha^2} \\
&=
v_\beta^2
+
\frac{W_{\alpha\beta}}
{\delta_{\alpha\beta}v_\alpha^2}
\left(v_\alpha v_\beta'-v_\alpha'v_\beta\right)
=v_\beta^2-
\frac{W_{\alpha\beta}^2}
{\delta_{\alpha\beta}v_\alpha^2}.
\end{aligned}
\end{equation}
Summing \eqref{5.6B} over \(\beta\neq\alpha\), it then yields from \eqref{5.6A} that
\begin{equation}\label{eq:Q-prime}
Q_\alpha'=-v_\alpha^2
\ \ \, \text{on }\, \{v_\alpha\neq0\}.
\end{equation}

It remains to justify that \(Q_\alpha\) can be extended across the zero points of
\(v_\alpha\). Write
\[
Q_\alpha=\frac{P_\alpha}{v_\alpha},
\ \
P_\alpha:=
v_\alpha'
+
\sum_{\beta\neq\alpha}
\frac{W_{\alpha\beta}v_\beta}{\delta_{\alpha\beta}} ,
\]
and let \(x_0\) be a zero point of \(v_\alpha\). Since \(v_\alpha\) is a nontrivial
solution of the linear equation \eqref{eq:reduced-v-equation}, we have
\(v_\alpha'(x_0)\neq0\). Hence, \(x_0\) is a simple zero point of \(v_\alpha\).
On the other hand, evaluating \eqref{eq:reduced-alg-identity} at \(x_0\)
and using $
W_{\alpha\beta}(x_0)=v_\alpha'(x_0)v_\beta(x_0), $
we obtain that
\[
0
=
(v_\alpha'(x_0))^2
\Big(
1+
\sum_{\beta\neq\alpha}
\frac{v_\beta(x_0)^2}{\delta_{\alpha\beta}}
\Big).
\]
Since \(v_\alpha'(x_0)\neq0\), this implies that
$1+\sum_{\beta\neq\alpha}
\frac{v_\beta(x_0)^2}{\delta_{\alpha\beta}}
=0. $
Consequently,
\[
P_\alpha(x_0)
=
v_\alpha'(x_0)
\left(
1+
\sum_{\beta\neq\alpha}
\frac{v_\beta(x_0)^2}{\delta_{\alpha\beta}}
\right)
=0 .
\]
Thus, both \(P_\alpha\) and \(v_\alpha\) vanish at \(x_0\), while
\(v_\alpha\) has a simple zero at \(x_0\). So \(P_\alpha/v_\alpha\) has a
finite limit at \(x_0\). We extend \(Q_\alpha\) to \(x_0\) by this limit.
Since \(P_\alpha\) and \(v_\alpha\) are \(C^1\), this extension is continuous.
Extending \(Q_\alpha\) by the above removable limits, we obtain a continuous
function on \(\mathbb R\). Since the zeros of \(v_\alpha\) are isolated,
there are only finitely many zeros in any compact interval. Hence, integrating
\eqref{eq:Q-prime} on each nodal interval and summing them, the interior boundary
terms are cancelled by the continuity of \(Q_\alpha\). We thus get from \eqref{eq:Q-prime} that for any
\(R_1<R_2\),
\begin{equation}\label{eq:int-v-Q}
\int_{R_1}^{R_2}v_\alpha^2\,dx
=
Q_\alpha(R_1)-Q_\alpha(R_2).
\end{equation}

Finally, we derive from Lemma~\ref{lem:L1-asymptotic} that there exist nonzero constants
\(c_\gamma^-\) and \(c_\gamma^+\) such that  
\[
v_\gamma(x)=c_\gamma^-e^{\eta_\gamma x}
+o(e^{\eta_\gamma x}),
\ \
v_\gamma'(x)=\eta_\gamma c_\gamma^-e^{\eta_\gamma x}
+o(e^{\eta_\gamma x})
\quad \text{as }\ x\to-\infty,
\]
and
\begin{equation}\label{5.15M}
v_\gamma(x)=c_\gamma^+e^{-\eta_\gamma x}
+o(e^{-\eta_\gamma x}),
\ \
v_\gamma'(x)=-\eta_\gamma c_\gamma^+e^{-\eta_\gamma x}
+o(e^{-\eta_\gamma x})
\quad \text{as }\ x\to+\infty,
\end{equation}
where $\eta_\gamma=\sqrt{|\mu_\gamma|}>0$ and $\gamma=1,2,\cdots,K$.
The right-hand asymptotics \eqref{5.15M} follow by applying
Lemma~\ref{lem:L1-asymptotic} to the reflected function
\(\widetilde v_\gamma(x):=v_\gamma(-x)\). The corresponding leading
coefficients of \eqref{5.15M} cannot vanish, since otherwise Lemma~\ref{lem:L1-asymptotic} would
imply that \(v_\gamma\) vanishes on a right half-line, and hence
\(v_\gamma\equiv0\) by the uniqueness of ODE.
We hence derive from above that
\[
\frac{v_\alpha'}{v_\alpha}\to\eta_\alpha
\quad\text{as }\ x\to-\infty,
\ \ \
\frac{v_\alpha'}{v_\alpha}\to-\eta_\alpha
\quad\text{as }\ x\to+\infty.
\]
Moreover, for every \(\beta\neq\alpha\),
\[
\frac{W_{\alpha\beta}v_\beta}{v_\alpha}
=O(e^{2\eta_\beta x})
\quad\text{as }\ x\to-\infty,\ \ \ \frac{W_{\alpha\beta}v_\beta}{v_\alpha}
=O(e^{-2\eta_\beta x})
\quad\text{as }\ x\to+\infty.
\]
This yields from \eqref{5.12M}  that
\[
Q_\alpha(x)\to\eta_\alpha
\quad\text{as }\ x\to-\infty,
\ \ \
Q_\alpha(x)\to-\eta_\alpha
\quad\text{as } \ x\to+\infty.
\]
Setting \(R_1\to-\infty\) and \(R_2\to+\infty\) for
\eqref{eq:int-v-Q}, we then conclude from above that
$
\int_{\mathbb R}v_\alpha^2\,dx
=
\eta_\alpha-(-\eta_\alpha)
=
2\eta_\alpha=2\sqrt{|\mu_\alpha|}.
$
Consequently, we have for every $\alpha=1,2,\cdots,K$,
\[
\int_{\mathbb R}\sum_{i\in I_\alpha}u_i^2\,dx
=
\int_{\mathbb R}v_\alpha^2\,dx
=
2\eta_\alpha=2\sqrt{|\mu_\alpha|}>0.
\]
This therefore completes the proof of Theorem~\ref{thm:L2iden}.
\end{proof}

\section{Profile reduction for existing minimizers as $p\to2$}
\label{sec:profile-reduction}
The rest part of this paper is devoted to the proof of
Theorem~\ref{thm:main}, where $N\ge 2$ is fixed. The purpose of this section is to reduce any
existing minimizer of $J_N(p)$ into a precise local description around
widely separated scalar profiles  as \(p\to2\). We first establish the profile
decomposition, and then introduce exact orthonormal coordinates and the translation
parameters, based on which we finally obtain in this section a coercive energy estimate.

Recall that every minimizer $U\in\mathcal M_N$ of $J_N(p)$ satisfies the matrix
Euler--Lagrange equation
\begin{equation}\label{eq:matrix-euler-lagrange}
-U''-|U|^{2p-2}U=\Lambda U
\quad\hbox{in }\R,\ \ \Lambda=\Lambda^\top\in\R^{N\times N},
\end{equation}
where $\Lambda$ is a Lagrange multiplier matrix. For a matrix
$A=(a_{ij})\in\R^{N\times N}$, its Frobenius norm is defined by
\begin{equation}\label{eq:frobenius-norm}
\|A\|_{\mathrm F}:=\Big(\sum_{i,j=1}^N |a_{ij}|^2\Big)^{1/2}.
\end{equation}
For \(p=2\), it yields from \cite[Theorem~7]{ii} that
\begin{equation}\label{48}
J_N(2)=NJ_1(2)=-\frac N{48},\ \ N\ge1.
\end{equation}
Let $Q_p>0$ and the associated  Lagrange multiplier $\kappa_p\in\R$ be as in \eqref{eq:scalar-profile-definition}, so that $Q_p$ satisfies $\|Q_p\|_2=1$, $E_p(Q_p)=J_1(p)$, and
\begin{equation}\label{q}
-Q_p''+\kappa_p^2Q_p-|Q_p|^{2p-2}Q_p=0\quad\hbox{in }\ \R,
\end{equation}
By the uniqueness of the initial value problem  \((u')^2 = \kappa_p^2 u^2 - \frac{|u|^{2p}}{p}\),  it then follows that every nontrivial real \(H^1(\mathbb{R})\)solution of equation \eqref{q} has a fixed sign. 	Thus,   we obtain  from  \cite{Wei07} that up to translation and sign,  $Q_p>0$ is the unique nontrivial real $H^1(\R)$ solution of \eqref{q}.

We also have
\[
\ker L_{+,p}=\operatorname{span}\{Q_p'\},\ \
\ker L_{-,p}=\operatorname{span}\{Q_p\},
\]
where
\begin{equation}\label{l+-}
L_{+,p}:=-\partial_x^2+\kappa_p^2-(2p-1)Q_p^{2p-2},
\ \
L_{-,p}:=-\partial_x^2+\kappa_p^2-Q_p^{2p-2}.
\end{equation}
Since $\|Q_p\|_2=1$, the explicit profile \eqref{eq:scalar-profile-definition} implies that $\kappa_p\in C^\infty(I_{1/8})$ and $\kappa_2=1/4$. Hence, for any $p\in I_{1/8}:=[\frac{15}{8},  \frac{17}{8}]$,
\[
0<\kappa_-:=\inf_{p\in I_{1/8}}\kappa_p\leq\kappa_p\leq \kappa_+:=\sup_{p\in I_{1/8}}\kappa_p<\infty.
\]
For any distinct center $z=(z_1,\cdots,z_N)\in \R^N$, we denote
\begin{equation}\label{5:A5}
\begin{aligned}
Q_{jp}(x)&:=Q_p(x-z_j),&
d_{ij}&:=|z_i-z_j|,&
D:=D(z)&:=\min_{i<j}d_{ij}.
\end{aligned}
\end{equation}
From now on, \(C>0\) denotes a generic constant and \(P(r)\) denotes a generic polynomial bound of the form \(C(1+r)^M\) with \(M\ge0\). They may depend on \(N\) and vary from line to line, but are uniform in \(p\) and \(z\) in the stated ranges.
Recall that
\begin{equation}\label{eq:s2-1}
	w_*:=Q_2=\frac{\sqrt2}{4}\operatorname{sech}(x/4),
	\ \  \kappa_2=\frac14,\ \  E_2(w_*)=-\frac1{48}.
\end{equation}
Define for $p\in I_{1/8}$ and $ z=(z_1, \cdots, z_N)\in\R^N$,
\begin{align}\label{wa}
	W_{p,z}:=\bigl(Q_{1p}(x),\cdots,Q_{Np}(x)\bigr)^\top,
\end{align}
where $Q_{jp}>0$ is as in (\ref{5:A5}).
In the following we also denote \(O(N)\) to be the set of real \(N\times N\) orthogonal matrices.

\begin{lemma}\label{prop:profile}
	Let \(U_n=(u_{1n}, \cdots, u_{Nn})\in\cM_N\) be a  minimizer of  \(J_N(p_n)\), where $p_n\in I_{1/8}:=[\frac{15}{8},  \frac{17}{8}]$ satisfies  \(\lim\limits_{n\to\infty}p_n=2\). Then up to a subsequence if necessary, there exist \(R_n^0\in O(N)\) and
\(z_n^0=(z_{1n}^0,\cdots,z_{Nn}^0)\in\R^N\) such that
\begin{align}\label{3.2}
& (R_n^0)^\top U_n=W_{p_n,z_n^0}+o (1)\ \ \text{strongly in}\ \, H^1(\R; \R^N) \ \, \text{as}\ \, n\to\infty,\\
&\qquad z_{1n}^0<\cdots<z_{Nn}^0,\ \
D(z_n^0)\to\infty\ \ \text{as}\ \ n\to\infty.\label{3.1}\end{align}
Moreover, the corresponding Lagrange multiplier matrix $\Lambda_n$ of $U_n$ is symmetric and satisfies
$$\lim\limits_{n\to\infty}\Big\|\Lambda_n+\frac1{16}I_N\Big\|_{\mathrm F}=0.
$$
\end{lemma}

\begin{proof}
Recall from  \eqref{eq:matrix-euler-lagrange} that  the corresponding Lagrange multiplier matrix $\Lambda_n$ of $U_n$ is symmetric. Suppose that \(R_n\in O(N)\) satisfies
$$\tilde \Lambda_n:=R^\top_n \Lambda_nR_n=\text{diag}(\mu_{1n}, \cdots, \mu_{Nn}) \ \  \mbox{for some}\ \, \mu_{1n}\leq\cdots\leq\mu_{Nn}.$$
Setting  $\tilde{U}_n:=R^\top_nU_n$,  it can be checked that $\tilde{U}_n$ is also a minimizer of $J_N(p_n)$, and
\begin{equation}\label{3.5}
	-\tilde U_n''-|\tilde U_n|^{2p_n-2}\tilde U_n=\tilde \Lambda_n\tilde U_n\ \,\ \text{in}\ \, \R,\ \
	\tilde \Lambda_n=\operatorname{diag}(\mu_{1n},\ldots,\mu_{Nn}),\ \   \mu_{1n}\le\cdots\le\mu_{Nn}.
\end{equation}
Thus, for convenience we next   write \(U_n\) and \(\Lambda_n\) for \(\widetilde U_n\) and \(\widetilde{\Lambda}_n\), respectively, so that \(U_n\) and \(\Lambda_n\) satisfy (\ref{3.5}).
We shall carry out the proof by the following three steps.

\medskip\noindent\textbf{Step 1.} In this step, we prove that
\begin{equation}\label{claim1}
	J_N(p_n)\to- N/{48},\ \
	E_2(U_n)\to-N/{48}\ \,\ \text{as}\ \ n\to\infty.
\end{equation}

Indeed,  since $Q_{p_n}$ is a minimizer of $J_1(p_n)$ and \(Q_{p_n}\to w_*\) in \(H^1(\R)\) as $n\to\infty$,  by the  explicit formula \eqref{eq:scalar-profile-definition} of  $Q_{p_n}$, one can  calculate that 
\begin{equation}\label{eq:s2-8}
	J_N(p_n)\le NJ_1(p_n),\ \
	\lim\limits_{n\to\infty}J_1(p_n)= J_1(2)=-1/{48},
\end{equation}
where  the inequality follows from \cite[Lemma 12 (ii)]{GLN}. Since \(\|U_n\|_2^2=N\),  it can be verified that
\begin{equation}\label{2.7}
	\|U_n\|_\infty^2\le2\int_\R|U_n||U_n'|dx
	\le2\sqrt N\|U_n'\|_2.
\end{equation}
Note from \eqref{eq:s2-8}  that  \(E_{p_n}(U_n)\le0\) holds for  sufficiently large \(n\). This yields  from \eqref{2.7} that for  sufficiently large \(n\),
\[
\|U_n'\|_2^2\le\frac1{p_n}\int_\R|U_n|^{2p_n}
\le\frac N{p_n}\|U_n\|_\infty^{2p_n-2}
\le C\|U_n'\|_2^{p_n-1},
\]
and hence \(\|U_n'\|_2\le C^{\frac{1}{3-p_n}}\leq C^2\), due to the fact $p_n\in I_{1/8}$, where $C>1$ is independent of $n$.   By the embedding  $H^1(\R;\R^N)\hookrightarrow L^\infty(\R;\R^N)$,
we obtain that  \(\{U_n\}\) is bounded uniformly in
\(H^1(\R;\R^N)\) and $L^\infty(\R;\R^N)$. 
Thus, since
\begin{equation*}
	t^{2+\alpha}|\log t|\le C(t^{3/2}+t^{5/2}),\ \ \forall\ t>0,\ |\alpha|\le1/4,
\end{equation*}
applying Taylor's expansion we deduce from (1.4) that
\begin{align}\label{2.9}
	E_2(U_n)=J_N(p_n)+O(|p_n-2|)\ \ \text{as}\ \ n\to\infty.
\end{align}
As a consequence of  \eqref{48}, \eqref{eq:s2-8}  and \eqref{2.9},  we have
\[\begin{split}
J_N(2)\le \liminf\limits_{n\to\infty}E_2(U_n)&=\liminf\limits_{n\to\infty}J_N(p_n)\le \limsup\limits_{n\to\infty}J_N(p_n)
\\
&\le \lim\limits_{n\to\infty} NJ_1(p_n)=-N/48= J_N(2),
\end{split}
\]
which therefore proves \eqref{claim1}.

\medskip\noindent\textbf{Step 2.} Let $\mu_{1n}\le\cdots\le\mu_{Nn}$ be given by \eqref{3.5}. This step is to prove that   up to a subsequence if necessary,
\begin{equation}\label{mu}
	\lim\limits_{n\to\infty}\mu_{in}=\mu_i=-1/16,\ \ i=1, \cdots, N.
\end{equation}

Actually, we first note from  \eqref{3.5} and \cite[Equation (43)]{GLN} that
\[
\mu_{in}=\|u_{in}'\|_2^2-
\int_\R|U_n|^{2p_n-2}u_{in}^2dx
\ge-\|U_n\|_\infty^{2p_n-2}\ge-C,\ \ i=1, \cdots, N,
\]
and
\begin{equation}\label{eq:profile-negative-multipliers}
	\mu_{in}\le J_N(p_n)-J_{N-1}(p_n)\le J_1(p_n),\ \ i=1, \cdots, N,
\end{equation}
where  the last inequality follows from \cite[Lemma 12 (ii)]{GLN}.
This shows from \eqref{eq:s2-8} that there exist $\mu_1, \cdots, \mu_N\in \R$ such that up to a subsequence if necessary,
\[
\lim\limits_{n\to\infty}\mu_{in}=\mu_i,\ \  \mu_1\le\cdots\le\mu_N\le J_1(2)<0.
\]
Since it follows from  \cite[Equation (44)]{GLN} that $\sum_{i=1}^N\mu_{in}=
\frac{p_n+1}{3-p_n}J_N(p_n)$,  we obtain from  \eqref{claim1} that
$
\sum_{i=1}^N\mu_i=-\frac N{16}.
$
This further implies that  \(\mu_1\le-1/16\).

To complete the proof of \eqref{mu},
since $\mu_1\le-1/16$, $\mu_1\le\cdots\le\mu_N$ and $\sum_{i=1}^N\mu_i=-\frac N{16}$,   it suffices to show that \(\mu_1\ge-1/16\).
By  the uniform boundedness of the sequence $\{U_n\}$ in $H^1(\R; \R^N)$, we derive from \eqref{3.5} that 
\[\begin{split}
1/16+o(1)\le-\mu_{1n}&\le \int_\R|U_n|^{2p_n-2}u_{1n}^2dx\\
&\le\|u_{1n}\|_\infty
\bigl\||U_n|^{2p_n-2}\bigr\|_2\|u_{1n}\|_2
\le C\|u_{1n}\|_\infty\ \ \text{as}\ \ n\to\infty,
\end{split}
\]
which yields that \(\|u_{1n}\|_\infty>c>0\) holds for any sufficiently
large $n$.  Choose \(x_n\in\R\) such that \(|u_{1n}(x_n)|\ge c>0\) holds for sufficiently
large $n$.  We then deduce that  up to a subsequence if necessary, 
\begin{equation}\label{2.10a}
\begin{split}
	U_n(\cdot+x_n)\rightharpoonup V\quad\text{in}\ \, H^1(\R; \R^N)\ \  \text{as}\ \ n\to\infty,\\
	-V''-|V|^2V=\operatorname{diag}(\mu_1,\cdots,\mu_N)V\  \ \text{in}\ \, \R,\,\
\end{split}
\end{equation}
where $V=(v_1, \cdots, v_N)\in H^1(\R; \R^N)$ satisfies  $v_1\not\equiv0$.
Denote  $ \widetilde V:=V/\sqrt2$, so that
\begin{align}\label{2.13}
	-\widetilde V''-2|\widetilde V|^2\widetilde V
	=\operatorname{diag}(\mu_1,\cdots,\mu_N)\widetilde V\ \ \ \text{in}\ \ \R.
\end{align}
For every \(\lambda\in\{\mu_1,\ldots,\mu_N\}\), let $\mathcal{I}_\lambda:=\{i:\, \mu_i=\lambda\}.$ Since \(\mu_i<0\) holds for all $i=1, \cdots, N$, it follows from  \eqref{2.13} and Theorem \ref{thm:completeness-multiple-root} that
\[
\dim\operatorname{span}\{v_i:\, i\in \mathcal{I}_\lambda\}\in\{0,1\}.
\]
In particular, we have $\dim\operatorname{span}\{v_i:\, i\in \mathcal{I}_{\mu_1}\}=1$, because  \(v_1\not\equiv0\).

Furthermore,
set \(\phi:=v_1/\|v_1\|_2\) and
\(w_{in}:=u_{in}(\cdot+x_n)\).  Since \(\|\phi\|_2=1\) and
\((w_{in})_{i=1}^N\) is orthonormal,  we have
\begin{align}\label{2.14}
\begin{split}
	0&\le
	\Big\|\phi-\sum_{i\in I_{\mu_1}}\langle\phi,w_{in}\rangle w_{in}\Big\|_2^2\\
&	=1-\sum_{i\in I_{\mu_1}}|\langle\phi,w_{in}\rangle|^2=1-\sum_{i\in I_{\mu_1}}|\langle\phi,\, v_{i}\rangle|^2+o(1)\ \ \text{as}\ \ n\to\infty,
\end{split}
\end{align}
where the second identity follows from the finiteness of \(\mathcal{I}_{\mu_1}\) and the weak convergence of $\{w_{in}\}_n$ in $H^1(\R)$.  Note that $1\in \mathcal{I}_{\mu_1}$ and $\dim\operatorname{span}\{v_i:\, i\in \mathcal{I}_{\mu_1}\}=1$. 
This implies that \(|\langle\phi,v_i\rangle|^2=\|v_i\|_2^2\). Combining  this  with  Theorem \ref{thm:L2iden}, we then conclude from \eqref{2.14} that
\[
4\sqrt{-\mu_1}=\sum_{i\in I_{\mu_1}}\|v_i\|_2^2\le1,
\]
and thus  \(\mu_1\ge-1/16\). Therefore, this completes the proof of Step 2.

\medskip\noindent\textbf{Step 3.}  In this step, we prove that  up to a subsequence if necessary, there exist an integer \(0<K\le N\), the vectors
\(a^{(1)},\cdots,a^{(K)}\in \R^N\), and a sequence
\(z_n^0=(z_{1n}^0,\cdots,z_{Kn}^0)\in\R^K\) such that  the vector $F_n^{(K)}
:=U_n-\sum_{k=1}^Ka^{(k)}w_*(\cdot-z_{kn}^0)$ satisfies
\begin{align}\label{2.12}
	\arraycolsep=1.5pt
	\left\{\begin{array}{lll}
		&F_n^{(K)}(\cdot+z_{kn}^0)\rightharpoonup0\
		\ \ \text{in }\,\ H^1(\R; \R^N)\ \ \text{as}\,\ n\to\infty,\ \ 1\le k\le K,\\[2.5mm]
		& \lim\limits_{n\to\infty}\|F_n^{(K)}\|_4=0,\\
		&\lim\limits_{n\to\infty}|z_{kn}^0-z_{\ell n}^0|=\infty, \ \ 1\le k\ne\ell\le K.
	\end{array}\right.
\end{align}

Set $V^{(1)}:=V$ and $x_n^{(1)}:=x_n,$
where \(V\) and \(x_n\) are as in \eqref{2.10a}.  It then follows that
\begin{align}\label{2.15a}
	&U_n(\cdot+x_n^{(1)})\rightharpoonup V^{(1)}\not\equiv0\quad\text{in }\ H^1(\R; \R^N)\ \  \text{as}\ \ n\to\infty.
\end{align}
Define  $F_n^{(1)}
:=
U_n-V^{(1)}(\cdot-x_n^{(1)}).
$
We then obtain from \eqref{2.15a} that
$$
F_n^{(1)}(\cdot+x_n^{(1)})\rightharpoonup 0\ \ \text{in}\ \ H^1(\R; \R^N)\ \ \text{as}\ \ n\to\infty.
$$
Moreover, since \begin{align}\label{2.17}
-(V^{(1)})''-|V^{(1)}|^2V^{(1)}
=-\frac1{16}V^{(1)}\ \,\ \text{in}\ \, \R,
\end{align}
applying  Theorems \ref{thm:completeness-multiple-root} and \ref{thm:L2iden}, and the uniqueness   of   solutions for equation \eqref{q},  we deduce from \eqref{2.17} that
\begin{align}\label{2.18}
\|V^{(1)}\|_2^2=1,\ \
V^{(1)}=a^{(1)}w_*(\cdot-y^{(1)}),\ \
|a^{(1)}|=1,\ \
E_2(V^{(1)})=-1/48.
\end{align}

If \(\liminf\limits_{n\to\infty}\|F_n^{(1)}\|_4=0\), then
\eqref{2.12}  holds with \(K=1\).
Otherwise, we proceed the proof of \eqref{2.12} by induction. Suppose that  there exists some \(1\le m\in\mathbb{N}\) such that  the first and third lines of \eqref{2.12}  hold with \(K=m\),  where  $$F_n^{(m)}:=U_n-\sum_{k=1}^mV^{(k)}(\cdot-x_n^{(k)}), \ \  V^{(k)}=a^{(k)}w_*(\cdot-y^{(k)}),\ \ |a^{(k)}|=1.$$
If \(\liminf\limits_{n\to\infty}\|F_n^{(m)}\|_4=0\), then
\eqref{2.12} holds with \(K=m\).
Otherwise, after passing to a subsequence, there exists \(c_0>0\) such that \(\lim\limits_{n\to\infty}\|F_n^{(m)}\|_4\ge c_0\). By the uniform boundedness of  \(\{F_n^{(m)}\}_n\)  in \(H^1(\R; \R^N)\),    the vanishing lemma
\cite[Lemma~1.21]{minimax} then yields 
that  there exist \(R>0\) and
\(\{x_n^{(m+1)}\}\subset\R\) such that up to a subsequence if necessary,
\begin{align}\label{2.19a}
\arraycolsep=1.5pt
\left\{\begin{array}{lll}
	\lim\limits_{n\to\infty}\int_{-R}^R
	|F_n^{(m)}(x+x_n^{(m+1)})|^2\,dx>0,\\[3mm]
	F_n^{(m)}(\cdot+x_n^{(m+1)})
	\rightharpoonup V^{(m+1)}\not\equiv0
	\quad\text{in }\ H^1(\R; \R^N)\ \ \text{as}\ \ n\to\infty.
\end{array}\right.
\end{align}
Since  
$F_n^{(m)}(\cdot+x_n^{(k)})\rightharpoonup0$ in $H^1(\R; \R^N)$ as $n\to\infty$, where $ 1\le k\le m$,
by  the compact embedding   $H^1(B_r)\hookrightarrow L^2(B_r)$ for any $r>0$,  one can deduce from \eqref{2.19a} that
$\lim\limits_{n\to\infty}|x_n^{(m+1)}-x_n^{(k)}|=\infty$ for $1\le k\le m$.
The third line of \eqref{2.12} then yields that
\begin{align}\label{2.20}
\lim\limits_{n\to\infty}|x_n^{(k)}-x_n^{(\ell)}|=\infty, \ \ 1\le k\ne\ell\le m+1.
\end{align}
Define
$$
F_n^{(m+1)}
:=F_n^{(m)}-V^{(m+1)}(\cdot-x_n^{(m+1)})
=U_n-\sum_{k=1}^{m+1}V^{(k)}(\cdot-x_n^{(k)}).
$$
By the  weak convergence  of \eqref{2.19a}, it thus follows from  \eqref{2.20} that
\begin{align}\label{2.15b}
-(V^{(m+1)})''-|V^{(m+1)}|^2V^{(m+1)}
=-\frac1{16}V^{(m+1)}\ \,\ \text{in}\ \ \R,
\end{align}
\[
F_n^{(m+1)}(\cdot+x_n^{(k)})\rightharpoonup0
\ \ \ \text{in}\ \, H^1(\R; \R^N) \ \ \, \text{as}\ \ n\to\infty,\ \  1\le k\le m+1.
\]
The same argument of \eqref{2.18} gives that
$V^{(m+1)}=a^{(m+1)}w_*(\cdot-y^{(m+1)})$ holds for some $a^{(m+1)}\in\R^N$ satisfying $|a^{(m+1)}|=1$.
Denote $z_{kn}^0:=y^{(k)}+x_n^{(k)}$ for $k=1, \cdots, K$. This completes the induction.

Since $\|w_*\|_2=1$, $|a^{(k)}|=1$ and
$F_n^{(K)}=U_n-\sum_{k=1}^Ka^{(k)}w_*(\cdot-z_{kn}^0)$,
  the first  and third lines of \eqref{2.12} yield that
\begin{align}\label{2.24}
\arraycolsep=1.5pt
\left\{\begin{array}{lll}
	\!\!\!&N=K+\|F_n^{(K)}\|_2^2+o(1)
	\quad\text{as }\ n\to\infty,\\[2.5mm]
	\!\!\!&\|U_n'\|_2^2
	=K\|w_*'\|_2^2+\|(F_n^{(K)})'\|_2^2+o(1)
	\quad\text{as }\ n\to\infty,
\end{array}\right.
\end{align}
which implies  that  \(K\le N\).
Thus, the construction stops at some \(1\le K\le N\). This thus gives \eqref{2.12}, and Step 3 is proved.

We are now ready to complete the proof of Lemma \ref{prop:profile}.  Since $|a^{(k)}|=1$,    the second and third  lines of \eqref{2.12} give that  $\lim\limits_{n\to\infty}\int_\R|U_n|^4dx
=K\int_\R|w_*|^4dx$.
Since $E_2(w_*)=-1/48$, we obtain from  \eqref{2.24}  that
\begin{align}\label{2.22}
E_2(U_n)=-\frac K{48}+\|(F_n^{(K)})'\|_2^2+o(1),\ \
N=K+\|F_n^{(K)}\|_2^2+o(1)\ \ \text{as}\ \ n\to\infty.
\end{align}
Since $K\le N$, the first identity of \eqref{2.22} gives from \eqref{claim1} that  \(K=N\), which further implies from \eqref{2.22} that
\begin{align}\label{2.23}
F_n^{(N)}\to0\ \ \text{ in}\ \ H^1(\R; \R^N)\ \, \text{as}\ \, n\to\infty.
\end{align}
Moreover, note that
$$(u_{1n}, \cdots, u_{Nn})=U_n=\sum_{k=1}^Na^{(k)}w_*(\cdot-z_{kn}^0)
+F_n^{(N)}.$$
Applying the third line of \eqref{2.12} and $\inte u_{in}u_{jn}dx=\delta_{ij}$,   we then obtain from \eqref{2.23} that $$\sum_{k=1}^N(a^{(k)})_i(a^{(k)})_j=\delta_{ij},$$
which yields that
\begin{equation}\label{2.23a}
\begin{split}
R^\top U_n=e_1w_*(\cdot-z_{1n}^0)+\cdots+e_Nw_*(\cdot-z_{Nn}^0)
+R^\top F_n^{(N)}, \\
R:= (a^{(1)},\cdots,a^{(N)})\in O(N).\qquad\qquad\qquad
\end{split}
\end{equation}
Here $e_j\in\R^N$ is the $j$-th standard basis vector, $1\le j\le N$.
Moreover,  because $R\in O(N)$, we have $\|R^\top F_n^{(N)}\|_{H^1}=\|F_n^{(N)}\|_{H^1}\to0$  as $n\to\infty$. Since \(N\) is finite,  we may assume that up to a  subsequence and
relabeling  if necessary, $z_{1n}^0<\cdots<z_{Nn}^0$ holds for  sufficiently large \(n>0\). This thus proves \eqref{3.1} in view of \eqref{2.12}.
Since \(Q_{p_n}\to w_*\) in \(H^1(\R)\) as $n\to\infty$,  we obtain from \eqref{2.23a} that  \eqref{3.2} holds for the original sequence satisfying $R_n^0:=R_nR$, where $R_n$ is the orthogonal matrix chosen at the beginning of the proof.
We finally have for every \(R\in O(N)\),
\[
\Big\|R^\top\Lambda_n R+\frac1{16}I_N\Big\|_{\mathrm F}
=\Big\|\Lambda_n+\frac1{16}I_N\Big\|_{\mathrm F}\to 0\ \,\ \mbox{as}\ \ n\to\infty,
\]
due to Step 2. This therefore completes the proof of Lemma  \ref{prop:profile}.  \end{proof}

\subsection{Adjustment of rotation and translation parameters}\label{sec:local-energy}

Recall from Lemma~\ref{prop:profile} that there exist \(z_n^0\) and the rotation \(R_n^0\) such that
\begin{equation}\label{eq:oct-1}
(R_n^0)^\top U_n=W_{p_n,z_n^0}+o(1)\quad \text{strongly in}\ \, H^1(\R; \R^N)\ \ \hbox{as}\ \ n\to\infty,
\end{equation}
where $p_n\to 2$ as $n\to\infty$.
This qualitative approximation is however not sufficient for the analysis below. The main purpose of this subsection  is to establish Lemma~\ref{lem:center}, where the rotation and translation parameters are further adjusted: more precisely, \(R_n\) and \(z_n\) are chosen close to \(R_n^0\) and \(z_n^0\), respectively, so that \(R_n^\top U_n\) can be represented by  a correction \(K\) satisfying suitable orthogonality conditions.

For any  vector $P=(P_1,\cdots,P_N)^\top$
and $Q=(Q_1,\cdots,Q_N)^\top$ in $H^1(\R;\R^N)$, we denote the pairing matrix of $P$ and $Q$ as
\begin{align}
\label{5:2:A}
\mathsf M(P, Q):=\big(\langle P_i,Q_j\rangle_{L^2}\big)_{i,j=1}^N,\ \ \mathsf M(P):=M(P,P).
\end{align}
Let \(W_{p,z}\) be as in \eqref{wa}, where \((p,z)\in I_{1/8}\times\R^N\) is chosen so that
\(D(z)>0\) is sufficiently large. Define
\begin{equation}\label{w}
	\begin{aligned}
		G_{p,z}:=\mathsf M(W_{p,z}),\ \ \Psi_{p,z}:=G_{p,z}^{-1/2}W_{p,z}\in\cM_N,
	\end{aligned}
\end{equation}
so that $\mathsf M(\Psi_{p,z})=I_N$. Note from
\eqref{a.2} and \eqref{a.7} in the
Appendix that  for any \(p\in I_{1/8}:=[\frac{15}{8}, \frac{17}{8}]\),
\begin{align}\label{a.7b}
	\|G_{p,z}-I_N\|_{\mathrm{F}}+
	\|G_{p,z}^{-1/2}-I_N\|_{\mathrm{F}}+\|\partial_{z_j}(G_{p,z}^{-1/2})\|_{\mathrm{F}}\leq CD(z)e^{-\kappa_pD(z)} 
\end{align}
uniformly as \(D(z)\to\infty\), which yields that   \eqref{w} is well defined. Moreover, we derive from \eqref{eq:oct-1}--\eqref{a.7b} that
\begin{equation}\label{eq:oct-2}
	(R_n^0)^\top U_n=\Psi_{p_n,z_n^0}+o(1)\,\ \  \text{strongly in}\ \, H^1(\R; \R^N)\ \ \hbox{as}\ \ n\to\infty,
\end{equation}
where $p_n\to 2$ as $n\to\infty$.

Denote for fixed $z$ and $U\in\cM_N$,
\[A=A(z,U):=\mathsf M(U,\Psi_{p,z}).\]
Since
$A-I_N=\mathsf M(U-\Psi_{p,z},\Psi_{p,z})$, taking $U=(R_n^0)^\top U_n$ and $z=z_n^0$,
it yields from \eqref{eq:oct-2} that
\[
\|A(z_n^0,(R_n^0)^\top U_n)-I_N\|_{\mathrm{F}}=o(1)\ \ \hbox{as}\ \ n\to\infty.
\]
Hence, \(A\bigl(z_n^0,(R_n^0)^\top U_n\bigr)\) is invertible for all sufficiently large \(n>0\), and
\(A\bigl(z_n^0,(R_n^0)^\top U_n\bigr)\) admits the following polar decomposition
\begin{equation}\label{eq:oct1}
	A=R_AS_A,\ \ \hbox{where}\,\ R_A:=A(A^\top A)^{-1/2}\in O(N),\,\ \hbox{and}\,\ S_A=(A^\top A)^{\frac12}>0.
\end{equation}
Denote
\begin{equation}\label{2.28}
	K(z,U):=A^{-1}U-\Psi_{p,z}.
\end{equation}
We then obtain from \eqref{2.28} that
\[
\mathsf M(K,\Psi_{p,z})=A^{-1}A-I_N=0,\ \ \,
I_N+\mathsf M(K)=\mathsf M(A^{-1}U)=(A^\top A)^{-1}.
\]
Since $R_A^\top A=(A^\top A)^{1/2}$ and
$U=A(\Psi_{p,z}+K)$,  we deduce from \eqref{eq:oct1} and \eqref{2.28} that $U$ admits the following exact representation
\begin{equation}\label{2.26}
	R_A^\top U=\cR_{\Psi_{p,z}}(K):=[I_N+\mathsf M(K)]^{-1/2}(\Psi_{p,z}+K)\in\cM_N,\ \ \mathsf M(K,\Psi_{p,z})=0.
\end{equation}

For each fixed $z$, the rotation $R$ and the correction $K$ are fixed by \eqref{eq:oct1} and \eqref{2.28}.
It remains to choose the centers \(z=(z_1,\cdots,z_N)\) near \(z^0\). Towards this aim, we now define
\begin{equation}\label{tder}
	T_j(p,z):=\partial_{z_j}\Psi_{p,z}
	-\mathsf M(\partial_{z_j}\Psi_{p,z},\Psi_{p,z})\Psi_{p,z},
	\ \ j=1,\cdots,N,
\end{equation}
so that $\mathsf M(T_j,\Psi_{p,z})=0$.
In the Appendix, we prove that there exist a
sufficiently large $D_*>0$ and a constant $C>0$ such that
\begin{align}\label{eq:reference-center-derivatives}
	&\ \ \ \ \ \ \ \ \  \ \sup_{\substack{p\in I_{1/8},\ D(z)\ge D_*}}
	\sum_{|\alpha|\le2}\|\partial_z^\alpha\Psi_{p,z}\|_{H^1}
	\le C,\\
	&\|\Psi_{p,z}- W_{p,z}\|_{H^1}+\|T_j(p, z)+e_j Q'_p(x-z_j)\|_{H^1}\leq CD(z)e^{-\kappa_pD(z)}\label{2.35},
\end{align}
where the vector $e_jQ_p'(x-z_j)$ has the $j$-th component $Q_p'(x-z_j)$,
and all other components are zero.  Thus, $T_j(p,z)$ can be regarded, up to an exponentially small error, as the translation derivative associated with the center of the $j$-th bubble.
Define for any $p\in I_{1/8}$ and $z\in\R^N$ satisfying $D(z)\ge D_*$,
\begin{equation}\label{xpz}
	X_{p,z}:=\Big\{K\in H^1(\R;\R^N):\
	\mathsf M(K,\Psi_{p,z})=0,\
	\langle K,T_j(p,z)\rangle=0,\  j=1, \cdots, N\Big\}.
\end{equation}
The following lemma shows that if \(U\in\cM_N\) is sufficiently close to \(\Psi_{p,z^0}\) in $H^1(\R; \R^N)$, then  one can choose \(z\) sufficiently  close to \(z^0\), so that under the corresponding rotation \(R_A(z,U)\), the correction \(K=K(z,U)\) satisfies \(K\in X_{p,z}\).

\begin{lemma}\label{lem:center} Suppose that $\varepsilon_0>0$ is sufficiently small,  and $D_0> D_*$ is sufficiently large. Then for any $p\in I_{1/8},$  $z^0\in \R^N\ \text{satisfying}\ D(z^0)\ge D_0$ and  $U\in\cM_N\ \text{satisfying}\ \|U-\Psi_{p,z^0}\|_{H^1}<\varepsilon_0,$
there exists  $(z, K, R)$  such that
\begin{equation}\label{2.33}
	K\in X_{p,z},\ \ R\in O(N),\ \  R^\top U=\cR_{\Psi_{p,z}}(K),
\end{equation}	
and
\begin{equation}\label{2.38b}
	|z-z^0|+\|K\|_{H^1}\le C\|U-\Psi_{p,z^0}\|_{H^1}.
\end{equation}
\end{lemma}

\begin{proof}
Choose sufficiently large \(D_0\ge D_*+2\)  and let $z^0\in\R^N$ satisfy  $D(z^0)\ge D_0$, so that \(D(z)\ge D_*\) holds for any $z\in\R^N$ satisfying \(|z-z^0|\le1\).
We  claim that there exist \(0<r_0<1\) and \(\varepsilon_0>0\),  independent of \(p\in I_{1/8}\) and \(z^0\),  such that 
\begin{align}\label{2.35a}
\|A(z,U)-I_N\|_{\mathrm F}<1/2,\ \ \forall \  (z,U)\in B_{r_0}(z^0)\times B_{\varepsilon_0}(\Psi_{p,z^0}),
\end{align}
where $A(z,U):=\mathsf M(U,\Psi_{p,z})$ and
$$B_{r_0}(z^0)\times B_{\varepsilon_0}(\Psi_{p,z^0}):=\big\{(z, U)\in \mathbb R^N\times \cM_N:\ |z-z^0|\leq r_0,\ \|U-\Psi_{p,z^0}\|_{H^1}\leq\varepsilon_0\big\}.$$
Indeed,  the derivative estimate \eqref{eq:reference-center-derivatives} gives that
\begin{align}\label{2.36}
\|\Psi_{p,z}-\Psi_{p,z^0}\|_{H^1}\le C|z-z^0|,
\end{align} which then yields from (\ref{5:2:A}) that
\begin{equation}\label{2.37}
\begin{split}
\|A(z,U)-I_N\|_{\mathrm F}=&\big\|\mathsf M(U-\Psi_{p,z^0},\Psi_{p,z})
+\mathsf M(\Psi_{p,z^0},\Psi_{p,z}-\Psi_{p,z^0})\big\|_{\mathrm{F}}\\
\le& C\bigl(\|U-\Psi_{p,z^0}\|_{H^1}+|z-z^0|\bigr). \end{split}
\end{equation}
By choosing  sufficiently small $r_0>0$ and $\varepsilon_0>0$, we then obtain that  the claim \eqref{2.35a} holds true.

We now obtain  from \eqref{2.35a} that  \(A(z,U)\) is invertible and
\(\|A(z,U)^{-1}\|_{\mathrm F}\le C\) holds for  $(z, U)\in B_{r_0}(z^0)\times B_{\varepsilon_0}(\Psi_{p,z^0})$. Let \(K(z,U)\) be as in \eqref{2.28}, and set
\(R(z,U):=R_A(z,U)\), where \(R_A\) is given by
\eqref{eq:oct1}.
Thus, we note from \eqref{2.26} that in order  to prove \eqref{2.33}, it next suffices to solve
$F(z,U)=0$ in  $B_{r_0}(z^0)$, where $U\in B_{\varepsilon_0}(\Psi_{p,z^0})$ is fixed,  and 
\begin{equation}\label{55:z0}
F:=(F_1,\cdots,F_N),\ \  F_i(z,U):=\big\langle K(z,U),\, T_i(p,z)\big\rangle.
\end{equation}
To solve \(F(z,U)=0\), we first claim that
\begin{equation}\label{z0}
\left.D_zF(z,\Psi_{p,z^0})\right|_{\substack{z=z^0}}=-\Gamma(p,z^0),
\end{equation}
where $\Gamma(p,z):=\big(\langle T_i(p,z),T_j(p,z)\rangle\big)_{i,j=1}^N$, and  $T_i(p,z)$ is given by \eqref{tder}. In fact, differentiating \(\mathsf M(\Psi_{p,z})=I_N\) yields that
\[
\mathsf M(\Psi_{p,z},\partial_{z_j}\Psi_{p,z})
=-\mathsf M(\partial_{z_j}\Psi_{p,z},\Psi_{p,z}).
\]
Since $A(z^0, \Psi_{p, z^0})=I_N$	and \(\partial_{z_j}(A^{-1})
=-A^{-1}(\partial_{z_j}A)A^{-1},\)
one gets that
\begin{equation}\label{eq:center-coordinate-tangent}
\begin{aligned}
	\left.\partial_{z_j}K(z,\Psi_{p,z^0})\right|_{z=z^0}
	=&\left.\partial_{z_j}\big[A(z,\Psi_{p,z^0})^{-1}\Psi_{p,z^0}-\Psi_{p,z}\big]\right|_{z=z^0}\\
	=&-\mathsf M(\Psi_{p,z^0},\partial_{z_j}\Psi_{p,z^0})\Psi_{p,z^0}
	-\partial_{z_j}\Psi_{p,z^0}=-T_j(p,z^0).
\end{aligned}
\end{equation}
This thus proves  \eqref{z0}, because \(K(z^0, \Psi_{p,z^0}) = 0\).

We next claim that
\begin{equation}\label{step2}
\|D_zF(z,U)+\Gamma(p,z^0)\|_{\mathrm F}
\le C\bigl(\|U-\Psi_{p,z^0}\|_{H^1}+|z-z^0|\bigr).
\end{equation}
Recall that $F_i(z,U)=\big\langle K(z,U),\, T_i(p,z)\big\rangle$,
which implies that
\begin{align}\label{2.41}
\partial_{z_j}F_i(z,U)+\Gamma_{ij}(p,z^0)
=&\big\langle\partial_{z_j}K(z,U)+T_j(p,z^0),\, T_i(p,z)\big\rangle\\
&+\big\langle T_j(p,z^0),\, T_i(p,z^0)-T_i(p,z)\big\rangle+\big\langle K(z,U),\, \partial_{z_j}T_i(p,z)\big\rangle.\nonumber
\end{align}
Similar to \eqref{2.36}, we can follow
\eqref{eq:reference-center-derivatives}  to obtain that
\begin{align}\label{2.42}
\|T_i(p,z)-T_i(p,z^0)\|_{H^1}\le C|z-z^0|\quad\hbox{and}\quad
\|\partial_{z_j}T_i(p,z)\|_{H^1}\le C.
\end{align}
Since
\[
A^{-1}-I_N=A^{-1}(I_N-A), \ \ K(z,U)=A^{-1}(U-\Psi_{p,z^0})
+(A^{-1}-I_N)\Psi_{p,z^0}+\Psi_{p,z^0}-\Psi_{p,z},
\]
we conclude from \eqref{2.36} and \eqref{2.37} that
\begin{equation}\label{2.54}
\|K(z,U)\|_{H^1}
\le C\bigl(\|U-\Psi_{p,z^0}\|_{H^1}+|z-z^0|\bigr).
\end{equation}
This hence  shows from \eqref{2.41} and \eqref{2.42} that
\begin{equation}\label{2.55}
\big|\partial_{z_j}F_i(z,U)+\Gamma_{ij}(p,z^0)\big|
\le C\Big(
\|\partial_{z_j}K(z,U)+T_j(p,z)\|_{H^1}
+\|K(z,U)\|_{H^1}+|z-z^0|
\Big).
\end{equation}
Moreover, differentiating  $K(z,U)=A(z, U)^{-1}U-\Psi_{p,z} $ gives that
\[
\begin{aligned}
\partial_{z_j}K(z,U)
=-A^{-1}\mathsf M(U,\partial_{z_j}\Psi_{p,z})A^{-1}U
-\partial_{z_j}\Psi_{p,z}.
\end{aligned}
\]
By the facts that $A^{-1}\mathsf M(U,\, \partial_{z_j}\Psi_{p,z})
=\mathsf M(A^{-1}U,\, \partial_{z_j}\Psi_{p,z})$ and
$$
A^{-1}U=\Psi_{p,z}+K(z,U),\ \
\mathsf M(\Psi_{p,z},\, \partial_{z_j}\Psi_{p,z})
=-\mathsf M(\partial_{z_j}\Psi_{p,z},\, \Psi_{p,z}),$$
we have
\begin{align*}
&\partial_{z_j}K(z,U)+T_j(p,z)\\
=&-\mathsf M(\Psi_{p,z},\partial_{z_j}\Psi_{p,z})K(z,U)-\mathsf M(K(z,U),\partial_{z_j}\Psi_{p,z})
(\Psi_{p,z}+K(z,U)),
\end{align*}
and thus
\begin{align}\label{2.45}
\big\|\partial_{z_j}K(z,U)+T_j(p,z)\big\|_{H^1}
\leq C\big(\|K(z,U)\|_{H^1}+\|K(z,U)\|_{H^1}^2\big).
\end{align}
As a consequence of \eqref{2.54}--\eqref{2.45},  we conclude that  the claim \eqref{step2} holds true.

Denote \(c_*:=\frac12\inf_{p\in I_{1/8}}\|Q_p'\|_2^2>0\). One can verify from  \eqref{2.35}  that  $$\|\Gamma(p,z)-\|Q_p'\|_2^2I_N\|_{\mathrm F} \le ce^{-\theta D(z)},$$
where $c, \theta>0$ are independent of $(p,z)$.  Increasing $D_0$ if necessary, we may assume that $ce^{-\theta D(z)}\leq c_*$ holds for any $D(z)\ge D_0$. This shows that
$$
\lambda_{\min}(\Gamma(p,z))
\ge \|Q_p'\|_2^2-c_*
\ge c_*,
$$
and hence
\begin{equation}\label{2.47}
\|\Gamma(p,z)^{-1}\|_{\mathrm F}\le \sqrt N/c_*,\ \ \forall\ p\in I_{1/8}, \ z\in\R^N\ \text{with}\ D(z)\ge D_0,
\end{equation}
where $\lambda_{\min}(\Gamma(p,z))$ denotes the smallest eigenvalue of $\Gamma(p,z)$.
Define
$$\mathcal T_U(z):=z+\Gamma(p,z^0)^{-1}F(z,U), \ \ \forall z\in B_{r_0}(z^0).$$
Since
\[
D_z\mathcal T_U(z)=\Gamma(p,z^0)^{-1}
\bigl[D_zF(z,U)+\Gamma(p,z^0)\bigr],
\]
it follows from \eqref{step2} and \eqref{2.47}  that
\[\|D_z\mathcal T_U(z)\|_{\mathrm F}\le C(\varepsilon_0+r_0) \ \ \mbox{on}\,  \ B_{r_0}(z^0).\]
By decreasing $(r_0, \varepsilon_0)$  if necessary, we get that $\|D_z\mathcal T_U(z)\|_{\mathrm F}\le 1/2$. 
This further implies from  \eqref{2.54} and \eqref{2.47}  that
\begin{align}\label{2.59}
|\mathcal T_U(z)-z^0|
&\leq
|\mathcal T_U(z)-\mathcal T_U(z^0)|
+
|\mathcal T_U(z^0)-z^0|\nonumber\\
&=|\mathcal T_U(z)-\mathcal T_U(z^0)|
+|\Gamma(p,z^0)^{-1}F(z^0, U)|
\\
&\leq
\frac12|z-z^0|+C\|U-\Psi_{p, z^0}\|_{H^1}.\nonumber
\end{align}
Choose $\varepsilon_0>0$ sufficiently small so that $C\varepsilon_0<r_0/4$.  We then conclude from above that
\[
|\mathcal T_U(z)-z^0|<r_0,\ \ \forall\ z\in B_{r_0}(z^0).
\]
Therefore,  for fixed $p\in I_{1/8}$, $z^0\in\R^N$ satisfying $D(z^0)\geq D_0$,  and $U\in B_{\varepsilon_0}(\Psi_{p, z^0})$, the contraction-mapping theorem  gives that there exists $z=z(U)\in B_{r_0}(z^0)$ such that $$z+\Gamma(p,z^0)^{-1}F(z,U)=\mathcal T_U(z)=z,$$ which yields that
$F(z(U), U)=0$.  
This therefore proves \eqref{2.33} in view of (\ref{55:z0}).

Finally,  substituting $\mathcal T_U(z(U))=z(U)$ into \eqref{2.59}, we derive that
$$|z(U)-z^0|\leq 2C\|U-\Psi_{p, z^0}\|_{H^1},$$
which then proves \eqref{2.38b} in view of \eqref{2.54}. The proof of Lemma \ref{lem:center} is therefore complete.\end{proof}



\subsection{Coercivity and energy reduction}\label{sec:taylor-reduction}

The purpose of this subsection  is to  prove Lemma~\ref{lem:energy-taylor} on  the energy lower bound of $J_N(p)$ as $p\to 2$. Towards this aim, we define
\begin{equation}\label{eq:linearized-reference-operator}
\mathcal L_{p,z}:=-\partial_{xx}+\kappa_p^2-\mathcal V_p(\Psi_{p,z}),
\end{equation}
where
\begin{equation}\label{v}
\mathcal V_p(\Psi_{p,z})
:=|\Psi_{p,z}|^{2p-2}I_N
+(2p-2)|\Psi_{p,z}|^{2p-4}\Psi_{p,z}\Psi_{p,z}^\top.
\end{equation}
Let  $X_{p,z}$ be given by \eqref{xpz} for $p\in I_{1/8}$ and $z\in\R^N$ satisfying $D(z)\geq D_*$. We first establish  the following uniform coercivity estimate on \(X_{p,z}\).

\begin{lemma}\label{prop:slice-coercivity}
There exist \(0<\delta<1/8\) and \(D_1>D_*\) such that for any
\(p\in I_\delta:=[2-\delta, 2+\delta]\) and  $z\in\R^N$ satisfying \(D(z)\ge D_1\),
\begin{equation}\label{3.4}
	\langle\mathcal L_{p,z}H,H\rangle\ge c_*\|H\|_{H^1}^2,\ \ \forall\  H\in X_{p,z},
\end{equation}
where $c_*>0$ is independent of $(p, z)$.
Consequently, for every \(f\in H^{-1}(\R;\R^N)\), there  exists a unique \(H_{f}\in X_{p,z}\) such that
\begin{equation}\label{eq:s3-13a}
	\|H_{f}\|_{H^1}\le c_*^{-1}\|f\|_{H^{-1}},\ \	\langle\mathcal L_{p,z}H_{f},\ G\rangle=\langle f,\, G\rangle, 	\ \ \forall\,  G\in X_{p,z},
\end{equation}
where $H^{-1}(\R;\R^N)$ denotes the dual space of $H^1(\R;\R^N)$.
\end{lemma}

\begin{proof}
Let \(L_{\pm,p}\) be given by  \eqref{l+-} and write \(q_{\pm,p}(h):=\langle L_{\pm,p}h,h\rangle\).
We first claim that there exist $\delta\in(0, 1/8]$ and $c>0$ such that for any \(p\in I_\delta\),
\begin{align}\label{step1}
& q_{-,p}(h)\ge c\|h\|_{H^1}^2,\ \
\text{if }\ h\perp Q_p,\\
& q_{+,p}(h)\ge c\|h\|_{H^1}^2,\ \
\text{if }\ h\perp \{Q_p,\, Q_p'\}.\label{2.56}
\end{align}
In fact, for \(p=2\),  it follows from  \cite[Lemma 2.2]{MMT2006} that there exists a constant $c_1>0$ such that
\begin{align}\label{2.57}
q_{-,2}(h)\ge c_1\|h\|_{H^1}^2\ \ \text{for any}\  \ h\perp Q_2,\ \  \  \
q_{+,2}(h)\ge c_1\|h\|_{H^1}^2\ \ \text{for any}\ \ h\perp \{Q_2,\, Q_2'\}.
\end{align}
To extend these estimates to a small neighborhood of \(p=2\), we denote \(\widetilde h:=h-\langle h,Q_2\rangle Q_2\) for \(h\perp Q_p\). We then have
\[
\widetilde h\perp Q_2,\ \
\|h-\widetilde h\|_{H^1}
\le\|Q_2\|_{H^1}\|Q_2-Q_p\|_2\|h\|_2
=o(1)\|h\|_{H^1}\ \ \text{as}\ \ p\to2.
\]
Since $\lim\limits_{p\to2}\kappa_p=\kappa_2$ and $ \lim\limits_{p\to2}\|Q_p^{2p-2}-Q_2^2\|_{\infty}=0$, 
we further obtain that
\[
q_{-,p}(h)=q_{-,2}(\widetilde h)+o(1)\|h\|_{H^1}^2\ \ \text{as}\ \ p\to2,
\]
which thus proves the claim \eqref{step1} in view of (\ref{2.57}). Moreover,  set
$$
\widetilde h:=h-\langle h,Q_2\rangle Q_2
-\frac{\langle h,Q_2'\rangle}{\|Q_2'\|_2^2}Q_2'  \ \ \mbox{for}\ \    h\perp \{Q_p,\, Q_p'\} .
$$
Similarly, following the facts that
$$\lim\limits_{p\to2}\kappa_p=\kappa_2,\ \ \lim\limits_{p\to2}\|(2p-1)Q_p^{2p-2}-3Q_2^2\|_{\infty}=0,$$
since \(Q_p\to Q_2\) and \(Q_p'\to Q_2'\) in \(L^2\) as $p\to2$,   we deduce that
$$\|h-\widetilde h\|_{H^1}=o(1)\|h\|_{H^1},\ \ q_{+,p}(h)=q_{+,2}(\widetilde h)+o(1)\|h\|_{H^1}^2 \ \  \mbox{as}\,\  p\to 2.
 $$
 This proves the claim \eqref{2.56} in view of \eqref{2.57}.

Applying \eqref{step1} and \eqref{2.56}, we now claim that  for any $p\in I_\delta$ and $h\in H^1(\R)$,
\begin{equation}\label{eq:scalar-coercivity-errors}
\begin{aligned}
	&q_{-,p}(h)\ge \frac{c}{4}\|h\|_{H^1}^2-C|\langle h,Q_p\rangle|^2, \ \ q_{+,p}(h)\ge \frac{c}{4}\|h\|_{H^1}^2
	-C\bigl(|\langle h,Q_p\rangle|^2+|\langle h,Q_p'\rangle|^2\bigr),
\end{aligned}
\end{equation}
where $c>0$ is as in \eqref{step1} and \eqref{2.56}.  Actually, since  $L_{-,p}Q_p=0$, it follows from  \eqref{step1} that
\[
\begin{aligned}
q_{-,p}(h)=q_{-,p}\bigl(h-\langle h,Q_p\rangle Q_p\bigr)
\ge c\|h-\langle h,Q_p\rangle Q_p\|_{H^1}^2
\ge \frac c2\|h\|_{H^1}^2-C|\langle h,Q_p\rangle|^2.
\end{aligned}
\]
Similarly, since $ L_{+,p}Q_p'=0 $ and $L_{+,p}Q_p=-(2p-2)Q_p^{2p-1}$,
one can  verify from \eqref{2.56} that
\[
\begin{aligned}
q_{+,p}(h)
&\ge q_{+,p}\Big(h-\langle h,Q_p\rangle Q_p-\frac{\langle h,Q_p'\rangle}{\|Q_p'\|_2^2}Q_p'\Big)
-C\big|\langle h,Q_p\rangle\big|\, \|h\|_2\\
&\ge c\Big\| h-\langle h,Q_p\rangle Q_p-\frac{\langle h,Q_p'\rangle}{\|Q_p'\|_2^2}Q_p'\Big\|_{H^1}^2
-C\big|\langle h,Q_p\rangle\big|\, \|h\|_2\\
&\ge \frac c2\|h\|_{H^1}^2 -C\bigl(|\langle h,Q_p\rangle|^2+|\langle h,Q_p'\rangle|^2\bigr)
-C\big|\langle h,Q_p\rangle\big|\, \|h\|_2\\
&\ge\frac c4\|h\|_{H^1}^2-C\big(|\langle h,Q_p\rangle|^2+
|\langle h,Q_p'\rangle|^2\big),
\end{aligned}
\]
where we have used Young's inequality, together with the fact that $\inf_{p\in I_\delta}\|Q_p'\|_2>0$.
These prove the claim \eqref{eq:scalar-coercivity-errors}.

We next estimate $\langle\mathcal L_{p,z}H,H\rangle$ for  \(H=(h_1,\cdots,h_N)^\top\in X_{p,z}\).
Fix  sufficiently large \(R>1\), and assume \(D(z)>4R\).  Choose smooth cutoff functions
\(\chi_1,\cdots,\chi_N\in C_c^\infty(\R;[0,1])\) such that
\(\chi_j(x)\equiv 1\) for \(|x-z_j|\le R\),
\(\chi_j(x)\equiv 0\) for \(|x-z_j|\ge2R\), and
\begin{equation}\label{sep-cutoff}
\sum_{j=0}^N\chi_j^2=1,\ \
\sum_{j=0}^N|\chi_j'|^2\le CR^{-2},
\end{equation}
so that \(\chi_0\in C^\infty(\R;[0,1])\) is an exterior cutoff function. The IMS identity
(cf. \cite[Theorem 3.2]{ims}) then yields from (\ref{sep-cutoff}) that for any $H\in X_{p,z}$,
\begin{equation}\label{2.60a}
\begin{split}
\langle\mathcal L_{p,z}H,H\rangle
&=\sum_{j=0}^N\langle\mathcal L_{p,z}\chi_jH,\chi_jH\rangle
-\int_\R\sum_{j=0}^N|\chi_j'|^2|H|^2\\
&\geq \sum_{j=0}^N\langle\mathcal L_{p,z}\chi_jH,\chi_jH\rangle
-CR^{-2}\|H\|_2^2.\end{split}
\end{equation}
Since $\kappa_-=\inf_{p\in I_{1/8}}\kappa_p>0$ and  \(\|\Psi_{p,z}\|_{L^\infty(\operatorname{supp}\chi_0)}
\le Ce^{-\theta R}\), we can  choose \(R>0\) sufficiently large  so that
$$\sup_{\operatorname{supp}\chi_0}
\|\mathcal V_p(\Psi_{p,z})\|_{\mathrm F}\le Ce^{-\theta R}\le\kappa_-^2/2.$$
This then yields from (\ref{eq:linearized-reference-operator}) that
\begin{equation}\label{3.13b}
\begin{aligned}
	\langle\mathcal L_{p,z}\chi_0H,\chi_0H\rangle
	\ge \min\{1, \kappa_-^2/2\}\|\chi_0H\|_{H^1}^2.
\end{aligned}
\end{equation}
In view of (\ref{2.60a}), it thus remains to estimate $\langle\mathcal L_{p,z}\chi_jH,\chi_jH\rangle$
from below for $j=1, \cdots,N$.

By the exponential decay of $Q_p$,  and the $L^\infty$-uniform boundedness of $\{\Psi_{p,z}\}$ and $\{Q_{p}\}$,  we deduce  from \eqref{2.35}  and (\ref{v})  that
\begin{align*}
\sup_{B_{2R}(z_j)}
\|\mathcal V_p(\Psi_{p,z})-\mathcal V_p(Q_{jp}e_j)\|_{\mathrm F}
\le C\sup_{B_{2R}(z_j)}|\Psi_{p,z}-Q_{jp}e_j|
\le Ce^{-\theta R},\ \ j=1, \cdots,N,
\end{align*}
where $C,\theta>0$ are independent of $(p,z, R)$. 
This then implies that for any $H\in X_{p,z}$,
\begin{equation}\label{2.60}
\begin{split}
&\Big|\langle\mathcal L_{p,z}\chi_jH,\chi_jH\rangle\ -\left\langle
\bigl(-\partial_{xx}+\kappa_p^2-\mathcal V_p(Q_{jp}e_j)\bigr)\chi_jH,
\chi_jH
\right\rangle\Big|\\
\le & Ce^{-\theta R}\|\chi_jH\|_2^2,\ \ j=1, \cdots,N.
\end{split}
\end{equation}
Note from \eqref{eq:scalar-coercivity-errors}  that  for any $ p\in I_\delta$ and  $H=(h_1, \cdots, h_N)^\top\in X_{p,z}$,
\begin{align}\label{2.61}
& \Big\langle
\big(-\partial_{xx}+\kappa_p^2-\mathcal V_p(Q_{jp}e_j)\big)\chi_jH,\,
\chi_jH\Big\rangle\nonumber\\
=&q_{+,p}\bigl((\chi_jh_j)(\cdot+z_j)\bigr)
+\sum_{i\ne j}q_{-,p}\bigl((\chi_jh_i)(\cdot+z_j)\bigr)\\[-1.5mm]
\geq&\frac{c}{4}\sum_{i=1}^N\|\chi_jh_i\|_{H^1}^2-C\sum_{i=1}^N|\langle\chi_jh_i,Q_{jp}\rangle|^2 -C|\langle\chi_jh_j,Q_{jp}'\rangle|^2,\ \ j=1, \cdots,N.\nonumber
\end{align}
Since $H\in X_{p, z}$, we have
\(\langle h_i,(\Psi_{p,z})_j\rangle=0\) and \(\langle H,T_j\rangle=0\) for $i, j=1, \cdots, N$, which yields that
\begin{align*}
\arraycolsep=1.5pt
\left\{\begin{array}{lll}
	\big\langle\chi_jh_i,\, Q_{jp}\big\rangle
	&=\big\langle h_i,\, Q_{jp}-(\Psi_{p,z})_j\big\rangle
	-\big\langle(1-\chi_j)h_i,\, Q_{jp}\big\rangle,\ \ i, j=1, \cdots, N,\\[2mm]
	\big\langle\chi_jh_j,\, Q_{jp}'\big\rangle
	&=\big\langle H,\, e_jQ_{jp}'+T_j\big\rangle
	-\big\langle(1-\chi_j)h_j,\, Q_{jp}'\big\rangle,\ \  j=1, \cdots, N.
\end{array}\right.
\end{align*}
By the exponential  estimate $$\|Q_{jp}\|_{L^2(B_R^c(z_j))} +\|Q_{jp}'\|_{L^2(B_R^c(z_j))} \le Ce^{-\theta R},$$
we further conclude from \eqref{2.35} that for any $ p\in I_\delta$ and  $H=(h_1, \cdots, h_N)^\top\in X_{p,z}$,
\begin{equation}\label{3.7b}
\sum_{i=1}^N|\langle\chi_jh_i,Q_{jp}\rangle|
+|\langle\chi_jh_j,Q_{jp}'\rangle|
\le Ce^{-\theta R}\|H\|_{H^1},  \ \ j=1, \cdots, N.
\end{equation}
This thus shows from \eqref{2.60} and  \eqref{2.61} that  for any $ p\in I_\delta$,
\begin{equation}\label{3.13a}
\langle\mathcal L_{p,z}\chi_jH,\chi_jH\rangle
\ge \frac{c}{4}\|\chi_jH\|_{H^1}^2-Ce^{-\theta R}\|H\|_{H^1}^2,\ \  j=1, \cdots, N.
\end{equation}

Denote $2c_* :=\min\{c/4, 1, \kappa_-^2/2\}>0$. Consequently,  we deduce from  \eqref{2.60a},  \eqref{3.13b} and \eqref{3.13a} that  for any $ p\in I_\delta$,
\begin{equation}\label{2.67}
\begin{split}
\langle\mathcal L_{p,z}H,H\rangle
&\geq 2c_*\sum_{j=0}^N\|\chi_jH\|_{H^1}^2
-C\big(R^{-2}+e^{-\theta R}\big)\|H\|_{H^1}^2\\
&\geq \bigl[2c_*-C(R^{-2}+e^{-\theta R})\bigr]\|H\|_{H^1}^2,
\end{split}
\end{equation}
where the last inequality follows from the fact that
\[
\sum_{j=0}^N\|\chi_jH\|_{H^1}^2
=\|H\|_{H^1}^2+\int_\R\sum_{j=0}^N|\chi_j'|^2|H|^2\geq \|H\|_{H^1}^2.
\]
Choose \(R>0\) sufficiently large  so that   \(C\big(R^{-2}+e^{-\theta R}\big)<c_*\) holds in \eqref{2.67},
and then take \(D_1>4R\).  This therefore proves \eqref{3.4}. Finally, note that  \(X_{p,z}\) is a closed subspace of \(H^1\), and the $L^\infty$-uniform boundedness of $\mathcal V_p(\Psi_{p,z})$ gives that
$$ |\langle\mathcal L_{p,z}H,G\rangle|
\le C\|H\|_{H^1}\|G\|_{H^1},\ \ \forall \, H, G\in X_{p,z}.$$
Applying  the Lax--Milgram theorem on \(X_{p,z}\), we then obtain from  \eqref{3.4} that  \eqref{eq:s3-13a} holds true.  This therefore completes the proof of Lemma \ref{prop:slice-coercivity}. \end{proof}

\vspace{0.05cm}

Let \(p\in I_\delta\) and $z\in \R^N$ satisfy $D(z)\ge D_1$, where $\delta >0$ and $D_1>0$ are as in Lemma \ref{prop:slice-coercivity}. Define
\begin{align}\label{f}
\cF_{\kappa_p}(U):=-U''+\kappa_p^2U-|U|^{2p-2}U,
\ \  U\in H^1(\R;\R^N),\ \  F:=-\cF_{\kappa_p}(\Psi_{p,z}),
\end{align}
and let \(H_{\mathrm{F}}\in X_{p,z}\) be uniquely given by Lemma \ref{prop:slice-coercivity}.  We then deduce from Lemma \ref{prop:slice-coercivity} that
\begin{equation}\label{eq:energy-linear-response}
\|H_{\mathrm{F}}\|_{H^1}\le c_*^{-1}\|F\|_{H^{-1}}, \ \
\langle\mathcal L_{p,z}H_{\mathrm{F}},G\rangle=\langle F,G\rangle,\ \ \forall\, G\in X_{p,z} .
\end{equation}
The definition of $X_{p,z}$ gives that \(\mathsf M(K,\Psi_{p,z})=0\) holds for any $K\in X_{p, z}$.  This then implies that $$\mathsf M(\Psi_{p,z}+K)=I_N+\mathsf M(K)\ge I_N,$$
and  hence
\begin{equation}\label{2.70}
\cR_{\Psi_{p,z}}(K):
=[I_N+\mathsf M(K)]^{-1/2}(\Psi_{p,z}+K)\in\cM_N
\end{equation}
is well-defined.

\begin{lemma}\label{lem:energy-taylor}
Suppose \(p\in I_\delta\) and  $z\in\R^N$ satisfies \(D(z)\ge D_1\), where   \(0<\delta<1/8\) and $D_1>D_*$ are as in Lemma \ref{prop:slice-coercivity}. Then there exists \(\varepsilon_0>0\) such that for any
\(\|F\|_{H^{-1}}\le\varepsilon_0\) and
\(K\in X_{p,z}\) satisfying \(\|K\|_{H^1}\le\varepsilon_0\),
\begin{equation}\label{eq:energy-square-lower}
	E_p(\cR_{\Psi_{p,z}}(K))
	\ge E_p(\Psi_{p,z})-\langle F,H_{\mathrm{F}}\rangle
	+\frac{c_*}{2}\|K-H_{\mathrm{F}}\|_{H^1}^2-C\|F\|_{H^{-1}}^3,
\end{equation}
where \(c_*>0\) is as in  \eqref{3.4}, $F$, $H_{\mathrm{F}}$ and $ \cR_{\Psi_{p,z}}(K)$ are given by \eqref{f}--\eqref{2.70}, respectively.
\end{lemma}

\begin{proof}
We first claim  that
\begin{equation}\label{eq:energy-taylor}
\begin{aligned}
	E_p(\cR_{\Psi_{p,z}}(K))={}&E_p(\Psi_{p,z})-2\langle F,K\rangle
	+\langle\mathcal L_{p,z}K,K\rangle\\
&+O\bigl(\|F\|_{H^{-1}}\|K\|_{H^1}^2+\|K\|_{H^1}^3\bigr)\ \ \mbox{as} \ \, \|K\|_{H^1}\to0,
\end{aligned}
\end{equation}
where
$f=O\bigl(\|F\|_{H^{-1}}\|K\|_{H^1}^2+\|K\|_{H^1}^3\bigr)$ as $\|K\|_{H^1}\to 0$
means that there exists $C >0$, independent of $(p,z,K)$,
such that
$$
|f|\le C\bigl(
\|F\|_{H^{-1}}\|K\|_{H^1}^2+\|K\|_{H^1}^3
\bigr) \ \ \mbox{as} \ \, \|K\|_{H^1}\to 0.
$$
Indeed, since \(\|\mathsf M(K)\|_{\mathrm F}\le\|K\|_2^2\),  similar to \eqref{a.3}, a matrix expansion gives that
\begin{equation}\label{2.71}
[I_N+\mathsf M(K)]^{-1/2}=I_N+O(\|K\|_{H^1}^2)\ \ \text{as}\ \ \|K\|_{H^1}\to0.
\end{equation}
Denote \(h:=\cR_{\Psi_{p,z}}(K)-\Psi_{p,z}\). It then  follows from \eqref{2.70} and \eqref{2.71}  that
\begin{equation}\label{2.72}
\|h\|_{H^1}=O(\|K\|_{H^1}),\ \
\|h-K\|_{H^1}=O(\|K\|_{H^1}^2)\ \ \text{as}\ \ \|K\|_{H^1}\to0.
\end{equation}
Moreover, since    \(\Psi_{p,z}+h,\,\Psi_{p,z}\in \cM_N\), it gives that
$$\|\Psi_{p,z}+h\|^2_2=\|\Psi_{p,z}\|^2_2=N.$$
We thus calculate that 
\begin{align}\label{2.73}
&E_p(\Psi_{p,z}+h)-E_p(\Psi_{p,z})\nonumber\\
=&\Big[E_p(\Psi_{p,z}+h)+\kappa_p^2\big\|\Psi_{p,z}+h\big\|_2^2\, \Big]
-\big[E_p(\Psi_{p,z})+\kappa_p^2\|\Psi_{p,z}\|_2^2\, \big]\nonumber\\
=&2\int_\R\big(\Psi'_{p,z}\cdot h'+\kappa_p^2\Psi_{p,z}\cdot h\big)dx
+\int_\R\big(|h'|^2+\kappa_p^2|h|^2\big)dx\nonumber\\
&-\frac1p\int_\R\big(|\Psi_{p,z}+h|^{2p}-|\Psi_{p,z}|^{2p}\big)dx\\
=&\int_\R\Big[|h'|^2+\kappa_p^2|h|^2
-|\Psi_{p,z}|^{2p-2}|h|^2
-(2p-2)|\Psi_{p,z}|^{2p-4}(\Psi_{p,z}\cdot h)^2\Big]dx\nonumber\\
&+2\int_\R\bigl(\Psi'_{p,z}\cdot h'+\kappa_p^2\Psi_{p,z}\cdot h
-|\Psi_{p,z}|^{2p-2}\Psi_{p,z}\cdot h\bigr)dx+O(\|h\|_{H^1}^3)\nonumber\\
=&-2\langle F,h\rangle
+\langle\mathcal L_{p,z}h,h\rangle+O(\|h\|_{H^1}^3)\ \ \text{as} \ \ \|K\|_{H^1}\to0,\nonumber
\end{align}
where the remainder term \(O(\|h\|_{H^1}^3)\) is uniform in \((p,z)\), due to the facts that \(2p \geq \frac{15}{4} > 3\) and \(\Psi_{p,z}\) is uniformly bounded in \(H^1\).
Replacing \(h\) by \(K\) for \eqref{2.73}, we thus deduce from \eqref{2.72} and (\ref{2.73}) that  the claim \eqref{eq:energy-taylor} holds true. 

Since \(K,H_{\mathrm{F}}\in X_{p,z}\), it derives from \eqref{eq:energy-linear-response} that 
\begin{equation}\label{eq:energy-square-identity}
\begin{split}
-2\langle F,K\rangle+\langle\mathcal L_{p,z}K,K\rangle
&=-\langle F,H_{\mathrm{F}}\rangle
+\langle\mathcal L_{p,z}(K-H_{\mathrm{F}}),K-H_{\mathrm{F}}\rangle\\
&
\ge -\langle F,H_{\mathrm{F}}\rangle+c_*\|K-H_{\mathrm{F}}\|_{H^1}^2,
\end{split}
\end{equation}
where  we have used \eqref{3.4} in the last inequality.
This then implies from \eqref{eq:energy-taylor}  that  there exist positive constants $C$ and $\varepsilon_0$, independent of $(p,z,K)$, such that for  any $p\in I_\delta$, $z\in\R^N$ satisfying $D(z)\ge D_1$, and $K\in X_{p,z}$ satisfying $\|K\|_{H^1}\le\varepsilon_0$,
\begin{equation}\label{2.76}
\begin{split}
E_p(\cR_{\Psi_{p,z}}(K))
\ge&E_p(\Psi_{p,z})-\langle F,H_{\mathrm{F}}\rangle+c_*\|K-H_{\mathrm{F}}\|_{H^1}^2\\
&-C\bigl(\|F\|_{H^{-1}}\|K\|_{H^1}^2+\|K\|_{H^1}^3\bigr).
\end{split}
\end{equation}
Since
$$\|K-H_{\mathrm{F}}\|_{H^1}\le\|K\|_{H^1}+c_*^{-1}\|F\|_{H^{-1}},$$
we deduce from \eqref{eq:energy-linear-response} that
\begin{equation}\label{5.83}
\begin{split}
&\quad \ C\big(\|F\|_{H^{-1}}\|K\|_{H^1}^2+\|K\|_{H^1}^3\big)\\
&\le C\bigl(
\|F\|_{H^{-1}}\|K-H_{\mathrm{F}}\|_{H^1}^2
+\|F\|_{H^{-1}}\|H_{\mathrm{F}}\|_{H^1}^2
+\|K-H_{\mathrm{F}}\|_{H^1}^3+\|H_{\mathrm{F}}\|_{H^1}^3\bigr)\\
&\le C\bigl(\|F\|_{H^{-1}}+\|K\|_{H^1}\bigr)
\|K-H_{\mathrm{F}}\|_{H^1}^2+C\|F\|_{H^{-1}}^3.
\end{split}
\end{equation}
Choose \(\varepsilon_0>0\) sufficiently small  such that \(C\bigl(\|F\|_{H^{-1}}+\|K\|_{H^1}\bigr)\leq c_*/2\).
We then conclude from \eqref{2.76} and (\ref{5.83}) that \eqref{eq:energy-square-lower} holds true, and the proof of Lemma \ref{lem:energy-taylor} is therefore complete.\end{proof}

\section{Energy expansion of $J_N(p)$ as $p\to2$}
\label{sec:energy-conclusion}

From this section on, for convenience we always denote
\begin{equation}\label{6.1A}
q_{ij}:=e^{-\kappa_pd_{ij}},\,\
q_D:=e^{-\kappa_pD},\ \
t_{ij}:=(p-2)\kappa_pd_{ij},
\end{equation}
where $d_{ij}>0$ and $D=D(z)>0$ are defined in (\ref{5:A5}), and the Lagrange multiplier $\kappa_p\in\R$ is as in \eqref{eq:scalar-profile-definition}.
The main purpose of  this section is to derive the following uniform energy lower bound of \(J_N(p)\) as \(p\to2\),

\begin{proposition}\label{prop:two-sided-energy-lower}
Suppose that $p\in I_\delta$ and $z\in\R^N$ satisfies $D(z)>D_0$, where $\delta>0$ is sufficiently small and $D_0>0$ is sufficiently large. Then there exists a sufficiently small $ \varepsilon_0>0$   such that  for any  \(K\in X_{p,z}\) satisfying \(\|K\|_{H^1}\le \varepsilon_0\),
\begin{equation}\label{eq:uniform-energy-lower}
	\begin{aligned}
		E_p(\cR_{\Psi_{p,z}}(K))-NJ_1(p)
		\ge{}&\frac14\sum_{i<j}e^{-2\kappa_p d_{ij}}
		\left[1+d_{ij}\left(1-\frac{1-e^{-(p-2)\kappa_pd_{ij}}}{(p-2)\kappa_pd_{ij}}\right)\right]\\
		&+\frac{c_*}{2}\|K-H_{\mathrm{F}}\|_{H^1}^2
		-C|2-p|D(z)q_{D(z)}^2
		e^{\max\{2-p,0\}\kappa_pD(z)},
	\end{aligned}
\end{equation}
where \(F=-\cF_{\kappa_p}(\Psi_{p,z})\), \(H_{\mathrm{F}}\)  and \(c_*\) are as in \eqref{f}, \eqref{eq:energy-linear-response}  and \eqref{3.4}, respectively.
\end{proposition}


Denote $E_{p,z}:=G_{p,z}-I_N$, where $G_{p,z}$ is defined by (\ref{w}).
The same arguments of \eqref{a.3} and \eqref{a.5}  in the Appendix then give that for any $p\in I_{1/8}$,
\begin{align}\label{g-1}
	G_{p,z}^{-1}=I_N-E_{p,z}+O(\|E_{p,z}\|^2_{\mathrm{F}})\ \ \text{and}\ \ \|E_{p,z}\|_{\mathrm{F}}\leq P(D(z))e^{-\kappa_p D(z)}
\end{align}
uniformly  as $D(z)\to\infty$.  Moreover, set
\begin{equation}\label{eq:sep-1}
a_p:=(2\kappa_p\sqrt p)^{1/(p-1)}, \ \  p\in I_{1/8}.
\end{equation}
It is clear that $a_p\in C^\infty(I_{1/8})$.  The explicit formula of $Q_p$ gives that  for any \(p\in I_{1/8}\),
\begin{equation}\label{eq:s1-5}
\begin{split}
	Q_p(x)&=a_pe^{-\kappa_p|x|}
	\bigl(1+O(e^{-2(p-1)\kappa_p|x|})\bigr),\\
	Q_p'(x)&
	=-\kappa_pa_p\operatorname{sgn}(x)e^{-\kappa_p|x|}
	\Big(1+O\big(e^{-2(p-1)\kappa_p|x|}\big)\Big),
\end{split}
\end{equation}
uniformly as $|x|\to\infty$.
Without loss of generality, we always assume that $z=(z_1, \cdots, z_N)^\top\in\R^N$ satisfies \(z_1<\cdots<z_N\).
Set
\begin{equation}\label{eq:full-leading-pair-fields}
h_p^\pm(x):=a_pe^{\pm\kappa_px}Q_p(x)^{2p-2},\,\ H_{ij}^{(0)}:=h_p^-(\cdot-z_j)e_i+h_p^+(\cdot-z_i)e_j, \ \ i, j=1, \cdots, N.
\end{equation}
The explicit formula of \(Q_p\) gives   that  for any $p\in I_{1/8}$, \(h_p^-(x)=h_p^+(-x)\) and $0<h_p^\pm(x)
\le Ce^{-\frac34\kappa_p|x|}$ uniformly  in $\R$.
In order to prove Proposition~\ref{prop:two-sided-energy-lower}, we first derive  the following error estimate.

\begin{lemma}\label{lem:residual-response}
There exists a sufficiently large  $D_0>0$ such that for any \(p\in I_{1/8}\) and $z\in\R^N$ satisfying \(D(z)>D_0\),
\begin{equation}\label{eq:full-source-leading}
\Big\|F-\sum_{i<j}e^{-\kappa_p d_{ij}}H_{ij}^{(0)}\Big\|_2 \le C_Ne^{-\frac{11}{8}\kappa_pD(z)},
\end{equation}
where \(F\) is defined in
\eqref{f}.
\end{lemma}

\begin{proof}
Since $$(-\partial_x^2+\kappa_p^2)\Psi_{p,z}=G_{p, z}^{-1/2}(-\partial_x^2+\kappa_p^2)W_{p,z}=G_{p, z}^{-1/2}(Q_{ip}^{2p-1})_{i=1}^N,$$
we have
\begin{equation}
	F=G_{p,z}^{-1/2}\bigl(|\Psi_{p,z}|^{2p-2}Q_{ip}-Q_{ip}^{2p-1}\bigr)_{i=1}^N.
	\label{eq:s3-16}
\end{equation}
Set $F^0:=\bigl(|W_{p,z}|^{2p-2}Q_{ip}-Q_{ip}^{2p-1}\bigr)_{i=1}^N$. We shall proceed the proof by the following three steps.

\medskip\noindent\textbf{Step 1.} In this step, we claim that  for any $p\in I_{1/8}$ and $z\in\R^N$ satisfying $D(z)>1$,
\begin{equation}\label{4.9a}
\begin{split}
	&\Big\|F^0-
	\Big(\sum_{j\ne i}Q_{ip}Q_{jp}^{2p-2}\Big)_{i=1}^N\Big\|_2\\
=&\Big\|\Big[\Big(\sum_{j=1}^NQ_{jp}^2\Big)^{p-1}-\sum_{j=1}^NQ_{jp}^{2p-2}\Big](Q_{ip})_{i=1}^N\Big\|_2
	\le Ce^{-\frac{11}{8}\kappa_p D(z)}.
\end{split}
\end{equation}

Actually, we note that  for any $p\in I_{1/8}$,
\begin{align}\label{4.7}
	\arraycolsep=1.5pt
	\left\{\begin{array}{lll}
		\Big(\sum\limits_{j=1}^NQ_{jp}^2\Big)^{p-1}
		-\sum\limits_{j=1}^NQ_{jp}^{2p-2}
		=
		\sum\limits_{k=2}^N
		\Big[
		\Big(\sum\limits_{j=1}^{k}Q_{jp}^2\Big)^{p-1}
		-\Big(\sum\limits_{j=1}^{k-1}Q_{jp}^2\Big)^{p-1}
		-Q_{kp}^{2p-2}
		\Big]\ \ \text{in}\ \ \R,\\[3.5mm]
		\Big|
		\Big(\sum\limits_{j=1}^{k}Q_{jp}^2\Big)^{p-1}
		-\Big(\sum\limits_{j=1}^{k-1}Q_{jp}^2\Big)^{p-1}
		-Q_{kp}^{2p-2}
		\Big|\leq
		C\min\Big(\sum\limits_{j=1}^{k-1}Q_{jp}^2,\, Q_{kp}^2\Big)^{7/8}\ \ \text{in}\ \ \R,
	\end{array}\right.
\end{align}
where  we have used the second estimate of \eqref{eq:two-sided-power-bounds}.
Since $(a+b)^{7/8}\leq a^{7/8}+b^{7/8}$ holds for $a, b\geq0$,  and
$$
\min\Big(\sum_{j=1}^{k-1}Q_{jp}^2(x),\ Q_{kp}^2(x)\Big)
\leq
\sum_{j=1}^{k-1}\min\big(Q_{jp}^2(x),\, Q_{kp}^2(x)\big),\ \ x\in\R,
$$
we get that  
\[
\begin{aligned}
	\min\Big(\sum_{j=1}^{k-1}Q_{jp}^2,\, Q_{kp}^2\Big)^{7/8}
	\leq
	\sum_{j=1}^{k-1}\min(Q_{jp}^2,\, Q_{kp}^2)^{7/8}=
	\sum_{j=1}^{k-1}\min\big(Q_{jp},\, Q_{kp}\big)^{7/4}, \ \ x\in\R.
\end{aligned}
\]
It then  yields from \eqref{4.7} that   for any $p\in I_{1/8}$, 
\begin{equation}\label{eq:raw-source-power-bound}
	\Big|\Big(\sum_{j=1}^NQ_{jp}^2(x)\Big)^{p-1}-\sum_{j=1}^NQ_{jp}^{2p-2}(x)\Big|
	\le C\sum_{j<k}\min\big(Q_{jp}(x),Q_{kp}(x)\big)^{7/4},\ \ x\in\R.
\end{equation}
One can verify that  for $j\ne k$,
$$
|x-z_j|\geq\frac{d_{jk}}2-|x-\frac{z_j+z_k}{2}| ,\ \ \ \max\{|x-z_j|,|x-z_k|\}=\frac{d_{jk}}2+|x-\frac{z_j+z_k}{2}|,\ \ x\in\R,
$$
which implies that
\[
\begin{aligned}
|x-z_j|+\frac74\max\{|x-z_j|,|x-z_k|\}
&\ge \frac{11}{8}d_{jk}.
\end{aligned}
\]
By the exponential decay of $Q_p$,  we thus deduce from above that
\begin{align}\label{4.9}
	\big\|Q_{jp}\min(Q_{jp},Q_{kp})^{7/4}\big\|_2\leq C e^{-\frac{11}{8}\kappa_p d_{jk}}
	\leq C e^{-\frac{11}{8}\kappa_p D(z)},\ \ j\ne k.
\end{align}
Since
\begin{align*}
	&Q_{ip}(x)\min\big(Q_{jp}(x),\, Q_{kp}(x)\big)^{7/4}\\
	\leq&
	Q_{jp}\min\big(Q_{jp}(x),\, Q_{kp}(x)\big)^{7/4}+Q_{ip}\min\big(Q_{ip}(x),\, Q_{jp}(x)\big)^{7/4},\ \ i\neq j,\ i\neq k,
\end{align*}
we further deduce from \eqref{4.9} that for $ j\ne k, $
\begin{align*}
	\big\|Q_{ip}\min(Q_{jp},Q_{kp})^{7/4}\big\|_2
	\leq C e^{-\frac{11}{8}\kappa_p D(z)},\ \ i=1, \cdots, N.
\end{align*}
This there proves the claim \eqref{4.9a} in view of \eqref{eq:raw-source-power-bound}.

\medskip\noindent\textbf{Step 2.} In this step, we claim that  for any $p\in I_{1/8}$ and  $D(z)>1$,
\begin{align}\label{4.8}
	\|F^0\|_2\le C\,q_D,\ \
	\Big\|F^0-\sum_{i<j}e^{-\kappa_p d_{ij}}H_{ij}^{(0)}
	\Big\|_2\le Ce^{-\frac{11}{8}\kappa_pD(z)},
\end{align}
where $q_D$ is as in (\ref{6.1A}), and $H_{ij}^{(0)}$ is defined by (\ref{eq:full-leading-pair-fields}).

We first prove that
\begin{equation}\label{oct-1}
|Q_p(x)-a_pe^{-\kappa_px}|
\leq Ce^{-\frac32\kappa_px}\ \,\ \text{in}\ \ \mathbb R.
\end{equation}
Indeed, it yields from \eqref{eq:s1-5} that  for any $p\in I_{1/8}$,
\[
|Q_p(x)-a_pe^{-\kappa_px}|
\le Ce^{-(2p-1)\kappa_px}
\le Ce^{-\frac32\kappa_px},\qquad x\ge0.
\]
On the other hand,
\[
|Q_p(x)-a_pe^{-\kappa_px}|
\le C\bigl(e^{\kappa_px}+e^{-\kappa_px}\bigr)
\le Ce^{-\frac32\kappa_px}\qquad x<0.
\]
This proves \eqref{oct-1}.
We then calculate that  for any $p\in I_{1/8}$,  $D(z)>1$ and $i<j$,
\begin{align*}
	&|Q_{ip}Q_{jp}^{2p-2}-e^{-\kappa_p d_{ij}}h_p^-(x-z_j)|\\
	=&
	|Q_{ip}-a_pe^{-\kappa_p(x-z_i)}|Q_{jp}^{2p-2}
	\leq
	Ce^{-\frac32\kappa_p(x-z_i)}
	e^{-(2p-2)\kappa_p|x-z_j|}\\
	=&
	Ce^{-\frac32\kappa_pd_{ij}}
	e^{-\frac32\kappa_p(x-z_j)}
	e^{-(2p-2)\kappa_p|x-z_j|}
	\leq
	Ce^{-\frac32\kappa_pd_{ij}}e^{-\kappa_p|x-z_j|/4}\ \ \text{in}\ \ \R,
\end{align*}
and hence
\begin{align}\label{4.13}
	\|Q_{ip}Q_{jp}^{2p-2}-e^{-\kappa_p d_{ij}}h_p^-(x-z_j)\|_2\leq Ce^{-\frac32\kappa_pd_{ij}}\leq Ce^{-\frac32\kappa_pD(z)},\ \ i<j.
\end{align}
Similarly, one can verify that  for any $p\in I_{1/8}$ and $D(z)>1$,
$$\|Q_{jp}Q_{ip}^{2p-2}-e^{-\kappa_p d_{ij}}h_p^+(x-z_i)\|_2\leq  Ce^{-\frac32\kappa_pD(z)},\ \ i<j.
$$	
This  thus implies from  \eqref{4.13} that    for any $p\in I_{1/8}$ and $D(z)>1$,
\begin{align*}
	&
	\Big\|
	\Big(\sum_{j\ne i}Q_{ip}Q_{jp}^{2p-2}\Big)_{i=1}^N
	-\sum_{i<j}e^{-\kappa_p d_{ij}}\big[h_p^-(x-z_j)e_i+h_p^+(x-z_i)e_j\big]
	\Big\|_2
	\\
	=&\Big\|
	\sum_{i<j}
	\left[
	Q_{ip}Q_{jp}^{2p-2}e_i
	+Q_{jp}Q_{ip}^{2p-2}e_j\right]
	-\sum_{i<j}e^{-\kappa_p d_{ij}}\big[h_p^-(x-z_j)e_i+h_p^+(x-z_i)e_j\big]
	\Big\|_2
	\\
	\le&
	\sum_{i<j}
	\Big[
	\left\|Q_{ip}Q_{jp}^{2p-2}
	-e^{-\kappa_p d_{ij}}h_p^-(x-z_j)\right\|_2
	+
	\left\|Q_{jp}Q_{ip}^{2p-2}
	-e^{-\kappa_p d_{ij}}h_p^+(x-z_i)\right\|_2
	\Big]\\
	\le& Ce^{-\frac32\kappa_pD(z)}.
\end{align*}
We then deduce from \eqref{4.9a}  that $$\Big\|F^0-\sum_{i<j}e^{-\kappa_p d_{ij}}H_{ij}^{(0)}
\Big\|_2\le Ce^{-\frac{11}{8}\kappa_pD(z)}.$$
Since  $\{H_{ij}^{(0)}\}_{p,z}$ is bounded uniformly in $L^2(\R; \R^N)$, we further obtain that $\|F^0\|_2\le C\,q_D$. This completes the proof of the claim \eqref{4.8}.

\medskip\noindent\textbf{Step 3.} To complete the proof of Lemma \ref{lem:residual-response}, in view of \eqref{4.8}, it remains to prove that for any \(p\in I_{1/8}\) and sufficiently large \(D(z)>0\),
\begin{align}\label{4.14}
	\|F-F^0\|_2\le  Ce^{-\frac{11}{8}\kappa_pD(z)}.
\end{align} 

Since it yields from \eqref{eq:s3-16} that
\begin{align}\label{4.14a}
	F-F^0=(G_{p,z}^{-1/2}-I_N)F^0+
	G_{p,z}^{-1/2}\bigl([|\Psi_{p,z}|^{2p-2}-|W_{p,z}|^{2p-2}]W_{p,z}\bigr).
\end{align}
We first claim that  for any \(p\in I_{1/8}\) and sufficiently large \(D(z)>0\),
\begin{equation}\label{4.15b}
	\big\||\Psi_{p,z}|^2-|W_{p,z}|^2\big\|_2\le P(D)q_D^2.
\end{equation}
Actually, it follows that
\begin{equation}\label{4.15a}
	|\Psi_{p,z}|^2-|W_{p,z}|^2=W_{p,z}^\top(G_{p,z}^{-1}-I)W_{p,z}=\sum_{i=1}^N(G_{p,z}^{-1}-I)_{ii}Q_{ip}^2
	+2\sum_{i<j}(G_{p,z}^{-1})_{ij}Q_{ip}Q_{jp}.
\end{equation}
Set $E_{p,z}:=G_{p,z}-I_N$. Since \((E_{p,z})_{ii}=0\),  one gets from \eqref{g-1} that  for any $p\in I_{1/8}$,
\begin{equation}\label{sep0929}
	(G_{p,z}^{-1}-I)_{ii}=O(\|E_{p,z}\|_{\mathrm F}^2)\le P(D)q_D^2,\ \
	|(G_{p,z}^{-1})_{ij}|=O(\|E_{p,z}\|_{\mathrm F})\le P(D)q_D,\ \ i\ne j,
\end{equation}
uniformly as $D(z)\to\infty$.
By the exponential decay of $Q_p$, one can calculate that for any \(p\in I_{1/8}\) and sufficiently large \(D(z)>0\),
\begin{align}	\label{eq:s3-16a}
	\|Q_{ip}Q_{jp}\|_2\le C(1+d_{ij})^{1/2}e^{-\kappa_p d_{ij}},\ \ i\ne j.
\end{align}
As a consequence  of \eqref{4.15a}--\eqref{eq:s3-16a}, we obtain that the claim \eqref{4.15b} holds true.

Applying the uniform boundedness  of $\{\Psi_{p,z}\}_{p, z}$ and $\{W_{p,z}\}_{p,z}$ in $L^\infty(\R; \R^N)$, we obtain that for any $p\in I_{1/8}$,
\[
\bigl||\Psi_{p,z}|^{2p-2}-|W_{p,z}|^{2p-2}\bigr|\,|W_{p,z}|
\le C\big||\Psi_{p,z}|-|W_{p,z}|\big|\, |W_{p,z}|\leq C\big||\Psi_{p,z}|^2-|W_{p,z}|^2\big|\ \ \text{in}\ \ \R
\]
uniformly  as $D(z)\to\infty$.  This then yields from \eqref{4.15b} that for any $p\in I_{1/8}$,
\begin{align}\label{4.19a}
	\big\|
	(|\Psi_{p,z}|^{2p-2}-|W_{p,z}|^{2p-2})\,W_{p,z}\big\|_2
	\le P(D)q_D^2
\end{align}
uniformly   as $D(z)\to\infty$.
Moreover, by  the first estimate in  \eqref{4.8}, we derive from \eqref{a.7b} that  for $p\in I_{1/8}$,
\begin{align}\label{4.19}
	\|(G_{p,z}^{-1/2}-I_N)F^0\|_{L^2}
	&\le P(D)q_D\|F^0\|_{L^2}
	\le P(D)q_D^2
\end{align}
uniformly as $D(z)\to\infty$.
We thus deduce from \eqref{4.14a}, \eqref{4.19a} and \eqref{4.19} that  \eqref{4.14} holds true, and the proof of Lemma  \ref{lem:residual-response} is therefore complete.\end{proof}

\vspace{0.05cm}

We now choose $0<\delta<1/8$ so that Lemma~\ref{prop:slice-coercivity} and the  coercivity estimate \eqref{step1} hold.  
Thus, it follows from \eqref{step1} that for any $p\in I_\delta$, $L_{-,p}$ is invertible on $Q_p^\perp$. Define  for $p\in I_\delta$,
\begin{equation}\label{eq:endpoint-source-response}
	\phi_p^\pm:=L_{-,p}^{-1}
	\big(h_p^\pm-\langle h_p^\pm,Q_p\rangle Q_p\big)
	\in  Q_p^\perp,
\end{equation}
where $L_{-,p}^{-1}$ denotes the inverse of $L_{-,p}$ restricted to $Q_p^\perp$. Since \(h_p^-(x)=h_p^+(-x)\) and $Q_p(-x)=Q_p(x)$,  we have  \(\phi_p^-(x)=\phi_p^+(-x)\).
Note that $$c\|\phi_p^\pm\|_{H^1}^2\leq\langle L_{-,p} \phi_p^\pm,\, \phi_p^\pm\rangle\leq C\|\phi_p^\pm\|_2.$$
 Thus, $\{\phi_p^\pm\}_{p\in I_\delta}$ is  bounded uniformly  in $H^1(\R;\R)$. Moreover, since  $f_p^\pm:=
Q_p^{2p-2}\phi_p^\pm
+h_p^\pm
-\langle h_p^\pm,Q_p\rangle Q_p$ satisfies  $|f_p^\pm(x)|\le Ce^{-\frac34\kappa_p|x|}$ in $\R$,
we have
\begin{align}\label{4.15}
	|\phi_p^\pm(x)|+|(\phi_p^\pm)'(x)|
	\le Ce^{-\frac34\kappa_p|x|}\ \ \text{in}\ \ \R.
\end{align}
Consequently, the sequence $\{e^{\kappa_p|x|/2}\phi_p^\pm\}_{p\in I_\delta}$ is  bounded  uniformly in $H^1(\R;\R)$.
For convenience, define
\[
\Phi_{ij}^{(0)}:=\phi_p^-(\cdot-z_j)e_i+\phi_p^+(\cdot-z_i)e_j,\ \ i,j=1, \cdots,N.
\]
We next derive the following estimate.

\begin{lemma}\label{lem:quadratic-response-coefficient}
Let $\delta>0$ be as in \eqref{eq:endpoint-source-response}. Then  we have  for any $p\in I_\delta$,
\begin{equation}\label{eq:response-energy-explicit}
\langle F, H_{\mathrm{F}}\rangle
=\frac12\sum_{i<j}e^{-2\kappa_p d_{ij}}
		+O\Big(|2-p|\sum_{i<j}e^{-2\kappa_p d_{ij}}+q_D^{19/8}\Big)\ \ \mbox{uniformly as} \,\ D(z)\to\infty,
\end{equation}
where \(F\) and \(H_{\mathrm F}\) are defined in
	\eqref{f} and \eqref{eq:energy-linear-response}, respectively.
\end{lemma}

\begin{proof}
We shall carry out the proof of Lemma \ref{lem:quadratic-response-coefficient} by three steps.

\medskip\noindent\textbf{Step 1.} Denote $\Phi:=\sum_{i<j}e^{-\kappa_p d_{ij}}\Phi_{ij}^{(0)}$, and let $\mathcal V_p$ and $\mathcal L_{p,z}$ be as in \eqref{v} and \eqref{eq:linearized-reference-operator}, respectively. In this step, we claim that  for any $p\in I_\delta$, 
\begin{equation}\label{eq:s4-full-defect}
\begin{split}
	\|\mathcal L_{p,z}\Phi-F\|_{X_{p,z}^*}
	:=&\sup_{\substack{Z\in X_{p,z},\, \|Z\|_{H^1}=1}}
	|\langle\mathcal L_{p,z}\Phi-F,\, Z\rangle|\\
	=&O(q^{11/8}_D)\ \ \mbox{uniformly as} \,\ D(z)\to\infty.
\end{split}
\end{equation}

Actually, we note from Lemma \ref{lem:residual-response} that
\begin{equation}\label{4.32a}
	\|\mathcal L_{p,z}\Phi-F\|_{X_{p,z}^*}\leq C q_D^{11/8}
	+\sup_{\substack{Z\in X_{p,z},\, \|Z\|_{H^1}=1}}
	\big|\big\langle\mathcal L_{p,z}\Phi-\sum_{i<j}e^{-\kappa_p d_{ij}}H_{ij}^{(0)},\, Z\big\rangle\big|.
\end{equation}
By  the definition of $\mathcal V_p$, one can verify from \eqref{eq:endpoint-source-response} that
\begin{equation}\label{4.31b}
\begin{split}
	&\mathcal L_{p,z}\Phi_{ij}^{(0)}-H_{ij}^{(0)}\\
	=&\bigl[
	\mathcal V_p(Q_{ip}e_i)-\mathcal V_p(\Psi_{p,z})
	\bigr]
	\phi_p^+(\cdot-z_i)e_j+
	\bigl[
	\mathcal V_p(Q_{jp}e_j)-\mathcal V_p(\Psi_{p,z})
	\bigr]
	\phi_p^-(\cdot-z_j)e_i \\[2mm]
	&	-\langle h_p^-,Q_p\rangle Q_{jp}e_i
	-\langle h_p^+,Q_p\rangle Q_{ip}e_j,\ \ i<j.
\end{split}
\end{equation}
It follows  from \eqref{xpz} that for any  \(Z=(Z_1, \cdots, Z_N)^\top\in X_{p,z}\),  
\begin{equation}\label{M:4.32a}
\langle Q_{jp}e_i,\, Z\rangle=\langle Q_{jp},Z_i\rangle=0,\ \  i, j=1, \cdots, N,
\end{equation}

We note that for any $p\in I_\delta$ and sufficiently large $D(z)>0$, the sequences $\{Q_{ip}e_i\}_{p,z}$ and $\{\Psi_{p,z}\}_{p,z}$ are bounded uniformly in $L^\infty(\R; \R^N)$. We then  verify  from  \eqref{v} that  for any  $p\in I_\delta$ and sufficiently large $D(z)>0$,
\[
\bigl\|\mathcal V_p(Q_{ip}e_i)
-\mathcal V_p(\Psi_{p,z})\bigr\|_{\mathrm F}
\le C(|Q_{ip}e_i|+|\Psi_{p,z}|)^{2p-3}|Q_{ip}e_i-\Psi_{p,z}|
\leq C |Q_{ip}e_i-\Psi_{p,z}| .
\]
Since
\[
\Psi_{p,z}-Q_{jp}e_j=(\Psi_{p,z}-W_{p,z})+\sum_{k\ne j}Q_{kp}e_k,
\]
we conclude that for $i\neq j$,
\begin{align}\label{4.33a}
&\Big\|\bigl[
\mathcal V_p(Q_{ip}e_i)-\mathcal V_p(\Psi_{p,z})\bigr]\phi_p^\pm (\cdot-z_i)e_j\Big\|_2\nonumber\\
\leq &\Big\|\big\|
\mathcal V_p(Q_{ip}e_i)-\mathcal V_p(\Psi_{p,z})\big\|_{\mathrm{F}}\phi_p^\pm(\cdot-z_i)e_j\Big\|_2\\
\leq &
C\Big\|\,\big|\Psi_{p,z}-Q_{ip}e_i\big|\,
\big|\phi_p^\pm(\cdot-z_i)\big|\Big\|_2\nonumber\\
\le &C\|\Psi_{p,z}-W_{p,z}\|_\infty\|\phi_p^\pm\|_2
+C\sum_{k\ne i}\|Q_{kp}\phi_p^\pm(\cdot-z_i)\|_2.\nonumber
\end{align}
By   the exponential decay of $(Q_p, Q_p')$, we calculate from \eqref{4.15} that  \begin{align}\label{4.33b}
	\|Q_{kp}\phi_p^\pm (\cdot-z_i)\|_2=O(q_D^{3/4}),\ \ k\neq i.
\end{align}
We then deduce from
\eqref{4.33a}, \eqref{4.33b} and \eqref{eq:s4-full-background}  that
\begin{align*}
	\Big\|\bigl[
	\mathcal V_p(Q_{ip}e_i)-\mathcal V_p(\Psi_{p,z})\bigr]\phi_p^\pm(\cdot-z_i)e_j\Big\|_2
	= O(q_D^{3/4}),
\end{align*}
which thus yields from  (\ref{4.31b}) and (\ref{M:4.32a}) that
\begin{align}\label{6:4.33a}
|\langle\mathcal L_{p,z}\Phi_{ij}^{(0)}-H_{ij}^{(0)},Z\rangle|
\le Cq_D^{3/4}\|Z\|_{H^1}, \ \ \forall\ Z\in X_{p,z}.
\end{align}
Multiplying (\ref{6:4.33a}) by \(e^{-\kappa_pd_{ij}}\) and  summing over $i<j$, we therefore obtain from (\ref{6:4.33a}) and \eqref{4.32a} that the claim \eqref{eq:s4-full-defect} holds true.

\medskip\noindent\textbf{Step 2.} This step is to prove that for any $p\in I_\delta$,
\begin{equation}\label{eq:linear-leading-profiles}
	\big\|H_{\mathrm{F}}-\Phi\big\|_{H^1}
	=O\big(q_D^{11/8}\big) \ \ \text{uniformly}\ \   \text{as}\ \, D(z)\to\infty,
\end{equation}
where \(H_{\mathrm F}\) is defined by \eqref{eq:energy-linear-response}, and $\Phi=\sum_{i<j}e^{-\kappa_p d_{ij}}\Phi_{ij}^{(0)}$ is given by Step 1.

Let $T_i(p,z)$ be as in \eqref{tder}, and denote
\begin{align*}
	\Gamma(p, z):=\big(\Gamma_{ij}(p,z)\big)_{i, j=1}^N,\ \ \ \Gamma_{ij}(p,z):=\big\langle T_i(p,z),\, T_j(p,z)\big\rangle,\ \ i,j=1, \cdots,N.
\end{align*}
Recall from \eqref{2.35} that for any $p\in I_\delta$,  $\|T_j(p, z)+e_j Q'_p(x-z_j)\|_{H^1}\to0$
uniformly  as $D(z)\to\infty$,
which yields that  for any  $p\in I_\delta$  and sufficiently large $D(z)>0$, $\Gamma$ is invertible and the sequence $\{\|\Gamma^{-1}\|_{\mathrm{F}}\}_{p,z}$ is bounded uniformly, due to the fact that   \(\inf_{p\in I_\delta}\|Q_p'\|_2^2>0\). 
Denote
\[
\Pi_{p,z}\Phi:=\Phi-\mathsf M(\Phi,\Psi_{p,z})\Psi_{p,z}
-\sum_{i,j=1}^N(\Gamma^{-1})_{ij}\langle \Phi,T_i(p,z)\rangle T_j(p,z).
\]
We first claim that for $p\in I_\delta$,
\begin{equation}\label{eq:leading-projection-error}
	\|\Pi_{p,z}\Phi-\Phi\|_{H^1}=O(q_D^{7/4})\ \
	\text{uniformly}\ \   \text{as}\ \, D(z)\to\infty.
\end{equation}
Actually, the same argument of \eqref{4.33b} gives that
\begin{equation}\label{eq:s4-cross-center}
	\begin{aligned}
		|\langle\phi_p^\pm(\cdot-z_i),Q_{kp}\rangle|
		+|\langle\phi_p^\pm(\cdot-z_i),Q_{kp}'\rangle|
		\le Ce^{-\frac{3}{4}\kappa_pd_{ik}},\ \ i\ne k,
	\end{aligned}
\end{equation}
Following the facts that $\phi_p^{\pm}\perp Q_p$ and \(\|\Phi\|_2= O(q_D)\), we then deduce  from  \eqref{eq:s4-cross-center} and \eqref{eq:s4-full-background}  that for  any $p\in I_\delta$ and $k=1, \cdots, N$,
\begin{equation}\label{4.31a}
\begin{split}
	&\|\mathsf M(\Phi,\Psi_{p,z})\|_{\mathrm{F}}+|\langle \Phi, T_k(p,z)\rangle|\\
=&\|\mathsf M(\Phi,W_{p,z})\|_{\mathrm{F}}+|\langle \Phi, -e_kQ'_{kp}\rangle|+O\big(P(D)q_D^2\big)=O\big(q_D^{7/4}\big)
\end{split}
\end{equation}
uniformly as $D(z)\to\infty$. By the uniform boundedness of \(\{\|\Gamma^{-1}\|_{\mathrm{F}}\}_{p,z}\),   we thus get from \eqref{4.31a} that the claim \eqref{eq:leading-projection-error} holds true.

Since \(\mathsf M(\Psi_{p,z})=I_N\) and \(\mathsf M(T_i(p,z),\Psi_{p,z})=0\) for $i=1, \cdots, N$, we have
\[
\big\langle  \Psi_{p,z},T_k(p,z)\big\rangle=0,\ \ \mathsf M(\Pi_{p,z}\Phi,\Psi_{p,z})=0.\]
We then deduce that
\[
\begin{aligned}
	&\quad\langle\Pi_{p,z}\Phi,T_k(p,z)\rangle\\
	&=\langle\Phi,T_k(p,z)\rangle
	-\big\langle\mathsf M(\Phi,\Psi_{p,z})\Psi_{p,z},T_k(p,z)\big\rangle
	-\sum_{j=1}^N\langle\Phi,T_j(p,z)\rangle
	\sum_{i=1}^N\Gamma_{ki}(\Gamma^{-1})_{ij}\\
	&=\langle\Phi,T_k(p,z)\rangle-\langle\Phi,T_k(p,z)\rangle=0,
\end{aligned}
\]
which implies that  \(\Pi_{p,z}\Phi\in X_{p,z}\), and hence \(H_{\mathrm{F}}-\Pi_{p,z}\Phi\in X_{p,z}\).
Choose  $H=H_{\mathrm{F}}-\Pi_{p,z}\Phi$, so that it yields from \eqref{3.4} that
\begin{align*}
	c_*\|H_{\mathrm{F}}-\Pi_{p,z}\Phi\|_{H^1}^2
	\le&\big\langle\mathcal L_{p,z}\big(H_{\mathrm{F}}-\Pi_{p,z}\Phi\big),\, H_{\mathrm{F}}-\Pi_{p,z}\Phi\big\rangle\\
	=&\big\langle F-\mathcal L_{p,z}\Pi_{p,z}\Phi,\, H_{\mathrm{F}}-\Pi_{p,z}\Phi\big\rangle\\
\le&\|F-\mathcal L_{p,z}\Pi_{p,z}\Phi\|_{X_{p,z}^*}
	\|H_{\mathrm{F}}-\Pi_{p,z}\Phi\|_{H^1}.\nonumber
\end{align*}
This further gives  from \eqref{eq:s4-full-defect} and \eqref{eq:leading-projection-error} that
\begin{align*}
	\|H_{\mathrm{F}}-\Phi\|_{H^1}\leq& \|H_{\mathrm{F}}-\Pi_{p,z}\Phi\|_{H^1}+\|\Pi_{p,z}\Phi-\Phi\|_{H^1}\\
	\le &\|\Pi_{p,z}\Phi-\Phi\|_{H^1}+C\|F-\mathcal L_{p,z}\Pi_{p,z}\Phi\|_{X_{p,z}^*}\\
	\le &C\|F-\mathcal L_{p,z}\Phi\|_{X_{p,z}^*}
	+C\|\Pi_{p,z}\Phi-\Phi\|_{H^1}= O(q_D^{11/8}),
\end{align*}
which therefore proves \eqref{eq:linear-leading-profiles}.

\medskip\noindent\textbf{Step 3.} In this step, we are ready to finish the proof of Lemma \ref{lem:quadratic-response-coefficient}. Since it follows from \eqref{eq:full-source-leading} that  for any $p\in I_\delta$, \(\|F\|_2=O(q_D)\) uniformly as $D(z)\to\infty$,
we get from \eqref{eq:s3-13a}  that  for any $p\in I_\delta$,
\begin{equation}\label{eq:residual-response-size}
	\|H_{\mathrm{F}}\|_{H^1}\le C \|F\|_{H^{-1}}\le C\|F\|_2=O(q_D) \ \ \mbox{uniformly  as}\,\ D(z)\to\infty.
\end{equation}
Let $H_{ij}^{(0)}$  be as in \eqref{eq:full-leading-pair-fields}.
We  then conclude  from \eqref{eq:full-source-leading},  \eqref{eq:linear-leading-profiles} and \eqref{eq:residual-response-size} that  for any $p\in I_\delta$,
\begin{equation}\label{4.36a}
\begin{split}
	\langle F, H_{\mathrm{F}}\rangle=&\big\langle \sum_{i<j}e^{-\kappa_pd_{ij}}H_{ij}^{(0)}, \Phi\big\rangle+\big\langle F-\sum_{i<j}e^{-\kappa_pd_{ij}}H_{ij}^{(0)}, H_{\mathrm{F}}\big\rangle\\
&-\big\langle \sum_{i<j}e^{-\kappa_pd_{ij}}H_{ij}^{(0)}, \Phi-H_{\mathrm{F}}\big\rangle \\
	=&
	\sum_{i<j}\sum_{k<l}e^{-\kappa_p(d_{ij}+d_{kl})}
	\langle H_{ij}^{(0)},\Phi_{kl}^{(0)}\rangle
	+O\big(q_D^{19/8}\big)  \ \ \mbox{uniformly  as}\,\ D(z)\to\infty.
\end{split}
\end{equation}
Note that for $i<j$ and $k<l$,
\begin{align*}
	\langle H_{ij}^{(0)},\, \Phi_{kl}^{(0)}\rangle
	=&\big\langle h_p^-(\cdot-z_j)e_i+h_p^+(\cdot-z_i)e_j,\, \phi_p^-(\cdot-z_l)e_k+\phi_p^+(\cdot-z_k)e_l\big\rangle\\
	=&\begin{cases}
		2\langle h_p^+,\phi_p^+\rangle,&(i,j)=(k,l);\\
		0,&\{i,j\}\cap\{k,l\}=\varnothing;\\
		O(q_D^{1/2}),&\text{otherwise};
	\end{cases}
\end{align*}
where the last estimate follows from the exponential decay of $(h_p^\pm, \phi_p^\pm)$.  This thus  yields from  \eqref{4.36a} that for any $p\in I_\delta$,
\begin{equation}\label{eq:response-energy-expansion}
	\langle F,H_{\mathrm{F}}\rangle
	=2\langle h_p^+,\phi_p^+\rangle\sum_{i<j}e^{-2\kappa_pd_{ij}}
	+O\big(q_D^{19/8}\big) \ \ \mbox{uniformly  as}\,\ D(z)\to\infty.
\end{equation}

We next estimate the term $\langle h_p^+,\phi_p^+\rangle$ of (\ref{eq:response-energy-expansion}).  A direct calculation gives that
\begin{align*}
	h_2^+=\frac14w_*-w_*',\ \
	\langle h_2^+,w_*\rangle=\frac14,\ \
	L_{-,2}\phi_2^+=-w_*',\ \ L_{-,2}(xw_*/2)=-w_*'.
\end{align*}
Since \(w_*\perp\{\phi_2^+, \, xw_*/2\}\) and ker$L_{-,2}=\hbox{span}\{w_*\}$, we obtain from the above that $\phi_2^+= xw_*/2$ in $\R$, which then yields that \begin{align}\label{14}
	\langle h_2^+,\phi_2^+\rangle=1/4.
\end{align}
By the explicit formula of $(Q_p, h_p^+)$, since $\kappa_p, a_p\in C^1(I_\delta)$,   one can  verify that  for any $p\in I_\delta$,
$$|\partial_pQ_p(x)|
+|\partial_ph_p^+(x)|
+|\partial_p(Q_p(x)^{2p-2})|
\le C(1+|x|)e^{-c|x|}\ \ \text{in}\ \ \R.$$
We hence deduce that
\begin{align}\label{p-2}
	\|h_p^+-h_2^+\|_2+\|Q_p-Q_2\|_2
	+|\kappa_p^2-\kappa_2^2|
	+\|Q_p^{2p-2}-Q_2^2\|_\infty
	\le C|p-2|.
\end{align}
Moreover, set $
v:=\phi_p^+-\phi_2^+
+\langle\phi_2^+,Q_p\rangle Q_p.$
Then $\langle v,Q_p\rangle=0$ and
\begin{align}\label{4.39}
	|\langle\phi_2^+,Q_p\rangle|
	=|\langle\phi_2^+,Q_p-Q_2\rangle|
	\le C|p-2|.
\end{align}
Applying $L_{-,p}Q_p=0$  and following the equations of $(\phi_p^+, \phi_2^+)$,  one gets   from \eqref{step1} that
\[
\begin{aligned}
	c\|v\|_{H^1}^2
	\le \langle L_{-,p}v,v\rangle
	=&\langle h_p^+-h_2^+,v\rangle
	+\langle h_2^+,Q_2\rangle\langle Q_2-Q_p,v\rangle\\
	&+\int_{\mathbb R}
	\bigl(\kappa_2^2-\kappa_p^2+Q_p^{2p-2}-Q_2^2\bigr)
	\phi_2^+v\,dx\le C|p-2|\|v\|_{H^1},
\end{aligned}
\]
which yields that $\|v\|_{H^1}\le C|p-2|$. We thus obtain from \eqref{4.39} that
\[
\|\phi_p^+-\phi_2^+\|_2
\le \|v\|_2+|\langle\phi_2^+,Q_p\rangle|
\le C|p-2|.
\]
We then derive from \eqref{14} and \eqref{p-2} that
\begin{equation*}
	2\langle h_p^+,\phi_p^+\rangle=\frac12+O(|2-p|),
\end{equation*}
which thus yields from \eqref{eq:response-energy-expansion} that
\eqref{eq:response-energy-explicit} holds true. The proof of Lemma \ref{lem:quadratic-response-coefficient} is therefore complete.\end{proof}

\subsection{Proof of Proposition~\ref{prop:two-sided-energy-lower}}

The purpose of this subsection is to complete the proof of Proposition~\ref{prop:two-sided-energy-lower}. Towards this aim, we now derive the following energy expansion with explicit coefficients.

\begin{lemma}\label{lem:reference-energy}
Under the assumptions of
Proposition~\ref{prop:two-sided-energy-lower}, we have for any \(p\in I_\delta\),
\begin{equation}
	\begin{aligned}
		E_p(\Psi_{p,z})-NJ_1(p)
		=&\sum_{i<j}e^{-2\kappa_p d_{ij}}
		\bigg[d_{ij}\Big(4\kappa_pa_p^4
		-a_p^{2p}\frac{1-e^{-(p-2)\kappa_pd_{ij}}}{(p-2)\kappa_pd_{ij}}\Big)+1\bigg]\\
		&+O\Big(|2-p|\sum_{i<j}e^{-2\kappa_p d_{ij}}e^{\max\{(2-p)\kappa_pd_{ij},0\}}
		+P(D)q_D^{5/2}\Big)
	\end{aligned}
	\label{eq:reference-energy-explicit}
\end{equation}
uniformly  as  \(D(z)\to\infty\).
\end{lemma}

\begin{proof} We proceed the proof of Proposition \ref{lem:reference-energy} by the following two steps.

\noindent\textbf{Step 1.}
In this step, we prove that for any \(p\in I_{1/8}\),
\begin{align}\label{eq:reference-energy-expansion}
E_p(\Psi_{p,z})-NJ_1(p)
		=&\sum_{i<j}\Big\{
		4\kappa_pa_p^2e^{-\kappa_p d_{ij}}\inte Q_{ip}Q_{jp}dx\\
&\qquad-\frac1p\int_\R
		\big[(Q_{ip}^2+Q_{jp}^2)^p-Q_{ip}^{2p}-Q_{jp}^{2p}\big]dx\Big\}
+O\Big(P(D)q_D^{5/2}\Big)\nonumber
\end{align}
uniformly as  \(D(z)\to\infty\).

To prove (\ref{eq:reference-energy-expansion}), we note that
\[
E_p(\Psi_{p,z})-NJ_1(p)
=\big[E_p(\Psi_{p,z})-E_p(W_{p,z})\big]+\big[E_p(W_{p,z})-NJ_1(p)\big].
\]
Since $Q_p$ is a minimizer of  \(J_1(p)\),  we have  for any $p\in I_\delta$,
\begin{align}\label{4.26}
	E_p(W_{p,z})-NJ_1(p)
	&=-\frac1p\int_\R\Big(|W_{p,z}|^{2p}-\sum_{i=1}^NQ_{ip}^{2p}\Big)dx\\
	&=-\frac1p\sum_{i<j}\int_\R
	\Big[\big(Q_{ip}^2+Q_{jp}^2\big)^p-Q_{ip}^{2p}-Q_{jp}^{2p}\Big]dx
	+O \Big(P(D)q_D^{5/2}\Big)\nonumber
\end{align}
uniformly as $D(z)\to\infty$, where the second identity follows from Lemma~\ref{lem:three-label-energy} in the Appendix.  We  first claim that for any $\ p\in I_\delta$,
\begin{equation}\label{eq:reference-gram-combined}
	\begin{aligned}
		E_p(\Psi_{p,z})-E_p(W_{p,z})
		=&\sum_{i,j=1}^N(G_{p,z}^{-1}-I)_{ij}
		\int_\R\Big[Q_{ip}^{2p-1}Q_{jp}-|W_{p,z}|^{2p-2}Q_{ip}Q_{jp}\Big]dx\\
		&+O\left(P(D)q_D^{15/4}\right)\ \ \text{uniformly }\ \text{as}\,\ D(z)\to\infty.
	\end{aligned}
\end{equation}

Indeed, we set
 $$   \eta:=|\Psi_{p,z}|^2-|W_{p,z}|^2.$$
Since \(\|\Psi_{p,z}\|_2^2=N\)  and
\begin{equation}
	\int_\R \big[Q_{ip}'Q_{jp}'+\kappa_p^2Q_{ip}Q_{jp}\big]dx
	=\int_\R Q_{ip}^{2p-1}Q_{jp}dx
	=\int_\R Q_{jp}^{2p-1}Q_{ip}dx,\ \ i,j=1, \cdots,N,
	\label{eq:scalar-cross-identity}
\end{equation}
we deduce that
\begin{equation}\label{eq:reference-energy-exact}
	\begin{aligned}
		E_p(\Psi_{p,z})+\kappa_p^2N
		&=\int_\R\bigl(|\Psi_{p,z}'|^2+\kappa_p^2|\Psi_{p,z}|^2\bigr)dx
		-\frac1p\int_\R|\Psi_{p,z}|^{2p}dx\\
		&=\sum_{i,j}(G_{p,z}^{-1})_{ij}\int_\R \big(Q_{ip}'Q_{jp}'+\kappa_p^2Q_{ip}Q_{jp}\big)dx
		-\frac1p\int_\R(|W_{p,z}|^2+\eta)^pdx\\
		&=\sum_{i,j}(G_{p,z}^{-1})_{ij}\int_\R Q_{ip}^{2p-1}Q_{jp}
		-\frac1p\int_\R(|W_{p,z}|^2+\eta)^p,
	\end{aligned}
\end{equation}
which yields that
\begin{equation}\label{4.27}
\begin{split}
	E_p(\Psi_{p,z})-E_p(W_{p,z})=&\sum_{i,j}(G_{p,z}^{-1}-I_N)_{ij}\int_\R Q_{ip}^{2p-1}Q_{jp}\\
	&-\frac1p\int_\R\big[(|W_{p,z}|^2+\eta)^p-|W_{p,z}|^{2p}\big]dx.
\end{split}
\end{equation}
Applying the  matrix bound of \eqref{a.7}, we derive that  for any $p\in I_\delta$,  \(\eta|W_{p,z}|^{-2}\in[-1/2,\, 1/2]\)  uniformly as $D(z)\to\infty$.  By  Taylor's  expansion with Lagrange remainder, we then get that for any $p\in I_\delta$,
\begin{equation}\label{4.32}
	\Big|\big(1+\eta|W_{p,z}|^{-2}\big)^p-1-p\eta|W_{p,z}|^{-2}\Big|\leq C(\eta|W_{p,z}|^{-2})^2\ \ \text{in}\ \,\R  \,\ \text{uniformly}\ \, \text{as}\,\ D(z)\to\infty.
\end{equation}
Following the fact that
\[
\eta^2|W_{p,z}|^{2p-4}
=|\eta|^{15/8}\left(|\eta||W_{p,z}|^{-2}\right)^{1/8}
|W_{p,z}|^{2p-15/4}\le C|\eta|^{15/8}\ \ \text{in}\,\ \R,
\]
we  deduce that  for any $p\in I_\delta$,
\[
\frac{1}{p}\Big|(|W_{p,z}|^2+\eta)^p-|W_{p,z}|^{2p}-p|W_{p,z}|^{2p-2}\eta\Big|
\le C\eta^2|W_{p,z}|^{2p-4}\le C|\eta|^{15/8}\ \ \text{in}\,\ \R
\]
uniformly as $D(z)\to\infty$. By the exponential decay of $Q_p$,
we derive from \eqref{4.15a} and \eqref{sep0929} that
$\|\eta\|_{15/8}=O(P(D)q_D^2)$ as $D(z)\to\infty$. It follows that
\begin{equation*}
	\frac1p\int_\R\big[(|W_{p,z}|^2+\eta)^p-|W_{p,z}|^{2p}\big]
	=\int_\R |W_{p,z}|^{2p-2}\eta+O\Big(P(D)q_D^{15/4}\Big)
\end{equation*}
uniformly as $D(z)\to\infty$,
which hence implies from \eqref{4.27} that the claim \eqref{eq:reference-gram-combined} holds true.

We now prove that for any $p\in I_\delta$,
\begin{equation}\label{eq:reference-gram-leading}
	E_p(\Psi_{p,z})-E_p(W_{p,z})
	=2\sum_{i<j}\langle Q_{ip},Q_{jp}\rangle\int_\R Q_{ip}^{2p-1}Q_{jp}dx
	+O\!\left(P(D)q_D^{23/8}\right)
\end{equation}
uniformly as $D(z)\to\infty$. Actually, we first recall from \eqref{4.8} that $$F^0=\bigl(|W_{p,z}|^{2p-2}Q_{ip}-Q_{ip}^{2p-1}\bigr)_{i=1}^N$$ satisfies $\|F^0\|_2=O(q_D)$ as $D(z)\to\infty$.  We then obtain from \eqref{g-1} and \eqref{eq:reference-gram-combined} that
\begin{equation}\label{4.31}
	\begin{aligned}
		&E_p(\Psi_{p,z})-E_p(W_{p,z})\\
		=&\sum_{i\ne j}(E_{p,z})_{ij}\int_\R\Big[|W_{p,z}|^{2p-2}Q_{ip}-Q_{ip}^{2p-1}\Big]Q_{jp}dx
		+O\left(P(D)q_D^3\right)
	\end{aligned}
\end{equation}
uniformly as $D(z)\to\infty$.
Note from \eqref{eq:two-sided-power-bounds}  that
\begin{align}\label{4.33}
	&\Big||W_{p,z}|^{2p-2}
	-Q_{ip}^{2p-2}-Q_{jp}^{2p-2}\Big|\nonumber\\
	\le&
	\left||W_{p,z}|^{2p-2}
	-(Q_{ip}^2+Q_{jp}^2)^{p-1}\right|+
	\left|(Q_{ip}^2+Q_{jp}^2)^{p-1}
	-Q_{ip}^{2p-2}-Q_{jp}^{2p-2}\right|\\
	\leq& C\Big(\sum_{k\ne i,j}Q_{kp}^2\Big)^{7/8}+C\min(Q_{ip}^2,Q_{jp}^2)^{7/8}
	\leq C\sum_{k\ne i,j}Q_{kp}^{7/4}+C(Q_{ip}Q_{jp})^{7/8}\ \ \text{in}\ \ \R.\nonumber
\end{align}
Using the explicit formula of $Q_p$, we deduce that
\begin{align}\label{4.34}
	\int_\R Q_{kp}^{7/4}Q_{ip}Q_{jp}dx=O(q_D^2),\ \  i\ne j,\  k\notin\{i,j\}.
\end{align}
We then get from \eqref{eq:s3-16a}, \eqref{4.33}  and \eqref{4.34} that  for any $p\in I_\delta$,
\begin{equation}\label{eq:reference-offdiagonal}
	\begin{split}
		&\Big|\int_\R\Big[|W_{p,z}|^{2p-2}Q_{ip}-Q_{ip}^{2p-1}\Big]Q_{jp}dx-\inte Q_{jp}^{2p-1}Q_{ip}dx\Big|\\
		\leq&
		C\sum_{k\ne i,j}\int_\R Q_{kp}^{7/4}Q_{ip}Q_{jp}dx+C\int_\R(Q_{ip}Q_{jp})^{15/8}dx
		=O\Big(P(D)q_D^{15/8}\Big)
\end{split}\end{equation}
uniformly as $D(z)\to\infty$.
This yields from \eqref{g-1} and \eqref{4.31} that \eqref{eq:reference-gram-leading} holds true.

Since it follows from \eqref{eq:s1-5}  that  for any $p\in I_{1/8}$ and $i\neq j$,
\begin{align}
	\arraycolsep=1.5pt
	\left\{\begin{array}{lll}
		\left|Q_p(x-d_{ij})-a_pe^{-\kappa_p d_{ij}}e^{\kappa_px}\right|
		\le Ce^{-p\kappa_p d_{ij}}e^{\kappa_px}\ \ \ &\text{on} \ \ (-\infty, \, d_{ij}/2],\\[3mm]
		Q_p(x)\le Ce^{-\kappa_px},\ \
		Q_p(x-d_{ij})\le Ce^{-\kappa_p d_{ij}}e^{\kappa_px}\ \ \ &\text{on} \ \ (d_{ij}/2, \, \infty),
	\end{array}\right.
\end{align}
uniformly as $D(z)\to\infty$,
we have
\[
\left|
\int_{\mathbb R}Q_{ip}^{2p-1}Q_{jp}\,dx
-a_pe^{-\kappa_p d_{ij}}\int_{\mathbb R}e^{\kappa_px}Q_p^{2p-1}\,dx
\right|\\
\le Ce^{-p\kappa_p d_{ij}}  \ \ \text{uniformly}\ \, \text{as}\,\ D(z)\to\infty.
\]
We then calculate that   for $i\neq j$,
\begin{equation}\label{4.36}
	\begin{aligned}
		\int_\R Q_{ip}^{2p-1}Q_{jp}dx
		&=a_pe^{-\kappa_p d_{ij}}\int_\R e^{\kappa_px}Q_p(x)^{2p-1}\,dx
		+O\Big(P(D)q_D^{15/8}\Big)\\
		&=a_pe^{-\kappa_p d_{ij}}\left[e^{\kappa_px}
		\big(\kappa_pQ_p(x)-Q_p'(x)\big)\right]_{-\infty}^{+\infty}
		+O\Big(P(D)q_D^{15/8}\Big)\\
		&=2\kappa_pa_p^2e^{-\kappa_p d_{ij}}
		+O\Big(P(D)q_D^{15/8}\Big),
	\end{aligned}
\end{equation}
where we have used \eqref{q} for the second identity.  As a consequence of \eqref{4.26}, \eqref{eq:reference-gram-leading} and \eqref{4.36},  we obtain that  \eqref{eq:reference-energy-expansion} holds, which hence establishes Step 1.

\medskip\noindent\textbf{Step 2.} Applying \eqref{eq:reference-energy-expansion} of Step 1, it remains to prove that for any $p\in I_{1/8}$ and $i\neq j$,
\begin{align}\label{eq:energy-overlap}
	e^{\kappa_p d_{ij}}\langle Q_{ip},Q_{jp}\rangle
	&=a_p^2d_{ij}+O\Big(|2-p|+q_D^{7/8}\Big),
\end{align}
and
\begin{equation}\label{eq:energy-corridor}
\begin{split}
	&\frac{e^{2\kappa_pd_{ij}}}{p}\int_\R\big[(Q_{ip}^2+Q_{jp}^2)^p-Q_{ip}^{2p}-Q_{jp}^{2p}\big]dx\\
	=&a_p^{2p}d_{ij}\frac{1-e^{-t_{ij}}}{t_{ij}}-1
	+O\Big(|2-p| e^{\max\{-t_{ij},0\}}+q_D^{3/4}\Big)\end{split}
\end{equation}
uniformly  as $D(z)\to\infty$, where  $t_{ij}:=(p-2)\kappa_pd_{ij}$.

We first prove \eqref{eq:energy-overlap}.  Actually, note from \eqref{eq:s1-5}  that  for any $p\in I_{1/8}$ and $i\neq j$,
\begin{align}\label{4.42a}
	\arraycolsep=1.5pt
	\left\{\begin{array}{lll}
		Q_p(x-d_{ij})=a_pe^{-\kappa_p d_{ij}} e^{\kappa_px}
		\bigl(1+O(e^{-(2p-2)\kappa_p(d_{ij}-x)})\bigr)\ \ \ &\text{on} \ \ (-\infty, \, d_{ij}/2],\\[3mm]
		e^{\kappa_px}Q_p(x)-a_p=O(e^{-(2p-2)\kappa_px})\ \ \ &\text{on} \ \ (d_{ij}/2, \, \infty),
	\end{array}\right.
\end{align}
uniformly as $D(z)\to\infty$.
We then calculate that for any $p\in I_{1/8}$ and $i\neq j$,
\begin{equation}\label{4.42}
\begin{split}
	e^{\kappa_p d_{ij}}\langle Q_{ip},Q_{jp}\rangle
	=&2e^{\kappa_pd_{ij}}\int_{-\infty}^{d_{ij}/2}Q_pQ_p(x-d_{ij})\,dx \\
=&2a_p\int_{-\infty}^{d_{ij}/2}e^{\kappa_px}Q_p\,dx
	+O(e^{-(p-1)\kappa_p d_{ij}})\\
	=&a_p^2d_{ij}+2a_p\int_{-\infty}^{d_{ij}/2}
	\big[e^{\kappa_px}Q_p-a_p\mathbf1_{(0,\infty)}(x)\big]\,dx	+O(e^{-(p-1)\kappa_p d_{ij}}) \\
	=&a_p^2d_{ij}+2a_p\inte
	\big[e^{\kappa_px}Q_p-a_p\mathbf1_{(0,\infty)}(x)\big]\,dx	+O(e^{-(p-1)\kappa_p d_{ij}})\end{split}
\end{equation}
uniformly  as $D(z)\to\infty$, where  the first identity follows from the evenness of \(Q_p\), and the second identity follows from the uniform boundedness of \(\|e^{\kappa_px}Q_p\|_{L^\infty((-\infty, d_{ij}/2])}\).
Moreover, recall that
\begin{equation}\label{eq:scalar-exact-tail}
	e^{\kappa_px}Q_p(x)
	=a_p\bigl(1+e^{-2(p-1)\kappa_px}\bigr)^{-1/(p-1)},\ \ x\geq0.
\end{equation}
Differentiate \eqref{eq:scalar-exact-tail} with respect to $p$, and use the evenness of $Q_p$ and the fact that
$\kappa_p,a_p\in C^1(I_{1/8})$. One can then verify that for any $p\in I_{1/8}$,
\begin{align}\label{4.45}
	\left|
	\partial_p\left[e^{\kappa_px}Q_p(x)
	-a_p\mathbf1_{(0,\infty)}(x)\right]
	\right|
	\le C(1+|x|)e^{-c|x|}\ \ \text{in}\ \ \R,
\end{align}
where the constants $C>0$  and $c>0$ are independent of $x$ and $p$. Since $\inte
\big[e^{\kappa_2x}Q_2-a_2\mathbf1_{(0,\infty)}(x)\big]\,dx=0$,  we conclude that
\begin{align}\label{4.44}
	|f(p)|=|f(p)-f(2)|=\big|\int_2^pf'(s)ds\big|\leq C|p-2|,
\end{align}
where $f(p):=\inte\big[e^{\kappa_px}Q_p-a_p\mathbf1_{(0,\infty)}(x)\big]\,dx$. This thus implies from \eqref{4.42} that  \eqref{eq:energy-overlap} holds true.

We next prove \eqref{eq:energy-corridor} as follows. 
Denote \(v(x):=Q_p^2(x-d_{ij})Q_{p}^{-2}\le1\) for $x\in(-\infty, d_{ij}/2]$.  Then
\begin{align*}
	\frac{1}{p}\Big[\big(Q_{p}^2+Q_{p}^2(x-d_{ij})\big)^p-Q_{p}^{2p}-Q_{p}^{2p}(x-d_{ij})\Big]
	=Q_{p}^{2p-2}Q_{p}^2(x-d_{ij})+R_p(x),
\end{align*}
where
\(R_p(x):=\frac{1}{p}Q_{p}^{2p}\big[(1+v)^p-1-v^p-pv\big]\).  This implies that
\begin{equation}\label{4.47}
\begin{split}
	&\quad\frac{e^{2\kappa_pd_{ij}}}{p}
	\int_{\mathbb R}
	\big[(Q_{ip}^2+Q_{jp}^2)^p-Q_{ip}^{2p}-Q_{jp}^{2p}\big]\,dx \\
	&=
	2e^{2\kappa_pd_{ij}}
	\int_{-\infty}^{d_{ij}/2}
	\big[Q_p^{2p-2}Q_p^2(x-d_{ij})+R_p\big]\,dx.
\end{split}
\end{equation}
By the mean value theorem in \(p\), we have for any \(p\in I_{1/8}\) and   $0\le v\le1$,
\begin{equation}\label{eq:599}
\partial_v^2\big[(1+v)^p-1-pv-v^2\big]
=p(p-1)(1+v)^{p-2}-2=O(|p-2|).
\end{equation}
Since the expression in the bracket of (\ref{eq:599}) and its first derivative vanish at \(v=0\),
integrating (\ref{eq:599}) twice gives that
$$\big|(1+v)^p-1-pv-v^2\big|\le C|p-2|v^2 \ \ \mbox{uniformly in }\ 0\le v\le1.$$
Consequently, we get that for any $p\in I_{1/8}$,
\begin{align*}
	|(1+v)^p-1-pv-v^p|
	\leq &|(1+v)^p-1-pv-v^2|+|\int_2^p v^s\log v\,ds|\\
	\leq& C|p-2|v^2+|p-2|\sup_{s\in I_{1/8}}v^s|\log v|\\
	\leq &C|p-2|v^{3/2}\ \ \text{on}\ \ (-\infty, d_{ij}/2]
\end{align*}
uniformly as $D(z)\to\infty$.
This shows that   for any $p\in I_{1/8}$ and $i\neq j$,
\begin{align}\label{eq:half-line-corridor-remainder}
	e^{2\kappa_p d_{ij}}\int_{-\infty}^{d_{ij}/2}|R_p(x)|\,dx
	\leq& C|p-2|e^{2\kappa_p d_{ij}}\int_{-\infty}^{d_{ij}/2} Q_{p}^{2p-3}Q_{p}^3(x-d_{ij})dx\nonumber\\
	\le& C|p-2| e^{-\kappa_p d_{ij}}\Big[
	\int_{-\infty}^0e^{2p\kappa_px}\,dx
	+\int_0^{d_{ij}/2}e^{2(3-p)\kappa_px}\,dx\Big]\\
	\le& C|p-2|\big[e^{-\kappa_p d_{ij}}+e^{-(p-2)\kappa_pd_{ij}}\big]
	\le C|p-2|e^{-(p-2)\kappa_pd_{ij}}\nonumber
\end{align}
uniformly as $D(z)\to\infty$.
Moreover, the same argument of \eqref{4.42} gives that   for any $p\in I_{1/8}$,
\begin{equation}\label{4.50}
\begin{split}
	&2e^{2\kappa_p d_{ij}}\int_{-\infty}^{d_{ij}/2}Q_{p}^{2p-2}Q_{p}^2(x-d_{ij})\,dx\\
	=&2a_p^2\int_{-\infty}^{d_{ij}/2}e^{2\kappa_px}Q_p^{2p-2}\,dx
	+O\Big(e^{-(2p-3)\kappa_p d_{ij}}\Big)\\ =&2a_p^2\int_{-\infty}^{d_{ij}/2}\big[e^{2\kappa_px}Q_p^{2p-2}-a_p^{2p-2}
e^{-2(p-2)\kappa_px}\mathbf1_{(0,\infty)}(x)\big]\,dx\\
&+a_p^{2p}d_{ij}\frac{1-e^{-t_{ij}}}{t_{ij}}	+O\Big(e^{-(2p-3)\kappa_p d_{ij}}\Big) \\ =&2a_p^2\inte\big[e^{2\kappa_px}Q_p^{2p-2}-a_p^{2p-2}e^{-2(p-2)\kappa_px}\mathbf1_{(0,\infty)}(x)\big]\,dx\\
&+a_p^{2p}d_{ij}\frac{1-e^{-t_{ij}}}{t_{ij}}	+O\Big(e^{-(2p-3)\kappa_p d_{ij}}\Big)
\end{split}
\end{equation}
uniformly   as $D(z)\to\infty$.  Similar to \eqref{4.45}, one can prove that
\begin{align*}
	\left|
	\partial_p\big[e^{2\kappa_px}Q_p^{2p-2}-a_p^{2p-2}e^{-2(p-2)\kappa_px}\mathbf1_{(0,\infty)}(x)\big]
	\right|
	\le C(1+|x|)e^{-c|x|}\ \ \text{in}\ \ \R.
\end{align*}
Since
$\int_\R\big[e^{2\kappa_2x}Q_2^{2}-a_2^{2}\mathbf1_{(0,\infty)}(x)\big]\,dx=-1$,
one gets that
$$\inte\big[e^{2\kappa_px}Q_p^{2p-2}-a_p^{2p-2}e^{-2(p-2)\kappa_px}\mathbf1_{(0,\infty)}(x)\big]\,dx=-1+O(|p-2|).$$
This further yields from \eqref{4.50} that  for any $p\in I_{1/8}$,
\begin{equation}\label{4.51}
\begin{split}
	&2e^{2\kappa_p d_{ij}}\int_{-\infty}^{d_{ij}/2}Q_{p}^{2p-2}Q_{p}^2(x-d_{ij})\,dx\\
	=&a_p^{2p}d_{ij}\frac{1-e^{-t_{ij}}}{t_{ij}}-1
	+O\big(|p-2|+e^{-(2p-3)\kappa_p d_{ij}}\big)
\end{split}
\end{equation}
uniformly as $D(z)\to\infty$.  As a consequence of \eqref{4.47},  \eqref{eq:half-line-corridor-remainder} and \eqref{4.51},  we obtain  that \eqref{eq:energy-corridor} holds true, which therefore completes the proof of  Lemma \ref{lem:reference-energy}.
\end{proof}

We are now ready to address the proof of Proposition~\ref{prop:two-sided-energy-lower}.

\begin{proof}[\normalfont\bfseries Proof of Proposition~\ref{prop:two-sided-energy-lower}]
Let  $D_0>1$ be sufficiently large and  $\delta, \varepsilon_0>0$ be sufficiently small.
Since it follows from \eqref{eq:residual-response-size} that  for any $p\in I_\delta$,
\(\|F\|_{H^{-1}}\leq C\,q_D\to0\) uniformly as $D(z)\to\infty$,   we deduce from Lemma \ref{lem:energy-taylor}  that for any $p\in I_\delta$ and $z\in\R^N$ satisfying $D(z)>D_0$,
\begin{equation}
	E_p(\cR_{\Psi_{p,z}}(K))\ge E_p(\Psi_{p,z})-\langle F,H_{\mathrm{F}}\rangle
	+\frac{c_*}{2}\|K-H_{\mathrm{F}}\|_{H^1}^2-C_Nq_D^3
	\label{eq:energy-square-separated}
\end{equation}
holds for \(K\in X_{p,z}\) satisfying \(\|K\|_{H^1}\le \varepsilon_0\).
Substituting \eqref{eq:response-energy-explicit} and \eqref{eq:reference-energy-explicit}  into
\eqref{eq:energy-square-separated}, 
we obtain that for any    $p\in I_\delta$,
\begin{equation}\label{eq:common-pair-energy-lower}
\begin{split}
	&E_p(\cR_{\Psi_{p,z}}(K))-NJ_1(p)
	\\
\ge{}&\sum_{i<j}e^{-2\kappa_p d_{ij}}
	\Big[d_{ij}\bigg(4\kappa_pa_p^4
	-a_p^{2p}\frac{1-e^{-(p-2)\kappa_pd_{ij}}}{(p-2)\kappa_pd_{ij}}\Big)+\frac12\bigg]\\
	&+\frac{c_*}{2}\|K-H_{\mathrm{F}}\|_{H^1}^2
	-C|2-p|\sum_{i<j}e^{-2\kappa_p d_{ij}}e^{\max\{2-p,0\}\kappa_pd_{ij}}
	-P(D)q_D^{19/8}
\end{split}
\end{equation}
uniformly as $D(z)\to\infty$. Since $\kappa_p, a_p\in C^1$, \(\kappa_2=1/4\) and \(a_2=1/\sqrt2\), we have
\[
\big|4\kappa_pa_p^4-1/4\big|
+\left|a_p^{2p}-1/4\right|\le C|2-p|,\ \ \forall\ p\in I_\delta.
\]
By the fact that
\[
0<\frac{1-e^{-t}}t=\int_0^1e^{-st}\,ds
\le e^{\max\{-t,0\}},\ \ \forall\ t\in\R,
\]
we then conclude that
\begin{align}\label{4.54}
	\Big|4\kappa_pa_p^4-a_p^{2p}\frac{1-e^{-(p-2)\kappa_pd_{ij}}}{(p-2)\kappa_pd_{ij}}
	-\frac14\Big(1-\frac{1-e^{-(p-2)\kappa_pd_{ij}}}{(p-2)\kappa_pd_{ij}}\Big)\Big|\le C|2-p|e^{\max\{2-p,0\}\kappa_pd_{ij}}.
\end{align}
Note that  for any sufficiently large  $D(z)>0$,
\begin{align}\label{j1}
	\arraycolsep=1.5pt
	\left\{\begin{array}{lll}
		\frac12\sum_{i<j}q_{ij}^2-P(D)q_D^{19/8}\ge\frac14\sum_{i<j}q_{ij}^2,\\[3mm]
		\sum_{i<j}d_{ij}e^{-2\kappa_p d_{ij}}
		e^{\max\{2-p,0\}\kappa_pd_{ij}}\leq
		\frac{N(N-1)}2\,Dq_D^2
		e^{\max\{2-p,0\}\kappa_pD},
	\end{array}\right.
\end{align}
where  the second inequality follows from the fact that the function
\(d\mapsto de^{-\min\{p,2\}\kappa_pd}\)
is decreasing for sufficiently large \(d>0\).
Substituting \eqref{4.54} and \eqref{j1} into \eqref{eq:common-pair-energy-lower}, we hence conclude that \eqref{eq:uniform-energy-lower} holds true. This therefore completes the proof of Proposition  \ref{prop:two-sided-energy-lower}. \end{proof}

\section{Existence and nonexistence of minimizers  as $p\to2$}
\label{sec:left-existence}
This section is concerned with the complete proof of  Theorem~\ref{thm:main}. Towards this purpose,  suppose that  $ U_n=(u_{1n},\cdots, u_{Nn})\in\cM_N$ is a minimizer of $J_N(p_n)$, where  $p_n\to2$ as $n\to\infty$. Recall from Lemmas  \ref{prop:profile} and \ref{lem:center} that up to a subsequence if necessary, there exist
\begin{align}\label{eq:s7-parameters}
	z_n=(z_{1n},\cdots,z_{Nn})\in\mathbb R^N,\ \ K_n\in X_{p_n,z_n},\ \  R_n\in O(N),
\end{align}
such that
\begin{align}\label{eq:s7-minimizer-chart}
	&R_n^\top U_n=\cR_{\Psi_{p_n,z_n}}(K_n):=
	\left[I_N+\mathsf M(K_n,K_n)
	\right]^{-1/2}\left(\Psi_{p_n,z_n}+K_n\right),\\[2mm]
	&z_{1n}<\cdots<z_{Nn},\ \  \
	\lim\limits_{n\to\infty}D(z_n)=\infty,\ \ \
	\lim\limits_{n\to\infty}\|K_n\|_{H^1}=0.\label{eq:s7-separated-limit}
\end{align}
Indeed, denote $\tilde{U}_n:=(R_n^0)^\top U_n$, where $R_n^0$ is as in \eqref{3.2}.  Apply  Lemma~\ref{lem:center} with $p=p_n, \ z^0=z_n^0$ and  $U=\tilde{U}_n$.  We then derive  from  Lemmas  \ref{prop:profile} and \ref{lem:center} that \eqref{eq:s7-parameters}--\eqref{eq:s7-separated-limit} hold true.
We first prove the nonexistence of minimizers for $J_N(p)$ with $p\in[2,2+\delta_N)$.

\begin{proof}[\normalfont\bfseries Proof of Theorem~\ref{thm:main}~(ii).]
Since the nonexistence of the case $p=2$ follows directly from \cite[Theorem~7]{ii}, we only need to prove the assertion for the case $2<p<3$.

Fix sufficiently large $C>0$, and  choose $0<\delta<\min\{1/8,\, (8eC)^{-1}\}$  and $D_0>16eC\kappa_-^{-1}$, where $\inf\limits_{p\in I_{1/8}}\kappa_p\geq\kappa_->0$.  
Note that
\begin{align*}
	\arraycolsep=1.5pt
	1-\frac{1-e^{-t}}t=\int_0^1(1-e^{-st})\,ds\geq
	\left\{\begin{array}{lll}
		\int_0^1e^{-1}st\,ds= \frac12e^{-1}t,\ \ \text{if}\ \ 0\leq t\leq1;\\[3mm]
		\int_0^1(1-e^{-s})\,ds= e^{-1},\ \ \text{if}\ \ t\geq 1;
	\end{array}\right.
\end{align*}
where the first line follows from the mean value theorem.
We then have
\begin{equation}
	\frac14+\frac d4\Big(1-\frac{1-e^{-(p-2)\kappa_pd}}{(p-2)\kappa_pd}\Big)-C(p-2)d
	\ge\frac14,\ \ \forall\ 2\leq p\leq2+\delta,\,\ d\geq D_0.
	\label{eq:positive-pair-normalized}
\end{equation}
Since \(Dq_D^2\le\sum_{i<j}d_{ij}e^{-2\kappa_p d_{ij}}\), it follows from \eqref{eq:uniform-energy-lower} and \eqref{eq:positive-pair-normalized} that there exist sufficiently small $\delta>0$ and $\varepsilon_0>0$, and  sufficiently large $D_0>0$ such that for any $p\in[2,2+\delta]$, $z\in\R^N$ satisfying $D(z)>D_0$, and $K\in X_{p,z}$ satisfying $\|K\|_{H^1}\le\varepsilon_0$,
\begin{equation}\label{eq:positive-pair-energy}
	E_p(\cR_{\Psi_{p,z}}(K))-NJ_1(p)
	\ge\frac14\sum_{i<j}e^{-2\kappa_pd_{ij}}
	+\frac{c_*}{2}\|K-H_{\mathrm{F}}\|_{H^1}^2>0.
\end{equation}
On the contrary, suppose  the nonexistence assertion of Theorem~\ref{thm:main} fails as $p\searrow 2$. Then there exists a sequence $\{p_n\}$, where  $p_n\searrow2$ as $n\to\infty$, such that $J_N(p_n)$ admits a minimizer \(U_n\in\cM_N\) for any large $n>0$.  Applying Lemma~\ref{prop:profile}, we then deduce from \eqref{eq:s7-parameters}--\eqref{eq:s7-separated-limit} that there exists $(z_n, R_n, K_n)\in \R^N\times O(N)\times  X_{p_n,z_n}$ such that
\begin{equation}\label{5.3}
	R_n^\top U_n=\cR_{\Psi_{p_n,z_n}}(K_n), \ \
	D(z_n)\to\infty,\ \ \|K_n\|_{H^1}\to0\ \ \text{as}\ \ n\to\infty.
\end{equation}
By the orthogonal
invariance  of $E_{p_n}$, we conclude from \eqref{eq:positive-pair-energy} and \eqref{5.3}   that for sufficiently large \(n>0\),
$$
J_N(p_n)=E_{p_n}(R_n^\top U_n)=E_{p_n}(\cR_{\Psi_{p_n,z_n}}(K_n))>NJ_1(p_n),
$$
which however contradicts \eqref{eq:s2-8}. Thus, there exists $\delta_N^+>0$ such that $J_N(p)$ is not attained for $p\in (2, 2+\delta_N^+)$. This proves  Theorem~\ref{thm:main}~(ii).
\end{proof}

\subsection{Proof of Theorem~\ref{thm:main}~(i)}
This subsection is devoted to the proof of Theorem~\ref{thm:main}~(i), which is concerned with the existence of minimizers $U_p$ for $J_N(p)$ as $p\nearrow 2$. Towards this aim, we always choose $15/8<p<2$ sufficiently close to $2$, so that $1/8\le\kappa_p\le1/2$.

The following lemma shows that the Lagrange multiplier matrix $\Lambda_p$ of the minimizer $U_p$ for $J_N(p)$ can be comparable with \(-\kappa_p^2I_N\), which is precisely needed for the spectral binding condition.

\begin{lemma}\label{prop:left-multiplier-clustering}
For any $N\in \mathbb{N}^+$, suppose  $J_N(p)$ admits a minimizer \(U_p\in\cM_N\) as \(p \nearrow 2\).
Then there exists \(\delta_N>0\)
such that for any \(p\in[2-\delta_N,2)\), the associated symmetric Lagrange multiplier matrix
\(\Lambda_p\) of  the minimizer \(U_p\) satisfies
\begin{equation}\label{eq:left-exponential-clustering}
	\|\Lambda_p+\kappa_p^2I_N\|_{\mathrm F}
	\le C_N e^{-b_N/\sqrt{2-p}},
\end{equation}
where the constants \( C_N>0\) and \(b_N>0\) are independent of \(p\) and \(U_p\).
\end{lemma}

\begin{proof}
For \( N = 1 \), the conclusion (\ref{eq:left-exponential-clustering}) follows  directly from the fact that the associated  Lagrange multiplier of the minimizer is \(-\kappa_p^2\).

We now assume that $N\ge2$. Fix  sufficiently small  \(\delta,\varepsilon_0>0\), and  sufficiently large \(D_0>1\), so that Proposition \ref{prop:two-sided-energy-lower} and Lemma \ref{lem:energy-taylor} can be applied.
By Lemmas \ref{prop:profile} and \ref{lem:center},
a standard sequential argument yields that there exists $\delta_N\in(0,\delta]$
such that  for each $p\in[2-\delta_N,2)$, every minimizer $U_p$
admits a representation satisfying
\begin{align}\label{7.8}
	R_p^\top U_p=\cR_{\Psi_{p,z_p}}(K_p),\ \
	D(z_p)>D_0,\ \  \|K_p\|_{H^1}\le\varepsilon_0,
\end{align}
where \((R_p,z_p,K_p)\) may depend on \(p\) and \(U_p\).


Since \(\kappa_p\to1/4\) as \(p\to2\), one can decrease \(\delta_N>0\)
so that \(1/8\le\kappa_p\le1/2\) on $[2-\delta_N,2)$.  Because
\[
0\le\frac{e^t-1}{t}-1
=\int_0^1(e^{st}-1)\,ds
\le\int_0^1ste^t\,ds=\frac12te^t, \ \ \forall\ t\ge 0,
\]
we have
\begin{align}\label{eq:left-pair-lower}
	&e^{-2\kappa_pd}\bigg[\frac14+\frac {d}4\Big(1-\frac{e^{(2-p)\kappa_pd}-1}{(2-p)\kappa_pd}\Big)
	-C(2-p)d e^{(2-p)\kappa_pd}\bigg]\nonumber\\
	\ge&\frac14e^{-2\kappa_pd}-\left(C+\frac{\kappa_p}{8}\right)
	(2-p)d^2e^{-p\kappa_pd}\\
\ge &\frac14e^{-2\kappa_pd}-2C
	(2-p)d^2e^{-p\kappa_pd},\ \ \forall\ p\in[2-\delta_N,2), \ d>1,\nonumber
\end{align}
where the constant  $C>1$ is fixed. Since
$$D(z_p)e^{-p\kappa_pD(z_p)}\le\sum_{i<j}d_{ij}e^{-p\kappa_pd_{ij}},$$
we conclude from \eqref{eq:uniform-energy-lower}  that for any  $p\in[2-\delta_N,2)$,
\begin{equation}
	\begin{aligned}
		&E_p\big(\cR_{\Psi_{p,z_p}}(K_p)\big)-NJ_1(p)\\
		\ge&\frac14\sum_{i<j}e^{-2\kappa_pd_{ij}}
		-2C(2-p)\sum_{i<j}d_{ij}^2e^{-p\kappa_pd_{ij}}+\frac{c_*}{2}\|K_p-H_{\mathrm F}\|_{H^1}^2,
	\end{aligned}
	\label{eq:left-all-pairs-lower}
\end{equation}
where \(F:=-\cF_{\kappa_p}(\Psi_{p,z_p})\) and \(H_{\mathrm{F}}\) are as in \eqref{f} and \eqref{eq:energy-linear-response}, respectively. Since \(D_0\) is chosen sufficiently large, we get that  for any
\(p\in[2-\delta_N,2)\), the function \(d\mapsto d^2e^{-p\kappa_pd}\)
is decreasing on \([D_0,\infty)\).  This yields that
$$\sum_{i<j}d_{ij}^2e^{-p\kappa_pd_{ij}}
\le\frac{N(N-1)}{2}D^2(z_p)e^{-p\kappa_pD(z_p)}.$$
We thus  deduce from  \eqref{eq:left-all-pairs-lower} that   for any  $p\in[2-\delta_N,2)$,
\begin{equation}\label{eq:left-minimum-gap-lower}
	\begin{aligned}
		&E_p\big(\cR_{\Psi_{p,z_p}}(K_p)\big)-NJ_1(p)\\
		\ge&q_{D(z_p)}^2\Big[\frac14-C_N(2-p)D^2(z_p)e^{(2-p)\kappa_pD(z_p)}\Big]
		+\frac{c_*}{2}\|K_p-H_{\mathrm{F}}\|_{H^1}^2.
	\end{aligned}
\end{equation}

Choose sufficiently small \(\eta_N > 0\),  so that
\(C_N\eta_N e^{\eta_N/2}\le1/8\).  We claim that
\begin{equation}\label{eq:left-separation-capture}
	D(z_p)>\sqrt{\eta_N/(2-p)},\ \ \forall\ p\in[2-\delta_N,2).
\end{equation}
On the contrary, suppose that \((2-p)D^2(z_p)\le\eta_N\) holds for some $p\in[2-\delta_N,2)$. We then have \((2-p)\kappa_pD(z_p)\le\eta_N/2\). This yields from  \eqref{7.8} and  \eqref{eq:left-minimum-gap-lower}  that
\begin{align*}
	J_N(p)-NJ_1(p)=&E_p(U_p)-NJ_1(p)=E_p(R^\top_pU_p)-NJ_1(p)\\
	=&E_p\big(\cR_{\Psi_{p,z_p}}(K_p)\big)-NJ_1(p)\ge\frac18q_{D(z_p)}^2
	+\frac{c_*}{2}\|K_p-H_{\mathrm{F}}\|_{H^1}^2>0,
\end{align*}
which however contradicts with the fact that $J_N(p)\leq NJ_1(p)$ in view of \cite[Lemma 12 (ii)]{GLN}.  Thus, the claim \eqref{eq:left-separation-capture} holds true.

Since \(\Psi_{p,z_p}\in\cM_N\) and \(R^\top_pU_p=\cR_{\Psi_{p,z_p}}(K_p)\) is a minimizer of $J_N(p)$,  we obtain from  Lemma \ref{lem:energy-taylor} and  \eqref{eq:residual-response-size} that  for any $p\in[2-\delta_N, 2)$,
\begin{equation}\label{eq:left-correction-energy}
	\begin{aligned}
		\frac{c_*}{2}\|K_p-H_{\mathrm{F}}\|_{H^1}^2
		&\le E_p(R^\top_pU_p)-E_p(\Psi_{p,z_p})+\langle F,H_{\mathrm{F}}\rangle+C_Nq_{D(z_p)}^3\\
		&\le\|F\|_{H^{-1}}\|H_{\mathrm{F}}\|_{H^1}+C_Nq_{D(z_p)}^3
		\le C_Nq_{D(z_p)}^2,
	\end{aligned}
\end{equation}
which  implies that
$$
\|K_p\|_{H^1}\le\|K_p-H_{\mathrm{F}}\|_{H^1}+\|H_{\mathrm{F}}\|_{H^1}\le C_Nq_{D(z_p)}.
$$
We then derive from  \eqref{2.72} that
$$\|R^\top_pU_p-\Psi_{p,z_p}\|_{H^1}\le C_Nq_{D(z_p)}, \ \ p\in[2-\delta_N, 2).$$
This further yields from \eqref{eq:residual-response-size} that  for any $p\in[2-\delta_N, 2)$,
\begin{align}\label{7.14}
	\|\cF_{\kappa_p}(R_p^\top U_p)\|_{H^{-1}}
	\le\|\cF_{\kappa_p}(\Psi_{p,z_p})\|_{H^{-1}}
	+C_N\|R_p^\top U_p-\Psi_{p,z_p}\|_{H^1}\le C_Nq_{D(z_p)},
\end{align}
where the first inequality follows from the $L^\infty$-uniform boundedness  of $\{\Psi_{p, z_p}\}$ and $\{R_p^\top U_p\}$.
Let $\Lambda_p$ be the  associated multiplier matrix of $U_p$. By  the definition of $\cF_{\kappa_p}$, we then obtain from \eqref{eq:matrix-euler-lagrange} that
$$
\cF_{\kappa_p}(R_p^\top U_p)
=(R_p^\top\Lambda_pR_p+\kappa_p^2I_N)R_p^\top U_p,
$$
which further yields that
\begin{equation}\label{eq:left-multiplier-pairing}
	(R_p^\top\Lambda_pR_p+\kappa_p^2I_N)_{ij}
	=\big\langle\big(\cF_{\kappa_p}(R_p^\top U_p)\big)_i,\, (R_p^\top U_p)_j\big\rangle_{H^{-1},H^1},\ \ i,j=1, \cdots, N.
\end{equation}
It thus follows from \eqref{7.14} and \eqref{eq:left-multiplier-pairing} that for any $p\in[2-\delta_N, 2)$,
\begin{align}\label{7.17}
	\|\Lambda_p+\kappa_p^2I_N\|_{\mathrm F}
	&=\|R_p^\top\Lambda_pR_p+\kappa_p^2I_N\|_{\mathrm F}
	\le C_N\|\cF_{\kappa_p}(R_p^\top U_p)\|_{H^{-1}}
	\le C_Nq_{D(z_p)},
\end{align}
where we have used the $H^1$-uniform boundedness  of $\{R_p^\top U_p\}$. Finally,   since $\inf_{p\in[2-\delta_N, 2)}\kappa_p\ge1/8$,  one gets  from \eqref{eq:left-separation-capture} that
$
q_{D(z_p)}=e^{-\kappa_pD(z_p)}
\le e^{-\sqrt{\eta_N}/(8\sqrt{2-p})}
$ holds for any $p\in[2-\delta_N, 2)$. This implies from  \eqref{7.17} that  \eqref{eq:left-exponential-clustering} holds true, and the proof of Lemma \ref{prop:left-multiplier-clustering} is therefore complete.
\end{proof}

We next introduce  the following analytical properties of the minimizers for $J_N(p)$.

\begin{lemma}\label{prop:gln-input}
Let $1<p<3$, and $N, k,\ell\in\mathbb{N}^+$.
\begin{enumerate}
\item[(i).]  If $J_N(p)$ admits a minimizer, then there exists
a minimizer $U=(u_1,\cdots,u_N)^\top\in\cM_N$ such that
$u_1,\cdots,u_N$ are the first $N$ eigenfunctions associated with
the first $N$ negative eigenvalues for the operator
$-\partial_{xx}-|U|^{2p-2}$ in $\R$.

\item[(ii).]
Let $U\in\cM_k$ and $\tilde U\in\cM_\ell$ be a minimizer of
$J_k(p)$ and $J_\ell(p)$ satisfying (i), respectively.
Let $\mu_k<0$ and $\tilde\mu_\ell<0$ be the $k$-th  eigenvalue of the operator
$-\partial_{xx}-|U|^{2p-2}$ in $\R$ and
the $\ell$-th eigenvalue of
$-\partial_{xx}-|\tilde U|^{2p-2}$ in $\R$, respectively.
If
$
0<p-1<
\sqrt{\frac{\min\{|\mu_k|,|\tilde\mu_\ell|\}}
	{\max\{|\mu_k|,|\tilde\mu_\ell|\}}},
$
then we have
$$J_{k+\ell}(p)<J_k(p)+J_\ell(p), \ \ \forall \,k,\ell\in\mathbb{N}^+.$$

\item[(iii).]
If $N\ge2$ and $J_N(p)$ satisfies
$$J_N(p)<J_k(p)+J_{N-k}(p),\ \ k=1, \cdots, N-1,$$
then $J_N(p)$ admits a minimizer.
\end{enumerate}
\end{lemma}

Since the proof of Lemma \ref{prop:gln-input} can follow essentially from
\cite[Propositions~16 and~20, Theorem~14]{GLN}, we omit its details for simplicity.
We remark that the results of \cite{GLN} are addressed for complex-valued orthonormal
functions. In the present setting, since the associated Schr\"odinger
operators have real potentials,   their occupied eigenfunctions may be
chosen in the real range. The proofs of the binding and compactness criteria therefore
apply equally to the real constraint class \(\cM_N\), and hence we shall use them
in this form below.

\begin{proof}[\normalfont\bfseries Proof of Theorem~\ref{thm:main} (i)]
We prove   Theorem~\ref{thm:main} (i) by induction on $N\in\mathbb{N}^+$.
For $N=1$, $Q_p$ is a minimizer of $J_1(p)$ for any $p\in I_{1/8}$, whose Lagrange multiplier is $-\kappa_p^2$. For $N\ge2$, suppose  $J_k(p)$ admits a minimizer for every $1\le k<N$, where $1<p<2$ is sufficiently close to $2$.
Choose  sufficiently small $\delta_N>0$, so that $J_k(p)$ admits the minimizers
 for any $p\in[2-\delta_N,2)$ and the estimate of
Lemma \ref{prop:left-multiplier-clustering} holds, where
$1\le k<N$ is arbitrary.


For each $1\le k<N$, let $U$ and $\tilde U$ be the minimizers of
$J_k(p)$ and $J_{N-k}(p)$ satisfying
Lemma \ref{prop:gln-input} (i), respectively, and suppose that  $\mu_k<0$ and $\tilde\mu_{N-k}<0$ is the $k$-th  eigenvalue  of the operator
$-\partial_{xx}-|U|^{2p-2}$ in $\R$ and the
$(N-k)$-th eigenvalue of the operator
$-\partial_{xx}-|\tilde U|^{2p-2}$ in $\R$, respectively.
It then follows from
Lemma \ref{prop:left-multiplier-clustering}  that
\begin{equation}\label{7.17a}
	\big||\mu_k|-|\tilde\mu_{N-k}|\big|\leq\big||\mu_k|-\kappa_p^2\big|
	+\big||\tilde\mu_{N-k}|-\kappa_p^2\big|
	\le C_N e^{-b_N/\sqrt{2-p}},\ \ \forall\ p\in[2-\delta_N,2).
\end{equation}
Since \(\kappa_p\ge1/8\) holds for any $p\in[2-\delta_N,2)$, by decreasing $\delta_N>0$ if necessary, we obtain from \eqref{7.17a} that
\begin{equation}\label{eq:left-binding-ratio}
	\sqrt{\frac{\min\{|\mu_k|,|\tilde\mu_{N-k}|\}}{\max\{|\mu_k|,|\tilde\mu_{N-k}|\}}}
	\ge1-C_N e^{-b_N/\sqrt{2-p}}
	>1-(2-p)=p-1,\ \ \forall\ p\in[2-\delta_N,2),
\end{equation}
where the strict inequality follows from the fact that $e^{-b_N/\sqrt{2-p}}=o(2-p)$ as $p\nearrow2$.
We then deduce from Lemma \ref{prop:gln-input} (ii)  that  for any $p\in[2-\delta_N,2),$
$$J_N(p)<J_k(p)+J_{N-k}(p),\ \ 1\le k<N,$$
which further yields from Lemma \ref{prop:gln-input} (iii) that $J_N(p)$ admits a minimizer   for any $p\in[2-\delta_N,2)$.
By induction, this therefore completes the proof of Theorem~\ref{thm:main} (i).
\end{proof}

\appendix
\section{Auxiliary estimates}\label{app:technical-estimates}

In this appendix, we prove some auxiliary estimates which are used in Sections~\ref{sec:profile-reduction} and~\ref{sec:energy-conclusion}.

\begin{proof}[\normalfont\bfseries Proof of  \eqref{eq:reference-center-derivatives} and \eqref{2.35}]
Since $0<\inf\limits_{p\in I_{1/8}}\kappa_p\leq\kappa_p\leq\sup\limits_{p\in I_{1/8}}\kappa_p<\infty$ on $I_{1/8}$,   the explicit formula of  $Q_p$ gives that  for any $p\in I_{1/8}$,
\begin{align}\label{a.1}
	\sum_{j=0}^3|Q_p^{(j)}(x)|
	\leq C e^{-\kappa_p|x|}\ \,\ \text{in}\ \ \R, 
\end{align}
where we denote the derivative $Q_p^{(j)}(x):=\frac{d^j}{dx^j}Q_p(x)$.
Denote $E_{p,z}:=G_{p,z}-I_N$. One can verify from  \eqref{a.1} that  for any $p\in I_{1/8}$ and  $z\in\R^N$,
\begin{equation}\label{a.2a}
\begin{split}
	&\big|\inte Q_p(x-z_i)Q_p(x-z_j)dx\big|+\big|\inte Q_p(x-z_i)Q'_p(x-z_j)dx\big|\\
&\le C\big(1+D(z)\big)e^{-\kappa_pD(z)},\ \ i\neq j,
\end{split}
\end{equation}
 and hence
\begin{align}\label{a.2}
	\|E_{p,z}\|_{\mathrm{F}}+\|\partial_{z_j}E_{p,z}\|_{\mathrm{F}}\le C\big(1+D(z)\big)e^{-\kappa_pD(z)}.
\end{align}
Here and below in the proof,  \( C > 0 \) is independent of \((z, p)\), but it may vary from line to line.

Applying (\ref{a.2}), we can choose \(D_*>1\) sufficiently large so that for any $p\in I_{1/8}$ and  $z\in\R^N$ satisfying $D(z)\ge D_*$,
\begin{align}\label{a.3a}
\|E_{p, z}\|_{\mathrm F}+\|\partial_{z_j}E_{p,z}\|_{\mathrm{F}}<1/2,\ \ j=1, \cdots,N.
\end{align}
The matrix Taylor expansion then  gives that   for any $p\in I_{1/8}$ and   \(D(z)\ge D_*\),
\begin{align}\label{a.3}
	G_{p,z}^{-1/2}
	=\big[I_N+E_{p,z}\big]^{-1/2}=
	I_N-\frac12E_{p,z}
	+\sum_{m=2}^{\infty}c_mE_{p,z}^m,\ \ c_m:=\binom{-1/2}{m},
\end{align}
where \(\binom{\alpha}{m}\) denotes the
generalized binomial coefficient
\[
\binom{\alpha}{m}
:=\frac{\alpha(\alpha-1)\cdots(\alpha-m+1)}{m!},
\ \ \alpha\in\R,\quad m=1,2,\cdots.
\]
Differentiating \eqref{a.3} gives that  for any $p\in I_{1/8}$ and   \(D(z)\ge D_*\),
\begin{align}\label{a.4}
	\partial_{z_j}(G_{p,z}^{-1/2})
	=&-\frac12\partial_{z_j}E_{p,z}
	+
	\sum_{m=2}^{\infty}
	c_m
	\sum_{\ell=0}^{m-1}
	E_{p,z}^\ell
	(\partial_{z_j}E_{p,z})
	E_{p,z}^{m-1-\ell},\ \ j=1, \cdots, N.
\end{align}
Note from \eqref{a.3a} that  for any $p\in I_{1/8}$ and   \(D(z)\ge D_*\),
\begin{align}\label{a.5}
	\Big\|
	\sum_{m=2}^{\infty}c_mE_{p,z}^m
	\Big\|_{\mathrm F}
	\leq
	C\|E_{p,z}\|_{\mathrm F}^2,
\end{align}
and
\begin{align}\label{a.6}
	&\left\|
	\sum_{m=2}^{\infty}
	c_m\Big(\sum_{\ell=0}^{m-1}
	E_{p,z}^\ell(\partial_{z_j}E_{p,z})E_{p,z}^{m-1-\ell}\Big)
	\right\|_{\mathrm F}\nonumber\\
	\leq &\sum_{m=2}^{\infty}
	|c_m|\Big(
	\sum_{\ell=0}^{m-1}
	\|E_{p,z}\|_{\mathrm F}^{\ell}
	\|\partial_{z_j}E_{p,z}\|_{\mathrm F}
	\|E_{p,z}\|_{\mathrm F}^{m-1-\ell}\Big)
	\\
	=&\sum_{m=2}^{\infty}
	m|c_m|
	\|E_{p,z}\|_{\mathrm F}^{m-1}
	\|\partial_{z_j}E_{p,z}\|_{\mathrm F}\leq C\|\partial_{z_j}E_{p,z}\|_{\mathrm F}.\nonumber
\end{align}
As a consequence of \eqref{a.2} and \eqref{a.3}--\eqref{a.6},  we conclude that   for any $p\in I_{1/8}$ and \(D(z)>D_*\),
\begin{align}\label{a.7}
	\|G_{p,z}^{-1/2}-I_N\|_{\mathrm{F}}+\|\partial_{z_j}(G_{p,z}^{-1/2})\|_{\mathrm{F}}\leq CD(z)e^{-\kappa_pD(z)},\ \ j=1, \cdots, N.
\end{align}

Since it follows from  \eqref{a.1} that $\{Q_p\}_{p\in I_{1/8}}$ and $\{Q_p'\}_{p\in I_{1/8}}$ are bounded uniformly in $H^1(\R)$,  we conclude  from \eqref{a.7} that for any $p\in I_{1/8}$ and \(D(z)>D_*\),
\begin{equation}\label{a.10a}
\big\|\Psi_{p,z}-W_{p,z}\big\|_{H^1}=\big\|(G_{p,z}^{-1/2}-I_N)W_{p,z}\big\|_{H^1}\le CD(z)e^{-\kappa_pD(z)}, \qquad
\end{equation}
\begin{equation}\label{a.11}
\begin{split}
&\big\|\partial_{z_j}\Psi_{p,z}+e_jQ_p'(\cdot-z_j)\big\|_{H^1}
=\big\|(\partial_{z_j}G_{p,z}^{-1/2})W_{p,z}
-(G_{p,z}^{-1/2}-I_N)e_jQ_{jp}'\big\|_{H^1}\\
&\ \ \ \  \ \ \ \ \ \  \ \  \ \ \ \ \ \  \ \ \ \ \ \  \ \   \ \ \ \ \ \ \ \  \ \  \le CD(z)e^{-\kappa_pD(z)},\ \ j=1, \cdots, N. \end{split}
\end{equation}
Note that
\begin{align}\label{a.12a}
\arraycolsep=1.0pt
\left\{\begin{array}{lll}
T_j+e_jQ_{jp}'=\partial_{z_j}\Psi_{p,z}+e_jQ_p'(\cdot-z_j)-\mathsf M(\partial_{z_j}\Psi_{p,z},\Psi_{p,z})\Psi_{p,z},\\
\mathsf M(\partial_{z_j}\Psi_{p,z},\Psi_{p,z})
=(\partial_{z_j}G_{p,z}^{-1/2})G_{p,z}^{1/2}
-G_{p,z}^{-1/2}\,\mathsf M(e_jQ'_{jp},W_{p,z})G_{p,z}^{-1/2}.
\end{array}\right.
\end{align}
We then deduce from \eqref{a.7}, \eqref{a.11} and \eqref{a.12a} that   for any $p\in I_{1/8}$ and \(D(z)>D_*\),
\begin{align}\label{a.14}
\|T_j+e_jQ_{jp}'\|_{H^1}\le CD(z)e^{-\kappa_pD(z)}+
C \|\mathsf M(e_jQ'_{jp},W_{p,z})\|_{\text{F}},\ \ j=1, \cdots, N,
\end{align}
Here we have also used the uniform boundedness of \(\{\|G_{p,z}^{1/2}\|_{\mathrm F}\}\), \(\{\|G_{p,z}^{-1/2}\|_{\mathrm F}\}\), and \(\{\|\Psi_{p,z}\|_{H^1}\}\). Moreover,
since \(\langle Q_p',Q_p\rangle=0\),  the $j$th row of the matrix \(\mathsf M(e_jQ'_{jp},W_{p,z})\)  is \((\langle Q'_{jp},Q_{\ell p}\rangle_{L^2})_{\ell=1}^N\),  and its all other elements are zero, we obtain  from \eqref{a.2a} that  for any $p\in I_{1/8}$ and \(D(z)>D_*\),
$$
\|\mathsf M(e_jQ'_{jp},W_{p,z})\|_{\text{F}}\le CD(z)e^{-\kappa_pD(z)},\ \ j=1, \cdots, N.
$$
which thus yields from \eqref{a.10a} and  \eqref{a.14}  that   for any $p\in I_{1/8}$ and \(D(z)>D_*\),
\begin{align}\label{eq:s4-full-background}
	\|\Psi_{p,z}-W_{p,z}\|_{H^1}+\|T_j(p, z)+e_j Q'_{jp}\|_{H^1}\leq CD(z)e^{-\kappa_pD(z)},\ \ j=1, \cdots, N.
\end{align}
This  proves \eqref{2.35}.

Finally, using \eqref{a.1} and differentiating \eqref{a.4} once more,
the same arguments of proving \eqref{a.2} and \eqref{a.6} then  give that  for any $p\in I_{1/8}$ and \(D(z)>D_*\),
\[
\|\partial_{z_i}\partial_{z_j}G_{p,z}\|_{\mathrm F}
+\|\partial_{z_i}\partial_{z_j}(G_{p,z}^{-1/2})\|_{\mathrm F}
\le CD(z)e^{-\kappa_pD(z)},\quad i,j=1,\cdots,N.
\]
Combining this with \eqref{a.1}, \eqref{a.7}, and
$\Psi_{p,z}=G_{p,z}^{-1/2}W_{p,z}$, we obtain that
\begin{equation}\label{a.10}
	\sup_{\substack{p\in I_{1/8}\\D(z)\ge D_*}}
	\sum_{|\alpha|\le2}\|\partial_z^\alpha\Psi_{p,z}\|_{H^1}
	\le C.
\end{equation}
This therefore completes the proof of  \eqref{eq:reference-center-derivatives}.
\end{proof}

\begin{lemma}\label{lem:two-sided-power-bounds}
For every \(M>0\), there exists  a constant \(C_M>0\)
such that for any \(\alpha\in[7/8,9/8]\) and \(0\le A,B\le M\),
\begin{equation}\label{eq:two-sided-power-bounds}
 	|(A+B)^\alpha-A^\alpha|\le C_MB^{7/8},\ \,\
	|(A+B)^\alpha-A^\alpha-B^\alpha|
	\le C_M\min(A,B)^{7/8}.
\end{equation}
\end{lemma}

\begin{proof}
If \(B=0\), then the first inequality of (\ref{eq:two-sided-power-bounds}) holds  immediately. We now suppose \(B>0\).
If  \(7/8\le\alpha\le1\), then
\[
0\le(A+B)^\alpha-A^\alpha
=\alpha\int_0^B(A+s)^{\alpha-1}\,ds
\le\alpha\int_0^B s^{\alpha-1}\,ds=B^\alpha.
\]
If \(1\le\alpha\le9/8\),  then the mean value theorem gives that
\[
0\le(A+B)^\alpha-A^\alpha
\le\alpha B(A+B)^{\alpha-1}\le C_MB.
\]
Since
\begin{align}\label{a.12}
B^\alpha=B^{7/8}B^{\alpha-7/8}
\le\max\{1,M^{1/4}\}B^{7/8},\ \
B=B^{7/8}B^{1/8}
\le\max\{1,M^{1/8}\}B^{7/8},
\end{align}
we conclude from above that the first inequality of \eqref{eq:two-sided-power-bounds} holds true.

We next prove the second inequality of \eqref{eq:two-sided-power-bounds}. If \(0<B\le A\),  then it follows from \eqref{a.12} that
\[
|(A+B)^\alpha-A^\alpha-B^\alpha|
\le |(A+B)^\alpha-A^\alpha|+B^\alpha
\le C_MB^{7/8},
\]
where the first inequality of \eqref{eq:two-sided-power-bounds} is also used.
Similarly, if \(0<A\le B\), then we have
$$|(A+B)^\alpha-A^\alpha-B^\alpha|
\le C_MA^{7/8}.$$
This therefore proves the second inequality of \eqref{eq:two-sided-power-bounds}, and we are done.
\end{proof}

\begin{lemma}\label{lem:three-label-energy}
Let \(2\le N\in\mathbb{N}^+\). Then  there exists a constant $c>0$ such that for any \(p\in I_{1/8}\) and $z\in\R^N$ satisfying $D(z)>1$,
\begin{equation}
	\begin{aligned}
		\left|\int_\R\bigg[
		\Big(\sum_{j=1}^NQ_{jp}^2\Big)^p
		-\sum_{j<k}(Q_{jp}^2+Q_{kp}^2)^p
		+(N-2)\sum_{j=1}^NQ_{jp}^{2p}\bigg]dx\right|
		\le ce^{-\frac{5}{2}\kappa_pD(z)},
	\end{aligned}
	\label{eq:three-label-energy}
\end{equation}
\end{lemma}

\begin{proof}
If \(N=2\), then the left-hand side of \eqref{eq:three-label-energy} vanishes for any \(p\in I_{1/8}\), and hence (\ref{eq:three-label-energy}) holds true.

We next focus on the case \(N\ge3\). Set \(f_p(s):=s^p\) and write
\[
\begin{aligned}
	\Delta_3f_p(A,B,C):=&f_p(A+B+C)-f_p(A+B)-f_p(A+C)-f_p(B+C)\\
	&+f_p(A)+f_p(B)+f_p(C),\ \ A, B, C>0.
\end{aligned}
\]
We then obtain that for any  \(p\in I_{1/8}\),
\begin{equation}\label{AA:5}
\Delta_3f_p(A,B,C)
=p(p-1)(p-2)\int_0^A\!\int_0^B\!\int_0^C
(s+t+u)^{p-3}\,du\,dt\,ds.
\end{equation}
Since $p<3$, the arithmetic--geometric mean inequality gives that
$$
(s+t+u)^{p-3}
\le 3^{p-3}(stu)^{p/3-1},\ \ \forall\ s>0, \ t>0, \ u>0,
$$
which further  implies from (\ref{AA:5}) that for any  \(p\in I_{1/8}\),
\begin{equation}\label{AA:6}
	\begin{aligned}
		|\Delta_3f_p(A,B,C)|
		&\le p(p-1)|p-2|3^{p-3}\int_0^A\int_0^B\int_0^C
		(stu)^{p/3-1}\,du\,dt\,ds\\
&\le c(ABC)^{p/3}.
	\end{aligned}
\end{equation}
Here the constant $c>0$ is independent of $(p,A, B, C)$.

Denote \(A_j:=Q_{jp}^2\) and \(S_m:=\sum_{j\le m}A_j\). Note that 
\begin{equation}
	\begin{aligned}
		f_p(S_N)-\sum_{j<k}f_p(A_j+A_k)+(N-2)\sum_jf_p(A_j)
		=\sum_{k=3}^N\sum_{m=2}^{k-1}\Delta_3f_p(S_{m-1},A_m,A_k).
	\end{aligned}
	\label{eq:finite-label-telescoping}
\end{equation}
Since $p/3<1$ and $2p/3\ge5/4$,   the uniform boundedness of
$\{Q_p\}_{p\in I_{1/8}}$ in $L^\infty(\R)$ gives from (\ref{AA:6}) that for any  \(p\in I_{1/8}\) and $D(z)>1$,
\begin{equation}\label{eq:three-label-pointwise}
|\Delta_3f_p(S_{m-1},A_m,A_k)|
\le c_1\sum_{\ell<m}(A_\ell A_mA_k)^{p/3}
\le c_1\sum_{\ell<m}(Q_{\ell p}Q_{mp}Q_{kp})^{5/4},
\end{equation}
where  \( c_1 > 0 \) is independent of \((z, p)\).
Since the product
$Q_{\ell p}Q_{mp}Q_{kp}$ is unchanged for three distinct indices $\ell,m,k$,
we may assume $z_\ell<z_m<z_k$. This then shows that  $z_k-z_\ell\ge2D(z)$, and
\[
|x-z_\ell|+|x-z_m|+|x-z_k|
\ge z_k-z_\ell+|x-z_m|
\ge2D(z)+|x-z_m|.
\]
We thus calculate from \eqref{a.1} that for any  \(p\in I_{1/8}\) and $D(z)>1$,
\begin{equation}\label{eq:three-center-integral}
\begin{aligned}
	\int_\R(Q_{\ell p}Q_{mp}Q_{kp})^{5/4}\,dx
	&\le c_1e^{-\frac52\kappa_pD(z)}
	\int_\R e^{-\frac54\kappa_p|x-z_m|}\,dx
	\le c_1e^{-\frac52\kappa_pD(z)}.
\end{aligned}
\end{equation}
Integrating \eqref{eq:three-label-pointwise}, we  deduce from \eqref{eq:finite-label-telescoping} and
\eqref{eq:three-center-integral}  that  \eqref{eq:three-label-energy} holds true. This therefore completes the proof of  Lemma \ref{lem:three-label-energy}.
\end{proof}

\vspace{15pt}

\section*{Declaration of AI-assisted tools}
During the preparation of the present paper, the authors partly used OpenAI models  as auxiliary tools for exploratory discussions and preliminary proof drafting. The
mathematical validation, the final proof choices, and the final text are the responsibility
of the human authors.

\end{document}